\pdfoutput=1

\documentclass[11pt]{article}
\usepackage{vmargin}
\setmarginsrb{2.5cm}{2cm}{2.5cm}{2cm}{.4cm}{4mm}{0cm}{7.5mm}

\usepackage{times}

\usepackage[T1]{fontenc}
\usepackage{epsfig}
\usepackage{graphicx}
\usepackage{amssymb}
\usepackage{amsmath}
\usepackage[utf8]{inputenc}
\usepackage{amsfonts}
\usepackage{amsthm}
\usepackage{paralist}
\usepackage{stmaryrd}
\usepackage[toc,page]{appendix}
\usepackage{bbm}
\usepackage{bm}
\usepackage{mathrsfs}
\usepackage{longtable}
\usepackage[labelfont=bf,font=it]{caption} 

\usepackage[
  colorlinks=true,
  pdfencoding=auto, 
  psdextra,
]{hyperref}[2012/08/13]
\usepackage{bookmark}

\usepackage[dvipsnames]{xcolor}

\usepackage{centernot}
\usepackage[capitalize]{cleveref}
 
 \usepackage{microtype}
  
\usepackage{booktabs,makecell} 

\theoremstyle{plain}
\newtheorem{corollary}{Corollary}[section]
\newtheorem{theorem}[corollary]{Theorem}
\newtheorem{lemma}[corollary]{Lemma}

\newtheorem*{conjecture*}{Conjecture}
\newtheorem{proposition}[corollary]{Proposition}
\newtheorem*{theorem*}{Theorem}
\newtheorem*{lemma*}{Lemma}
\newtheorem*{definition*}{Definition}
\newtheorem*{corollary*}{Corollary}
\theoremstyle{definition}

\newtheorem{definition}[corollary]{Definition}

\newtheorem{remark}{Remark}[corollary]

\newcommand{\R}{\mathbb{R}}
\newcommand{\Z}{\mathbb{Z}} 
\newcommand{\N}{\mathbb{N}}

\newcommand{\C}{\mathcal{C}} 
\newcommand{\Cf}{\mathfrak{C}} 
\newcommand{\comp}{\mathfrak{C}} 

\newcommand{\mult}{\times}
\newcommand{\G}{\mathcal{G}} 

\newcommand{\proba}{\mathbb{P}}
\newcommand{\E}{\mathbb{E}}
\newcommand{\V}{\mathcal{V}}
\newcommand{\cube}{\square}
\newcommand{\X}{\chi}
\newcommand{\1}{\mathbf{1}}

\newcommand{\e}{\varepsilon}
\newcommand*{\GP}{\text{GP}} 
\newcommand*{\GHP}{\text{GHP}}
\newcommand{\off}{\hbox{ {\rm off} }}
\newcommand{\be}{\begin{eqnarray}} 
\newcommand{\ee}{\end{eqnarray}}

\newcommand{\eps}{\varepsilon}
\newcommand{\ind}{\mathbbm{1}}
\newcommand{\llambda}{\lambda}   
\newcommand{\kkappa}{\kappa}    
\newcommand{\ZZ}{\mathbf{Z}}  
\newcommand{\MM}{\mathbf{M}}  
\newcommand{\XX}{\mathbb{X}} 
\newcommand{\dGHP}{d_{\mathrm{GHP}}} 

\newcommand{\dP}{d_{\mathrm{P}}} 
\newcommand{\dH}{d_{\mathrm{H}}} 
\newcommand{\dGP}{d_{\mathrm{GP}}} 
\newcommand{\distGHP}{\operatorname{dist}_{\mathrm{GHP}}^4} 
\newcommand{\gp}{\text{\textsc{gp}}}
\newcommand{\bA}{\mathbf A}
\newcommand{\bB}{\mathbf B}
\newcommand{\rX}{\mathrm X}

\newcommand{\vs}{\V_{\star}}
\newcommand{\rcomp}{{\widetilde{\comp}}} 

\newcommand{\weight}[1]{\mathbf{wt(}#1\mathrm{)}}

\DeclareMathOperator*{\limit}{\longrightarrow}

\DeclareMathOperator*{\diam}{diam}

\DeclareMathOperator*{\lr}{\longleftrightarrow}
\DeclareMathOperator*{\nlr}{\centernot\longleftrightarrow}

\DeclareMathOperator{\Id}{Id}

\newcommand{\onlyon}{\hbox{ {\rm only on} }}
\DeclareMathOperator*{\slr}{\leftrightarrow}
\DeclareMathOperator*{\snlr}{\nleftrightarrow}

\newcommand{\andd}{\hbox{ {\rm and} }}

\newcommand{\pp}{\mbox{\bf p}}

\begin{document}

\title{The scaling limit of critical hypercube percolation}
\author{Arthur Blanc-Renaudie \and Nicolas Broutin \and Asaf Nachmias} 
\date{\today}
\maketitle
\begin{abstract} 
We study the connected components in critical percolation on the Hamming hypercube $\{0,1\}^m$. We show that their sizes rescaled by $2^{-2m/3}$ converge in distribution, and that, considered as metric measure spaces with the graph distance rescaled by $2^{-m/3}$ and the uniform measure, they converge in distribution with respect to the Gromov--Hausdorff--Prokhorov topology. The two corresponding limits are as in critical Erd\H{o}s--R\'enyi graphs.
\end{abstract}

\section{Introduction}\label{sec:intro}
\subsection{Main results}\label{sec:intro:main_results}
Fix $\llambda \in \R$ and consider the Erd\H{o}s--R\'enyi graph $G(n,\tfrac{1+\llambda n^{-1/3}}{n})$, that is, the graph obtained from the complete graph on $n$ vertices by independently retaining each edge with probability $\tfrac{1+\llambda n^{-1/3}}{n}$ and erasing it otherwise. Denote by $(\C_1,\C_2,\ldots)$ the connected components sorted in decreasing order according to their sizes. A celebrated theorem of Aldous \cite{Aldous97} states that $n^{-2/3} (|\C_1|,|\C_2|,\ldots) \to \ZZ_\llambda$ in distribution,
where $\ZZ_\llambda = (|\gamma_1|,|\gamma_2|,\ldots)$ are the lengths, sorted in decreasing order, of excursions above zero of the process $\{W^\llambda_t - \min_{s\in[0,t]} W^\llambda_s\}_{t \geq 0}$ where $W^\llambda_t=W_t + \llambda t - t^2/2$ and $W_t$ is standard Brownian motion. 

Our goal is to prove this for percolation on the hypercube, that is, the graph whose vertex set is $\{0,1\}^m$ and two vertices form an edge when their Hamming distance is $1$. It is not clear for which $p\in(0,1)$ one can expect such a scaling limit ($p=1/m$ does not work since it is in fact \emph{subcritical}). We show that one should choose $p$ so that the expected component size containing a vertex matches in the two models. 

Fix $\llambda\in \R$ and consider the Erd\H{o}s--R\'enyi graph $G(n,{1+\llambda n^{-1/3} \over n})$ and write $\kkappa(\lambda) := \lim_{n} \E|\C(v)|/n^{1/3}$ where $\C(v)$ is the component containing vertex $v$ (the limit exists by Corollary \ref{cor:Masses}). We now turn to percolation on the hypercube. We write $V=2^m$ and since {$f(p)=\E_p|\C(v)|$} is an increasing polynomial in $p$ with $f(0)=1$ and $f(1)=V$, we may set $p_c = p_c(\llambda,m)$ to be the unique number in $(0,1)$ with 
\be\label{def:pc} \E_{p_c} |\C(v)| = \kkappa(\lambda) V^{1/3} \ .
\ee

\begin{theorem}
\label{main_thm} Fix $\lambda \in \R$ and let $p_c=p_c(\lambda,m) \in(0,1)$ defined by \eqref{def:pc}. 
Consider the ordered connected components $(\C_1,\C_2,\ldots)$ of percolation on the hypercube with edge probability $p_c$. Then
\be\label{thm:AldousInHypercubeMain} V^{-2/3} (|\C_1|,|\C_2|,\ldots) \stackrel{\mathrm{(d)}}{\longrightarrow} \ZZ_\llambda \, ,\ee
where the convergence is with respect to $\ell^2$.
\end{theorem}
This answers positively problem (3) in \cite[Section 8]{HN17} (reiterated in \cite[Problem 13.3]{HeydHofs17}). Next, consider again $G(n,{1+\llambda n^{-1/3} \over n})$ and denote by $M_i$ the metric measure space endowed on the vertices of $\C_i$ by the shortest path metric normalized by multiplying all distances by $n^{-1/3}$, and the counting measure multiplied by $n^{-2/3}$. Addario-Berry, the second author and Goldschmidt \cite{ABG12} (see also \cite{Addario2017MST}) proved that $(M_1, M_2, \ldots ) \to \MM_\llambda$ in distribution,
where $\MM_\llambda$ is a sequence of random compact metric measure space and convergence is in distribution with respect to the Gromov--Hausdorff--Prokhorov (GHP) distance on metric measure spaces; see definitions below. Our second result is that this holds for critical percolation on the hypercube.
    
\begin{theorem}
\label{thm:main_metric} Fix $\lambda \in \R$ and let $p_c=p_c(\lambda,m) \in(0,1)$ defined by \eqref{def:pc}. 
Consider the ordered connected components $(\C_1,\C_2,\ldots)$ of percolation on the hypercube with edge probability $p_c$ and let $M_i$ be the metric measure space on the vertices of $\C_i$ with the shortest path metric multiplied by $V^{-1/3}$, and the counting measure on the nodes multiplied by $V^{-2/3}$. Then,
\be \label{thm:ABBGInHypercubeMain} (M_1, M_2, \ldots ) \stackrel{\mathrm{(d)}}{\longrightarrow} \MM_\llambda \, ,\ee
where the convergence is with respect to the metric specified by  
\begin{equation}\label{eq:def_distGHP}
\distGHP(\bA,\bB)=\bigg(\sum_{i\ge 1} \dGHP(A_i,B_i)^4\bigg)^{1/4}\,,
\end{equation}
for sequences of metric measure spaces $\bA=(A_i)_{i\ge 1}$ and $\bB=(B_i)_{i\ge 1}$, and $\dGHP$ denotes the Gromov--Hausdorff--Prokhorov (GHP) distance.
\end{theorem}

We present the definitions needed to parse \cref{thm:main_metric}, notably the GHP distance, in \cref{sec:topology}. By now it is the standard topology of metric measure spaces strong enough to yield distributional limits of essentially all large scale geometric quantities of the critical components. For example, \cref{thm:main_metric} implies the convergence of the diameter of the $i$-th largest component, for any fixed $i\geq 1$, of the typical distance or the height seen from a random vertex (rescaled by $V^{-1/3}$), to the corresponding continuous random variables. See \cite{ABG12, ANS22} and also \cite{ABGSecond,BrMa2023a,MiSe2022a} where several constructions of the limit $\MM_\llambda$ are given. It is also strong enough to imply convergence of global quantities involving more than one component, such as the maximum diameter of connected components, the length of the largest cycle, or the limiting probability that the diameter of $\C_i$ is, say, twice as large as the diameter of $\C_j$, and many more. \\

Since the results of Aldous \cite{Aldous97} and of Addario-Berry, the second author and Goldschmidt \cite{ABG12,ABGSecond}, various inhomogeneous percolation models have been shown to exhibit scaling limits as in Theorems \ref{main_thm} and \ref{thm:main_metric}, see \cite{BhHoLe2010a,BhBuWa2013a,BhSe2016a,DeLeVa2019a,BhBuWa2014a,Federico2019a,Joseph2014a,BhSeWa2014a,BhHoLe2012b,BrDuWa2021a,BrDuWa2022b,BhDhHoSe2020a,BhBrSeWa2014a,BasBhaBro2023a,Blanc-Renaudie2021a,Wang2023a,CoGo2023a,GoHaSe2022a,BhSeHo2018a}. 
This paper is the first time where such scaling limits are established in the classical setup of percolation on a deterministic transitive graph that has a non-trivial geometry. This geometry poses a significant obstacle rendering all the methods of the aforementioned papers ineffective. For example, the familiar BFS exploration process is not Markovian in our setup and we cannot use the arsenal of tools from classical stochastic processes to prove its convergence to Brownian motion with drift. In fact, { the convergence of the BFS process does not follow from our results and we do not know how to prove it}. Instead, we provide a novel method combining the theory of critical percolation in high dimensions with tools from the study of inhomogeneous percolation. We outline this idea in \cref{Sec:outline} and believe that it will have numerous further applications.

In the rest of this section we give a brief background (\cref{Sec:background}), provide a more general theorem allowing to obtain the same results for critical percolation on other underlying graphs (\cref{sec:othergraphs}) such as high degree expanders of logarithmic girth, and conclude with an outline of the proof together with the organization of the paper (\cref{Sec:outline}).

\subsection{Background} \label{Sec:background}

Percolation on the hypercube was first studied by Erd\H{o}s and Spencer in 1979 \cite{ES79}. The first result regarding the percolation phase transition (the ``appearance of a giant'') was obtained in the seminal paper of Ajtai, Koml\'os, and Szemer\'edi \cite{AKS82} where it is shown that a linear size component appears with probability tending to $1$ as $m\to \infty$ when $p={1+\delta \over m}$ for any fixed $\delta>0$, see also the work of Bollob\'{a}s, Kohayakawa and \L uczak \cite{Bol95} for a detailed behavior in the supercritical phase. When $p={1- \delta \over m}$ it is not hard to show that all components are of size that is at most logarithmic in the number of vertices.

Thus, the phase transition occurs around the point $1/m$, and it turns out that one can zoom-in and obtain a much more precise behavior of the phase transition. We refer the reader to \cite{BCHSS1, HN14, HN17} for a comprehensive explanation of this critical phenomenon and give here only a very brief outlook. When one fixes $p_c\in(0,{}1)$ as the unique solution to $\E_{p_c}|\C|=V^{1/3}$, the critical {\bf scaling window} is $p=p_c(1+\lambda V^{-1/3})$ for $\lambda \in \R$. Outside of this window we expect that the sizes of the largest connected components should be concentrated. Furthermore, below the window the ratio of the sizes of the two largest connected components should tend to one, while above the window it should tend to $0$ with high probability. Even though this picture has only been partially proved rigorously (in particular, concentration is not fully established in the subcritical phase \cite{HulNac20} and the second largest component is not understood in the supercritical case \cite{HN17}) we do not expect any interesting distributional limits in these regimes. 

Inside the scaling window, Borgs, Chayes, van der Hofstad, Slade and Spencer \cite{BCHSS1} proved that the largest connected components have size of order $\Theta(V^{2/3})$ and Kozma and the third author \cite{KN2} proved that their diameter is of order $\Theta(V^{1/3})$ (that is, the size of the largest component rescaled by $V^{-2/3}$ is a tight sequence and so is its inverse; similarly for the diameter rescaled by $V^{-1/3}$). It is also not hard to argue that the standard deviations of the diameter and the size are also of respective orders $\Theta(V^{2/3})$ for the size and $\Theta(V^{1/3})$ for the diameter. Thus one expects non-trivial scaling limits of the connected components sizes and metric space structure. The contribution of the present paper, namely Theorems~\ref{main_thm} and~\ref{thm:main_metric}, is to establish that these scaling limits are the same as the ones of the classical Erd\H{o}s--R\'enyi random graphs obtained in \cite{Aldous97} and \cite{ABG12}.

\subsection{Other underlying graphs}\label{sec:othergraphs} We now describe the basic assumptions we need for the proofs in this paper to work. This yields a more general class of graphs (that includes the hypercube) under which the conclusions of \cref{main_thm} { and \cref{thm:main_metric}} hold. In particular, the assumptions below hold for high degree expander graphs with girth that is logarithmic in the volume, { and for} products of complete graphs, see \cite[Section 1.5]{HN17} as well as \cite[Theorem 1.4]{HN17} and its proof. We conclude that the results of Theorems \ref{main_thm} and \ref{thm:main_metric} hold for these graphs.

This class of graphs was first defined in \cite{HN17} and is geometric (for example, not too many short cycles, good expansion etc.) but is best described by certain random walk conditions that are usually easy to verify. The non-backtracking random walk (NBRW) is just a simple random walk not allowed to traverse back an edge it has just traversed. That is, in the first step it chooses uniformly between the $m$ neighbors of the initial vertex and at any later steps it choose uniformly among the $m-1$ neighbors which are not the neighbor visited in the previous step. We discuss this further in \cref{sec:nbrw}. 

We write $\pp^t(u,z)$ for the probability that the non-backtracking random walk starting from $u$ is at $z$ after $t\geq 0$ steps. For any $\xi>0$ we define the $\xi$-\emph{uniform mixing time} $T_{{\rm mix}}(\xi)$ by

$$ T_{{\rm mix}}(\xi) = \min \Big \{ t : \forall x,y, \quad\frac{\pp^t(x,y) + \pp^{t+1}(x,y)}{
2} \leq (1+\xi) V^{-1} \Big \}  \, . $$

\begin{theorem}\label{main_thm_general} Let $\{H_n\}_{n\geq 1}$ be a sequence of transitive graphs with $V\to \infty$ vertices, degree $m\to \infty$ and let $\lambda \in \R$ be fixed and set $p_c=p_c(m,\lambda)$ as in \eqref{def:pc}. Assume that there exists a positive sequence $\alpha_m$ with $\alpha_m\to 0$ and $\alpha_m \geq m^{-1}$ such that if we set $m_0 = T_{{\rm mix}}(\alpha_m)$, then $m_0=O(V^{1/15}\alpha_m)$ and
\begin{enumerate}

\item 
\be\label{pc_estimate} [p_c(m-1)]^{m_0} = 1 + O(\alpha_m) \, ,\ee

\item for any two vertices $x,y$
    \be\label{eq:TriangleRW}
    \sum_{u,v} \sum_{\substack{t_1,t_2,t_3=0\\t_1+t_2+t_3\geq 3}}^{m_0} \pp^{t_1}(x,u) \pp^{t_2}(u,v) \pp^{t_3}(v,y) = O(\alpha_m/\log{V}).
    \ee

\end{enumerate}
Then \cref{main_thm}  and \cref{thm:main_metric} hold for the graph sequence $\{H_n\}$.
\end{theorem}

The conditions of \cref{main_thm_general} are verified for the hypercube $H=\{0,1\}^m$; we collect here the relevant references from which Theorems \ref{main_thm} and \ref{thm:main_metric} follows.

\begin{proof}[Proof of Theorems \ref{main_thm} and \ref{thm:main_metric} assuming \cref{main_thm_general}] We check that the assumptions of \cref{main_thm_general} hold for the hypercube: In \cite{HS05} the lace expansion is employed to show that for any fixed $\lambda>0$
\be\label{eq:pc} p_c(\lambda) = {1 \over m-1} + O(m^{-3}) \, ,\ee
see also \cite[Theorem 1.6]{HN14} for an elementary proof. The fact that there is no $m^{-2}$ term is crucial, since by \cite[Lemma 7.1]{HN17} we have that $T_{{\rm mix}}(m^{-1}\log m)=O(m \log m)$ so we take $\alpha_m=m^{-1} \log m$ and have that $m_0=O(m\log m)$. This verifies \eqref{pc_estimate} when $H_n=\{0,1\}^n$. Lastly, \cite[Lemma 7.1]{HN17} also verifies \eqref{eq:TriangleRW} for the hypercube.\end{proof}

\begin{remark} \cref{main_thm_general} does not include the case of the high-dimensional torus, i.e., $G=(\Z_n)^d$ where $d$ is fixed and large (or any fixed $d>6$ with a spread-out torus). In this case one expects that an analogue of Theorems \ref{main_thm} and \ref{thm:main_metric} holds at $p_c(\Z^d)$, that is, at the critical percolation probability of the infinite lattice. Unfortunately, the approach and techniques used in this paper fail for the high-dimensional torus in various locations in the proof. For instance, it is crucial for us that the degree tends to $\infty$ and that $p_c \sim 1/m$, also, that the triangle condition \eqref{eq:TriangleCondition} or \eqref{eq:TriangleRW} are small. These facts are used throughout the paper in numerous key estimates that are no longer true in the torus case. We plan to address the problem of critical percolation on the torus in a future publication. 
\end{remark}

\subsection{Outline of the proof and organization} \label{Sec:outline}

We are led by the intuition that the critical clusters are formed by subcritical clusters coalescing so that the rate of coalescence of two subcritical clusters is proportional to the product of their cardinality (that is, according to Aldous' ``multiplicative coalescent'' introduced in \cite{Aldous97}). Hence we begin by studying large clusters in the slightly subcritical phase in percolation on $H$. 
Recall the definition of $(\alpha_m)_{m\in\N}$ from Section \ref{sec:othergraphs}. (In the hypercube we take $\alpha_m=m^{-1}\log(m)$.) We set
\be\label{def:ps} p_s := p_c(1-V^{-1/3}\alpha_m^{-1/3}) \, ,\ee
and consider the connected components of $H_{p_s}$. For technical reasons we would like to study only clusters that are not too small. To that aim we set 
$$ M_s:= V^{2/3} \alpha_m^4 \, ,$$
and let $\Cf_{p_s,M_s}$ denote the set of components of $H_{p_s}$ of size at least $M_s$. We remark that at $p_s$ the largest clusters are of size $\Theta(V^{2/3} \alpha_m^{2/3} \log \alpha_m)$ \cite{HulNac20} so $\Cf_{p_s,M_s}$ includes them. It will become evident later that as the clusters ``coalesce'' the ones of size smaller than $M_s$ do not contribute significantly to critical clusters, so it is safe to ignore them. 

Next we construct two auxiliary random graphs which both have $\Cf_{p_s,M_s}$ as their vertex set. The first is what we call the {\bf multiplicative component graph} $G_\mult$. For a component $A\in \Cf_{p_s,M_s}$ we set weight
\be\label{def:weights} w_A = |A|V^{-2/3} \, ,\ee
and let the edge $(A,B)$ in $E(G_\mult)$ be present with probability
\be\label {def:qab} q_{A,B}:=1-e^{-q w_A w_B} \, , \ee
independently of all other edges where $q>0$ is set to be
\begin{equation} q = q_\lambda=  V^{1/3}/\chi(p_s) + \lambda +o_{m\to \infty}(1)\, , \label{def_q} \end{equation}
where $\chi(p)=\E_p|\C(v)|$ is the expected size of the cluster containing $v$ in $H_p$ (by transitivity it does not depend on $v$).

The random graph $G_\mult$ is an instance of Aldous' \emph{multiplicative random graph} which is a well studied object. In \cref{Sec:Mult}, we apply Proposition 4 of Aldous \cite{Aldous97}, and Theorem 3.2 of Bhamidi, Broutin, Sen and Wang \cite{BhBrSeWa2014a}, as black boxes, to obtain that the scaling limits of $\G_\mult$, properly scaled, are $\ZZ_\llambda$ and $\MM_\llambda$, as defined above \cref{main_thm}. A delicate calculation that we perform (in \cref{Sec:Mult}) for that goal is a sharp estimate on the second moment of the size of a subcritical cluster (see \cref{lem:sigma3FirstMom}).

It is not clear that the components of $G_\mult$ should be close to critical percolation clusters of $H$. Note that conditioned on $H_{p_s}$, the probability that there is a $p_c$-open edge between two clusters $A$ and $B$ of $H_{p_s}$ is precisely
\be \label{eq:DONTUSEpab} 1-\left (\frac{1-p_c}{1-p_s} \right)^{\Delta_{A,B}} \, ,\ee
where $\Delta_{A,B}$ for the number of edges having one endpoint in $A$ and the other in $B$. In the complete graph, we always have $\Delta_{A,B}=|A||B|$ so multiplicativity is inherently present in that setup. In the hypercube it is reasonable to believe that $\Delta_{A,B}$ is close to $m|A||B|/V$; the latter is just the expectation of $\Delta_{A,B}$ if $A$ and $B$ were two independent uniformly drawn sets of size $|A|$ and $|B|$. We are unable to prove this uniformly over all $A$ and $B$ (as one has in the complete graph) but only in the $\ell^2$ sense, see \cref{prop:DeltaBis}. 

Thus, it is natural to take the second random graph, again on vertex set $\Cf_{p_s,M_s}$, with independent edge probabilities defined by \eqref{eq:DONTUSEpab}. However, at this point we know very little about the value of $p_c$ and cannot argue that the two random graphs will be close. Instead we argue indirectly and take $p_c'=p_c'(\lambda)\in(0,1)$ to be the unique number satisfying
\begin{equation} 
\log\bigg(\frac{1-p'_c(\lambda)}{1-p_s}\bigg)=-\frac{q_\lambda}{mV^{1/3}}\,, \label{def:p'clambdaINTRO} 
\end{equation}
and set 
\be p_{A,B}:=1-\left (\frac{1-p'_c(\lambda)}{1-p_s} \right)^{\Delta_{A,B}}=1-e^{-q_\lambda\Delta_{A,B}/(mV^{1/3})} \label{def_p_INTRO} \, . \ee
We now let $G_\comp$ be the random graph on vertex set $\Cf_{p_s,M_s}$ so that each edge $(A,B)$ is independently retained with probability $p_{A,B}$ and deleted otherwise. We call $G_\comp$ the {\bf sprinkled component graph}. 

In \cref{sec:CompGraph} we then prove, via a coupling between $G_\mult$ and $G_\comp$ that the components of $G_\comp$ converge to $\ZZ_\llambda$ and $\MM_\llambda$, as defined above \cref{main_thm}. Note that the component sizes in $G_\comp$ have exactly the same distribution as component sizes in $H_{p_c'}$ due to the way we chose $p_{A,B}$ in \eqref{def_p_INTRO} (it does not follow, however, that the the geometry of the two graphs is close; that is the purpose of \cref{sec:FullConvergence}). Thus the component sizes of $H_{p_c'}$ converge to $\ZZ_\llambda$. This suggests that $p_c'$ and $p_c$ are close and in \cref{sec:position} we show that indeed $|p_c-p_c'|$ is of order $o(m^{-1} V^{-1/3})$. This means that $p_c$ and $p_c'$ correspond to the same position in the scaling window, alternatively stated, one can choose the $o_m(1)$ in the definition of $q(\lambda)$ above so that the values of $p_c$ (which depend only on $m$ and $\lambda$) and $p_c'$ (which depend also on the choice of $q$) are in fact equal. This already implies that component sizes of $H_{p_c}$ converge to $\ZZ_\llambda$, i.e., this proves \cref{main_thm}, see \cref{sec:mainthmPart1}. Lastly, in \cref{sec:FullConvergence} we perform a delicate coupling between $G_\comp$ and $H_{p_c}$ yielding \cref{thm:main_metric}. As the argument in \cref{sec:FullConvergence} is rather { lengthy} we omit its outline and refer the reader to that section. 
\newcommand{\SplitTableNotations}{\end{tabular}
\end{table}

\begin{table}[htbp]
\centering
\begin{tabular}{ |c | c  | c | }
\hline
\textbf{Symbol} & \textbf{Explanation} &  \\
\hline }

\newpage
\subsection{Table of notations}
\begin{center}

\begin{longtable}{|l|l|l|}

\hline \multicolumn{1}{|c|}{\textbf{Symbol}} & \multicolumn{1}{c|}{\textbf{Explanation}} & \multicolumn{1}{c|}{\textbf{Location}} \\ \hline 
\endfirsthead

\multicolumn{3}{c}%
{{\bfseries  Table of notations -- continued from previous page}} \\
\hline \multicolumn{1}{|c|}{\textbf{Symbol}} & \multicolumn{1}{c|}{\textbf{Explanation}} & \multicolumn{1}{c|}{\textbf{Location}} \\ \hline 
\endhead

\hline \multicolumn{3}{|r|}{{Continued on next page}} \\ \hline
\endfoot

\hline
\endlastfoot

\hline
$E \circ F$ & disjoint occurrence of event $E$ and $F$ & \cref{rem:BKvsBKR} \\ 
$x\lr y$ & event that there exists an open path between $x$ and $y$ & \\ 
$x\nlr y$ & complement of $x \lr y$ & \\
$E \off A$ & event that $E$ holds after closing every edges adjacent to $A$ & \\
$x\slr y \onlyon A$ & event that $x\lr y$ holds but $(x\lr y\off A)$ does not & \\
\hline

$G(n,p)$ & Erd\H{o}s--R\'enyi random graph & \cref{sec:intro:main_results}  \\
$H$ & hypercube $\{0,1\}^m$ or a graph with conditions of \cref{main_thm_general} & \cref{sec:othergraphs} \\
$H_p$ & bond percolation on $H$ with probability $p\in[0,1]$. & \cref{sec:othergraphs}\\
$G_\times$ &  multiplicative component graph & \cref{Sec:Mult} \\
$G_\comp$ & sprinkled component graph &  \cref{sec:CompGraph} \\
$G_{\rcomp}$ & full component graph &  \cref{TransformTight} \\
\hline

$d_{\GHP}$ & Gromov--Hausdorff--Prokhorov distance & \cref{sec:topology} \\
$d_{\GP}$ & Gromov--Prokhorov distance & \cref{sec:topology} \\
$d_\cube$  &graph distance on $H_{p_c(\lambda)}$ & \cref{def:distanceCubeComp}\\
$d_\comp$  &graph distance on $G_\comp$ & \cref{def:distanceCubeComp} \\
$d_\times$  &graph distance on $G_\times$ & \cref{Sec:Mult} \\
$d_\rcomp$  &graph distance on $G_\rcomp$ & \cref{TransformTight} \\
\hline
 $(\C_1,\C_2,\ldots)$ & connected components of $H_p$ ordered by size & \cref{sec:intro:main_results} \\
$(\C_1^\comp, \C_2^\comp,\ldots)$ & connected components of $G_\comp$ & \cref{sec:CompGraph} \\
$(\C_1^\times, \C_2^\times,\ldots)$ & connected components of $G_\times$ & \cref{Sec:Mult} \\
\hline
$\Delta_{A,B}$ & number of edges between the sets $A$ and $B$ & \cref{Sec:outline} \\
$p_{A,B}$ & probability that there is an edge between clusters $A$ and $B$ in $G_\comp$ & \eqref{def_p_INTRO} \\
$q_{A,B}$ & probability that there is an edge between clusters $A$ and $B$ in $G_\times$ & \eqref{def:qab} \\
$\neq_{A,B}$ &  event that there exists a self-avoiding path between $A$ and $B$ that is & \cref{sec:counting_bad_pairs} \\
& present in one of $G_\times$ or $G_\comp$, but not in the other  &\\
\hline 
$|A|$ & number of vertices in the set $A$ & \cref{sec:intro:main_results} \\
$w_A$ & $w_A:=V^{-2/3} |A|$ &  \cref{Sec:outline} \\
$\weight{\C_i^\times}$ & weight of $\C_i^\times$, that is, $\sum_{A \in \C_i^\times} w_A$ & \cref{Sec:Mult} \\
\hline
$\ZZ_\llambda$ & $\ZZ_\llambda = (|\gamma_1|,|\gamma_2|,\ldots)$ excursion lengths & \cref{sec:intro:main_results} \\
$\MM_\llambda $ & sequence of limiting mm-spaces & \cref{sec:intro:main_results}\\
$(M_1,M_2,\dots,)$ & renormalized connected components of $H_{p_c(\lambda)}$ & \cref{sec:intro:main_results} \\
\hline

$V$ & number of vertices in $H$ & \cref{sec:othergraphs}\\
$\kappa(\lambda)$ & $\lim_{n \to\infty} n^{-1/3} \E_p|\C(V)|$ for $G(n,p)$ with $p=1/n+\lambda n^{-4/3}$ & \cref{sec:intro:main_results} \\
$p_c(\lambda)$ & the unique $p$ satisfying \eqref{def:p'clambdaINTRO} & \cref{sec:intro:main_results}\\

\hline 
$\alpha_m$ & $\alpha_m\to 0$ and $\alpha_m\geq m^{-1}$ & \cref{main_thm_general} \\
$m_0$ &  $m_0=T_{\text{mix}}(\alpha_m)$ and satisfies the conditions of \cref{main_thm_general} & \cref{sec:othergraphs} \\
\hline
$p_s$ & $p_s = p_c(1-V^{-1/3}\alpha_m^{-1/3})$ & \cref{Sec:outline} \\ 
$\Cf_{p_s,M_s}$ & components of $H_{p_s}$ of size at least $M_s$; the vertices of $G_\mult$ and $G_\comp$ & \cref{Sec:outline} \\
$q_\lambda$ & $q_\lambda=  V^{1/3}/\chi(p_s) + \lambda +o_{m\to \infty}(1)$ &  \eqref{def_q} \\
$p'_c(\lambda)$ & $\log((1-p'_c(\lambda))/(1-p_s))=-q_\lambda/(mV^{1/3})$ & \eqref{def:p'clambdaINTRO} \\
\hline
$\pp^{t}(u,z)$ & probability that NBRW from $u$ is at $z$ in $t$ steps & \cref{sec:nbrw} \\
$\pp^{t}(u,z; t_1)$ & same as $\pp^{t}(u,z)$ but backtrack at step $t_1<t$ is allowed & \cref{sec:nbrw} \\
\hline
\end{longtable}
\end{center}
\newcommand{\bbL}{\mathbb L}
\newcommand{\bX}{\mathbf X}

\section{Preliminaries}

\subsection{Topological notions}\label{sec:topology}

We provide here the definition of GHP convergence required to parse \eqref{thm:ABBGInHypercubeMain} as well as various abstract tools and definitions needed for the proof. Our space $\XX_c$ is the space of compact metric measure spaces $(X,d,\mu)$ where $\mu$ is a finite Borel measure on $(X,d)$ and such spaces are identified in $\XX_c$ if there is a bijective isometry between them which also preserves the measure. We call the elements of $\XX_c$ mm-spaces. To define the GHP distance, recall first that the {\bf Hausdorff distance} $\dH$ between two sets $A, A' \subset X$ is defined by
\[
	\dH(A, A') = \max \big\{ \sup_{a \in A} d(a, A'), \sup_{a' \in A'} d(a', A) \big\}.
\]
Next, for any $A\subset X$ and $\e>0$ we define $A^\e = \{x \in X : d(x,A) \leq \e\}$. If $\mu$ and $\nu$ are two finite Borel measures on $X$, the {\bf Prokhorov distance} $\dP$ between $\mu$ and $\nu$ is given by
\[
	\dP(\mu, \nu) = \inf \{ \epsilon > 0: \mu(A) \leq \nu(A^{\epsilon}) + \epsilon \text{ and } \nu(A) \leq \mu(A^{\epsilon}) + \epsilon \text{ for any closed set } A \subset X \}.
\]
Lastly, given two elements $(X,d,\mu)$ and $(X',d',\mu')$ of $\XX_c$ the {\bf Gromov--Hausdorff--Prokhorov} distance between them is defined to be
\[
		\dGHP((X,d,\mu),(X',d',\mu')) = \inf \left\{\dH (\phi(X), \phi'(X')) \vee \dP(\phi_* \mu, \phi_*' \mu') \right\},
\]
where the infimum is taken over all isometric embeddings $\phi: X \rightarrow F$, $\phi': X' \rightarrow F$ into some common metric space $F$. It is well known that $(\XX_c, \dGHP)$ is a Polish metric space \cite[Sections 1.3 and 6]{MiermontTessellations} and \cite[Theorem 2.5]{AbDeHo2012a} so the notion of convergence in distribution as in \eqref{thm:ABBGInHypercubeMain} is standard. 

\medskip
\emph{Gromov--Prokhorov and Gromov--Hausdorff--Prokhorov convergence.} The structure of our proof relies on a two-step argument similar to the decomposition of the uniform convergence for random functions into the convergence of the finite-dimensional distributions first, and then a strengthening via a proof of tightness. When the metric spaces are trees, this goes back to the seminal papers of Aldous \cite[][Section 3]{AldousCRTIII}, even though it is not phrased in these terms; this is extended to the case of metric spaces in \cite{GrevenPfaffWinterGwConvergence,AthreyaLohrWinterGap}.

Here, the weaker topology we will use relies on the Gromov--Prokhorov (GP) distance defined as follows. For two elements of $\XX_c$ denoted by $(X,d,\mu)$ and $(X',d',\mu')$ we define
\[ \dGP((X,d,\mu),(X',d',\mu')):=\inf_{S,\phi,\phi'} \dP(\phi_\star \mu, \phi'_\star\mu')\,,\]
where the infimum is taken over all metric spaces $S$ and isometric embeddings $\phi :X\to S$, $\phi' :X'\to S$ and $\phi_\star \mu$ is the push-forward measure of $\mu$ under $\phi$. In fact, $\dGP$ is only a pseudo-metric on $\XX_c$ so we actually consider it on the quotient space $\XX_c^{\gp}$ obtained from $\XX_c$ by identifying elements at GP distance $0$.

We will use another convenient characterization of the GP topology which relies on convergence of the law of distance matrices between random points: For every  metric measure space  $\rX=(X,d,\mu)$ let $(x_i)_{i\ge 1}$ be a sequence of i.i.d.\ random variables with common probability distribution $\mu/\mu(X)$ and let $M^\rX:=(d(x_i,x_j))_{i,j\ge 1}$. The following result is a straightforward extension of Theorem~5 of \cite{GrevenPfaffWinterGwConvergence} to our setting with finite Borel measures (rather than probability measures):

\begin{lemma} \label{equivGP} 
Let $(X_n,d_n,\mu_n)_{n\ge 1}$ and $\rX=(X,d,\mu)$ be mm-spaces. Then $\dGP(\rX_n,\rX)\to 0$ as $n\to \infty$ if and only if $M^{\rX_n}\to M^\rX$ in distribution and $\mu_n(X_n)\to \mu(X)$.
\end{lemma}

It is immediate from the definition that the convergence for $\dGHP$ implies the convergence for $\dGP$. Conversely, the following tightness criterion allows us to strengthen a GP convergence to a GHP convergence. It has been established by Athreya, Lohr and Winter in \cite{AthreyaLohrWinterGap}, see especially Theorem~6.1 there. For convenience, we adapt the formulation from Theorem 6.5 of \cite{ANS22}. In fact, \cite{ANS22} only deals with one single mm-space, while the following result concerns joint convergence of several mm-spaces. The proof is the same as Theorem 6.5 from \cite{ANS22}, and is thus omitted.

\begin{lemma}
\label{GP=>GHP} 
For $i\in \N$, let $((X^i_n,d^i_n,\mu^i_n))_{n\geq 1}$, and $(X^i,d^i,\mu^i)$ be random mm-spaces and assume that
\begin{compactitem}
\item[(i)] For any finite $S\subset \N$ we have the joint convergence in distribution 
$$ \Big ( (X^i_n,d^i_n,\mu^i_n)  \Big )_{i \in S}  \stackrel{\mathrm{(d)}}{\longrightarrow}  \Big ( (X^i,d^i,\mu^i) \Big )_{i \in S}$$ 
with respect to the GP topology. 
\item[(ii)] For every $i\in \N$ and any $\delta>0$, 
\[ \lim_{\e\to 0} \limsup_{n\to \infty} \proba\left (\inf_{x\in X^i_n} \mu^i_n(B(x,\delta)) \leq \e \right )=0.\]
\item[(iii)] For every $i\in \N$, almost surely $\mu^i$ has full support on $X^i$.
\end{compactitem} 
Then for any finite $S\subset \N$ we have the joint convergence in distribution 
$$ \Big ( (X^i_n,d^i_n,\mu^i_n)  \Big )_{i \in S}  \stackrel{\mathrm{(d)}}{\longrightarrow}  \Big ( (X^i,d^i,\mu^i) \Big )_{i \in S}$$ 
with respect to the GHP topology. 
\end{lemma}

\cref{thm:main_metric} concerns the convergence of a sequence of elements in $\XX_c$. In its proof, the core of the argument will consist in verifying the convergence of the metric spaces $M_i$, $i\ge 1$, for the Gromov--Prokhorov in \emph{(i)},
while the tightness in \emph{(ii)} to extend it to GHP 
will be quickly checked in \cref{Sec:GHPtight} and \cref{sec:proof_main_thm_metric}. 
It is also well known that the components of the limit $\MM_\lambda$ in \cref{thm:main_metric} have a measure with full support. Indeed,  by Theorem 3 (iii) of Aldous \cite{AldousCRTI}, the CRT has a measure with full support, and the components of $\MM_\lambda$ can be constructed by gluing points of a biased CRT (see \cite{ABGSecond,ABG12}), which preserves this property.

Lastly, for sequences of metric spaces, we let $\bbL_4$ be the set of sequences $\bX=\big( (X_i,d_i, \mu_i)\big)_{i \geq 1}$ of elements of $\XX_c$ for which 
\[\sum_{i\ge 1} \mu_i(X_i)^4<\infty \qquad \text{and} \qquad \sum_{i\ge 1} \diam(X_i)^4<\infty\,,\]
where $\diam(X)=\sup\{d(x,y):x,y\in X\}$ is the diameter of $(X,d)$. Recalling the distance $\distGHP$ on sequences of elements of $\XX_c$ defined in \eqref{eq:def_distGHP}, the space $(\bbL_4, \distGHP)$ is Polish, and the strenghening from a convergence in the product GHP topology to only boils down to verifying tightness of the real-valued sequences of masses and diameters in $\ell^4=\{(x_1,x_2,\dots): \sum_{i\ge 1} x_i^4<\infty\}$ \cite{Addario2017MST}. For the case of Theorem~\ref{thm:main_metric}, the tightness of $(M_i)_{i\ge 1}$ in $\bbL^4$ is established in \cref{sec:proof_main_thm_metric}.

\subsection{Percolation} \label{sec:prelim_percolation}
In this section we recall some of the basic definitions and results regarding hypercube percolation, or on any transitive finite graph sequence satisfying the conditions of \cref{main_thm_general}.  We provide only the bare minimum that is required for the proofs in this paper and refer the interested reader to \cite{BCHSS1, HN14, HN17,HeydHofs17} for a detailed treatment regarding critical percolation in high dimensions on finite graphs.

Recall the setting of \cref{main_thm_general} or, alternatively, assume that $H_n$ is the hypercube $\{0,1\}^n$ and that $\alpha_m=m^{-1}\log m$ and $m_0 = \Theta(m\log m)$; we remind the reader that the conditions of \cref{main_thm_general} were verified for the hypercube in \cref{sec:othergraphs}. The first consequence of the assumptions of \cref{main_thm_general} is that the \emph{triangle condition} holds for $H_n$. That is, \cite[Theorem 1.3(a)]{HN17} asserts that 
\be\label{eq:TriangleCondition} \sup_{x,y} \sum_{z,w} \proba_{p}(x \lr z) \proba_{p}(z \lr w) \proba_{p}(w \lr y) \leq {\bf 1}_{\{x=y\}} + {\alpha_m} + {C (\E_{p}|\C|)^3 \over V} \, ,\ee
for any $p\leq p_c(\lambda)$ and for some positive constant $C = C(\lambda) \in (0,\infty)$. The triangle condition was introduced by Aizenman and Newman \cite{AizenmanNewman84} and was first used by Barsky and Aizenman \cite{BarskyAizenman91} to study critical percolation. Since then it was shown that many important estimates and critical exponents can be derived from it, making it a fundamental tool in the study of critical percolation. 

Before proceeding to describe the various implications of \eqref{eq:TriangleCondition} that we use in this paper, let us remark that most of them only apply when the triangle diagram is small, that is, when $p_0=p_c(\lambda_0)$ for some small enough but fixed $\lambda_0 \in \R$, that is, at the bottom of the critical window while we would like to study percolation at $p_c(\lambda)$ for any fixed $\lambda \in \R$, that is, in the entire critical window. We explain how to do this at the end of this section, see \cref{rem:largerLambda}.

In \cite[Theorem 1.3 (i)]{BCHSS1} it is shown that \eqref{eq:TriangleCondition} implies that for some $C\in(0,\infty)$ we have
\be\label{eq:ClusterTail} \proba_{p_0} ( |\C| \geq k) \leq {C \over \sqrt{k}} \, ,\ee
for all $k \geq 1$, where $|\C|$ is the component containing the origin (or any other vertex). We remark that in \cite[Theorem 1.3 (i)]{BCHSS1} an upper bound on $k$ is also assumed, but it is only used for the lower bound on $\proba_{p_0} ( |\C| \geq k)$; the upper bound stated in \eqref{eq:ClusterTail} is valid for all $k$.

Another useful consequence of the triangle condition \eqref{eq:TriangleCondition} is that critical exponents governing the intrinsic metric attain their mean-field value. Given a vertex $x\in H$, an integer $r \geq 1$ and $p\in(0,1)$ we write $B(x,r)$ for the set of vertices in $y \in H$ such that there exists an open path in $H_p$ of length at most $r$ connecting $x$ and $y$. We further write $\partial B(x,r)$ for the set of vertices $y \in H$ such that the shortest path in $H_p$ between $x$ and $y$ is of length precisely $r$. In \cite[Theorem 1.2]{KN2} it is proved that the triangle condition \eqref{eq:TriangleCondition} implies that for any $r>0$
\be\label{eq:IntVolume} \E_{p_0} |B(x,r)| \leq Cr \, ,\ee
and 
\be\label{eq:IntBoundary} \proba_{p_0} ( \partial B(x,r)\neq \emptyset ) \leq C/r \, .\ee
where $C\in(0,\infty)$ is a constant.

An additional ingredient that we will use is an exponential bound on the probability that there exists a "long and thin'' ball. In \cite[Lemma 6.3]{NP3} it is shown that if \eqref{eq:IntVolume} and \eqref{eq:IntBoundary} hold for some $p\in [0,1]$, then there exist $C,c>0$ such that for any positive $R,M$ satisfying
$$ R \geq c MV^{-1/3} \quad \text{and} \quad R \geq c\sqrt{M} \, ,$$
we have 
\be\label{eq:ExistsLongThin} \proba_p (\exists v \,\, |B(v,R)| \leq M \andd \partial B(v,R) \neq \emptyset ) \leq C \Big ( {1 \over R} \vee V^{-1/3} \Big ) e^{-cR^2/M} {V \over M} \, .\ee
We remark that \cite[Lemma 6.3]{NP3} is slightly weaker than \eqref{eq:ExistsLongThin}, the event on the right-hand side has $|\C(v)|\leq M$ (instead of $|B(v,R)|\leq M$) however the proof of \cite[Lemma 6.3]{NP3} (including the proof of \cite[Lemma 6.2]{NP3}) works verbatim if one replaces $\C(v)$ with $B(v,R)$ and yields \eqref{eq:ExistsLongThin}. We also use the fact that \eqref{eq:IntVolume} and \eqref{eq:IntBoundary} imply that the maximal diameter of any component in $H_{p_0}$ (and hence to any $p_c(\lambda)$ by \cref{rem:largerLambda}) divided by $V^{1/3}$ is tight, that is, for any $\e>0$ there exists $A\geq 1$ such that
\be \label{eq:MaxDiamCrit} \proba_{p_0} ( \exists v \, \, \diam(\C(v)) \geq AV^{1/3} ) \leq \e \, ,
\ee
which is shown in \cite[Theorem 6.1]{NP3}.

Next, the triangle condition \eqref{eq:TriangleCondition} provides good control in the subcritical phase. Let $\eps=\eps(n)$ be a non-negative sequence such that $\eps \gg V^{-1/3}$ but $\eps=o(1)$ and set $p=p_0(1-\eps)$. Such $p$'s are outside of the scaling window and are ``slightly'' subcritical. In \cite[Theorem 1.2 (i)]{BCHSS1} it is shown that \eqref{eq:TriangleCondition} implies
\be\label{eq:subcritSucespt} \chi(p) = \E_p |\C(v)| = (1+o(1))\eps^{-1} \, ,\ee
and also \cite[Theorem 1.2 (ii)]{BCHSS1} shows that \eqref{eq:TriangleCondition} implies that with high probability
\be\label{eq:subcritMaxComp} |\C_1| \leq 2\chi^2 \log (V/\chi^3) \, .\ee
Taking the right-hand side of the last equation as $M$ and taking $R=C\chi(p_s)\log (V/\chi(p_s)^3)$ and plugging these inside \eqref{eq:ExistsLongThin} immediately gives a bound on the largest diameter (that is, largest distance between two vertices) in the subcritical phase, namely,
\be\label{eq:subcritMaxDiam} \max_{i\geq 1}\diam(\C_i) \leq C \chi \log (V/\chi(p_s)^3) \, ,\ee
holds with high probability. We remark that in \cite[Theorem 1.6]{HulNac20} it is proved that $C$ can be taken to be $(1+o(1))$ but we will not use this fact in this paper.

Our last ingredient is special for the hypercube or the graphs addressed in \cref{main_thm_general}; it does \emph{not} rely on the triangle condition rather on the assumption \eqref{pc_estimate}. Recall that $m_0 = T_{{\rm mix}}(\alpha_m)$ (in the hypercube one takes $m_0=m\log m$ and $\alpha_m=m^{-1}\log m$). In \cite[Lemma 3.13]{HN17} (with $\e=0$ and taking $r\to\infty$) it is shows that \eqref{pc_estimate} implies that for any $\lambda\in \R$, when setting $p_c=p_c(\lambda)$ as in \eqref{def:pc}, we have that for any $p\leq p_c$ and any two vertices $x,y$
\be\label{eq:UniformConnection} \proba_{p} (x \stackrel{\geq m_0}{\lr} y) \leq {(1+O(\alpha_m)) \chi(p) \over V} \, ,\ee
where $x \stackrel{\geq m_0}{\lr} y$ is the event that there exists an open path of length at least $m_0$ connecting $x$ to $y$. The sum over $y$ of the above probability is $(1-o(1))\chi(p)$ when $m_0$ is of smaller order than $\chi(p)$ (using \eqref{eq:IntVolume}), so \eqref{eq:UniformConnection} is sharp for most $y$'s. We remark that since the derivation of \eqref{eq:UniformConnection} does not rely on the triangle condition, it is valid for all $\lambda$ not just $\lambda\leq \lambda_0$ (and in fact for the supercritical phase as well, see \cite{HN17}).

\begin{remark}\label{rem:largerLambda}  Fix some $\lambda \in \R$ and consider $p_c=p_c(\lambda)$ as defined in \eqref{def:pc}. We explain now how can one use \eqref{eq:IntVolume} and \eqref{eq:IntBoundary} at $p_c$ rather than $p_0$. We first argue that there exists some constant $A=A(\lambda)\in (0,\infty)$ such that $p_c \leq p_0 + A m^{-1} V^{-1/3}$. Indeed, this is a direct consequence by the definition of $p_c(\lambda)$ and of \cite[Theorem 1.3]{HN17} part (b) stating that at $p=p_0 + A m^{-1} V^{-1/3}$ as long as $A\to\infty$ we have that $\E_p|\C|$ is of order $A^2 V^{1/3}$. Note that the assumptions of \cref{main_thm_general} allow us to appeal to \cite[Theorem 1.3]{HN17}. We may now apply \cite[Theorem 1.2]{KN2} and obtain that \eqref{eq:IntVolume} and \eqref{eq:IntBoundary} hold as written but where the constant $C$ now may depend on $\lambda$. 
\end{remark}

\begin{remark}\label{rem:BKvsBKR} Throughout the paper we use the classical van den Berg--Kesten inequality (BK inequality henceforth) \cite{BK1985} valid for monotone events, as well as Reimer's stronger version \cite{Reimer,Reimer2} valid for all events (BKR inequality henceforth, in order to indicate where Reimer's Theorem is used). It states that $\proba(A\circ B) \leq \proba(A)\proba(B)$ where  $A\circ B$ denotes the disjoint occurence of events $A$ and $B$, see \cite{Reimer,Reimer2}.
\end{remark}

\subsection{Non-backtracking walk on the hypercube and percolation} \label{sec:nbrw} The non-backtracking random walk (NBRW) is just a simple random walk not allowed to traverse back an edge it has just traversed. That is, in the first step it chooses uniformly between the $m$ neighbors of the initial vertex and at any later steps it choose uniformly among the $m-1$ neighbors which are not the neighbor visited in the previous step. 

We write $\pp^t(u,z)$ for the probability that the non-backtracking random walk starting from $u$ is at $z$ after $t\geq 0$ steps. For each integer $t\geq 0$ we bound the number of simple paths in $G$ between $u$ and $z$ by $m(m-1)^t\pp^t(u,z)$ --- this observation provides the link to percolation. For any $p \in [0,1]$ we have
$$ \proba_p(u  \stackrel{m_0}{\lr} z) \leq \sum_{t=1}^{m_0} p^t m(m-1)^t \pp^t(u,z) \, ,$$
and therefore whenever \eqref{pc_estimate} holds for any $p\leq p_c$ we have
\be\label{eq:PercolationNBWRW} \proba_p(u  \stackrel{m_0}{\lr} z) \leq (1+o(1)) \sum_{t=1}^{m_0} \pp^t(u,z) \, .\ee

In contrary to what our notation suggests (we stick to it however due to its simplicity), the non-backtracking walk is not a Markov process on the vertices rather on directed edges (see \cite{Nach09}), hence it does {\bf not} hold that $\sum_z \pp^{t_1}(u,z) \pp^{t_2}(z,v) = \pp^{t_1+t_2}(u,v)$. A very close inequality is proved in Lemma 7 of \cite{Nach09}, however,  we will only need crude bounds on such convolution sums which we now describe.

Let $N(u,z; t)$ denote the number of non-backtracking paths from $u$ to $z$ of length $t$ so that
$$ \pp^{t_1}(u,z) = {N(u,z; t_1) \over m (m-1)^{t_1-1}} \qquad \pp^{t_2}(z,v) = {N(z,v; t_2) \over m (m-1)^{t_2-1}} \, .$$
Next we have that $\sum_z N(u,z; t_1) N(z,v; t_2)$ equals the number of paths from $u$ to $v$ which are allowed to backtrack \emph{only} at step $t_1$. The total number of such paths starting from $u$ (which do not necessarily end at $v$) is precisely $m^2 (m-1)^{t_1+t_2-2}$ and so we deduce that, 
\be\label{eq:NBWconvolution} \sum_{z}  \pp^{t_1}(u,z) \pp^{t_2}(z,v) = \pp^{t_1+t_2}(u,v; t_1) \, ,\ee
where the right hand side is the probability that a non-backtracking random walk starting from $u$ which is allowed to backtrack at step $t_1$ and has total length $t_1+t_2$, terminates at $v$. 
 
\section{Convergence of the component multiplicative graph $G_\mult$} \label{Sec:Mult}

In this section we study the (component) multiplicative graph $G_\mult$ described in \cref{Sec:outline}; we briefly repeat here its definition. We set 
\be\label{eq:SubcritParameters} p_s = p_c(1-V^{-1/3}\alpha_m^{-1/3}) \, , \qquad M_s=V^{2/3} \alpha_m^{4} \, ,\ee
and consider the set of components $\Cf_{p_s,M_s}$ of $H_{p_s}$ of size at least $M_s$. Each component $A$ is given weight $w_A=V^{-2/3} |A|$ as in \eqref{def:weights} and edges are open independently with probability $q_{A,B}:=1-e^{-q w_A w_B}$ where 
\begin{equation} \label{def:qlambda}
q_{\lambda}=  V^{1/3}/\chi(p_s) + \lambda + o(1) \, .
\end{equation}

For a connected component $\C_{\times}$ of $G_\times$ we write $|\C_\times|$ (resp.\ $\weight{\C_\times}$) for the sum of sizes (resp.\ weights) of its vertices, that is, 
$$ |\C_{\times}|:=\sum_{A\in \C_{\times}} |A| \ \quad ; \quad \weight{\C_{\times}}:=\sum_{A\in \C_{\times}} w_A \, .$$ 
We denote by $\C^\times_1, \C^\times_2,\ldots$ the connected components of $G_\times$ in decreasing order of their sizes. Our goal in this section is to show the following two propositions.

\begin{proposition} \label{pro:MCGMassConverges} For any fixed $\lambda\in \R$ set $q_\lambda$ as in \eqref{def:qlambda}. Then as $m\to \infty$ we have  
\[\label{thm:AldousInMCG} (\weight{\C^{\times}_1},\weight{\C^{\times}_2},\ldots) \stackrel{\mathrm{(d)}}{\longrightarrow} \ZZ_\llambda.\]
where the convergence in distribution is for  the $\ell^2$ topology and $\ZZ_\llambda$ is defined above \cref{main_thm}. 
\end{proposition} 

Let $d_\times(\cdot,\cdot)$ denote the shortest path metric on $G_\times$. For every $i\geq 1$, let $\mu^\times_i$ be the measure on $\C^\times_i$ defined by $\mu^\times_i(\{A\})=|A|/V^{2/3}$
for every $A\in \C^\times_i$ and finally let $M^\times_i$ be the mm-space
$$ M^\times_i = (\C^\times_i, \chi(p_s)V^{-1/3} \cdot d_\times, \mu^\times_i) \, .$$ 
\begin{proposition}\label{pro:MCGDistConverges} For any fixed $\lambda\in \R$ set $q_\lambda$ as in \eqref{def:qlambda}. Then as $m\to \infty$ we have  
\[ \label{thm:ABBGInHypercube} (M_1^\times, M_2^\times,\dots) \stackrel{\mathrm{(d)}}{\longrightarrow} \MM_\llambda \, ,\]
where the convergence in distribution is with respect to the product GHP topology and $\MM_\llambda$ is defined above \cref{thm:main_metric}.
\end{proposition}

The component multiplicative graph $G_\times$ is an instance of Aldous' \emph{multiplicative graphs} and the proof of the two theorems above will follow by verifying the conditions of theorems by Aldous \cite{Aldous97} and Bhamidi, Broutin, Sen and Wang \cite{BhBrSeWa2014a} regarding scaling limits of multiplicative graphs. We use these two results as \emph{black boxes}. 

Let us introduce the general setup of multiplicative graphs. Let $w=(w_1,\ldots,w_n)$ be a positive real vector and $q>0$. Consider the random graph on $\{1,\ldots, n\}$ so that each edge $\{i,j\}$ is present with probability $1-\exp(-q w_i w_j)$ independently of all other edges. A component $\C$ of the resulting graph has weight $\weight{\C}=\sum_{i \in \C} w_i$ and let $\C_1,\C_2,\ldots$ denote the components ordered in a weakly decreasing order of their weights. For $r=1,2,3$ define 
$ \sigma_r = \sum_{i} w_i^r \, .$

Aldous \cite[Proposition 4]{Aldous97} showed that for any a fixed $\lambda \in \R$, if
\be \label{AldousAssumptions} \frac{\sigma_3}{\sigma_2^3} \to 1 \qquad q - \frac{1}{\sigma_2} \to \lambda \qquad \frac{\max_i w_i}{\sigma_2} \to 0 \, ,\ee
then, 
\be\label{thm:Aldous} (\weight{\C_1},\weight{\C_2}, \ldots) \stackrel{\mathrm{(d)}}{\longrightarrow} \ZZ_\llambda \, .\ee

Next, for each $i\geq 1$ let $M_i=(\C_i, d_{\C_i}(\cdot,\cdot), \mu_{\C_i})$ be the mm-space on the vertex set $\C_i$ where $d_{\C_i}(\cdot,\cdot)$ is the shortest path metric, normalized so that every distance is multiplied by $\sigma_2$, and $\mu_{\C_i} =\sum_{j\in \C_i} w_j \delta_j$. Theorem 3.2 in \cite{BhBrSeWa2014a} states that if \eqref{AldousAssumptions} holds and there exist constants $\eta_0>0$ and $r_0>0$ such that
\be \label{BBSW:assumptions} 
\frac{\max_i w_i}{\sigma_2^{3/2+\eta_0}} \to 0 \qquad \text{and} \qquad \frac{\sigma_2^{r_0}}{\min_i w_i}\to 0 \, ,\ee
then, 
\be\label{thm:BBSW} (M_1, M_2,\ldots) \to \MM_\llambda \, .\ee

Hence, the proof of Propositions~\ref{pro:MCGMassConverges} and \ref{pro:MCGDistConverges} will follow once we verify that with  probability $1-o(1)$ assumptions \eqref{AldousAssumptions} and \eqref{BBSW:assumptions} hold. 
\begin{lemma}\label{lem:conditions_for_Hps}
Consider the components $\Cf_{p_s,M_s}$ with $p_s, M_s$ defined in \eqref{eq:SubcritParameters}. Define the weights $w_A = V^{-2/3} |A|$ for any $A\in \Cf_{p_s,M_s}$ and set $q=q(\lambda)$ as in \eqref{def:qlambda}. Then, with probability $1-o(1)$ the assertions in \eqref{AldousAssumptions} and \eqref{BBSW:assumptions} hold.
\end{lemma}

This will be done using the following lemmas whose proof is provided in the next subsection.

\begin{lemma}\label{lem:sigma2Concentration} Consider percolation on $H$ with edge probability $p_s$. Then for any $A\geq 1$
$$ \proba\Bigg ( \Bigg | \sum_{\substack{j \geq 1\\ |\C_j|\geq M_s}} |\C_j|^2 - V\chi(p_s) \Bigg | \geq  A V \X(p_s) \alpha_m^{5/6} \Bigg ) = O(A^{-2}) \, .$$
\end{lemma}

\begin{lemma}\label{lem:sigma3Concentration} Consider percolation on $H$ with edge probability $p=(1-\eps)p_c$ where $\eps \gg V^{-1/3}$ and $\eps=o(1)$. Let $M(\eps)$ be a sequence satisfying $M(\eps) = o(\eps^{-2})$. Then with probability tending to $1$ we have
$$ \sum_{\substack{j \geq 1 \\ |\C_j|\geq M(\eps)}} |\C_j|^3 = (1+o(1)) V\chi(p)^3 \, .$$
\end{lemma}

\begin{proof}[Proof of Lemma~\ref{lem:conditions_for_Hps}]
By \eqref{eq:subcritSucespt} we have that $\X(p_s) = (1+o(1)) V^{1/3} \alpha_m^{1/3}$. Then, by \cref{lem:sigma2Concentration} with $A=\alpha_m^{-1/6} \to \infty$ together with our choice of weights \eqref{def:weights}, we obtain that with probability tending to $1$ 
$$ \sigma_2 = \frac{\chi(p_s)}{V^{1/3}} \Big ( 1 + O(A\alpha_m^{5/6}) \Big )  \, .$$
Since $\chi(p_s)=\Theta(V^{1/3} \alpha_m^{1/3})$ and $A=\alpha_m^{-1/6}$ we deduce that 
$1/\sigma_2 = V^{1/3}/\chi(p_s) + O(\alpha_m^{1/3})$ with probability $1-o(1)$. Hence with the choice of $q_\lambda$ as in \eqref{def:qlambda} it follows that the second condition in \eqref{AldousAssumptions} holds. Next, by Theorem~1.2 of \cite{BCHSS1} (see (1.17) there) the maximal component size for percolation on $H$ at $p_s$ is at most $O(\chi(p_s)^2 \log(V/\chi^3(p_s))$ (this is in fact the right order, see \cite{HulNac20}) which implies that $\max_A w_A = O(\alpha_m^{2/3}\log(\alpha_m))$ while $\sigma_2 = \Theta(\alpha_m^{1/3})$ so the third condition in \eqref{AldousAssumptions} holds. Lastly, \cref{lem:sigma3Concentration} and our choice of weights in \eqref{def:weights}, yield $\sigma_3 = (1+o(1))\alpha_m$ while $\sigma_2 = (1+o(1))V^{1/3}/\chi(p_s) = (1+o(1))\alpha_m^{1/3}$, so that the first condition in \eqref{AldousAssumptions} holds. 

For \eqref{BBSW:assumptions}, since $\sigma_2=O(\alpha_m^{1/3})$ and $\max_A w_A = O(\alpha_m^{2/3}\log(\alpha_m))$, the first condition of \eqref{BBSW:assumptions} holds with any choice of $\eta_0\in(0,1/6)$. Secondly, since $\min_A w_A \geq V^{-2/3} M_s = \alpha_m^{4}$ and $\sigma_2 = \Theta(\alpha_m^{1/3})$ the second condition of \eqref{BBSW:assumptions} holds with any fixed $r_0 > 12$. 
\end{proof}

For technical reasons we also need to show the following convergence in expectation, { conditionally on the graph $H_{p_s}$, or even just on the collection of weights $\{w_A: A\in \Cf_{p_s,M_s}\}$}:
\begin{lemma}\label{CVX} For any fixed $\lambda\in \R$ set $q_\lambda$ as in \eqref{def:qlambda}. As $m\to \infty$, writing $\ZZ_\llambda = (|\gamma^\llambda_1|,|\gamma^\llambda_2|,\ldots)$, 
\[\label{thm:AldousEsp} 
\E\Bigg[ \sum_{i\geq 1} V^{-4/3} |\C^{\times}_i|^2 { ~\Bigg|~ H_{p_s}\Bigg ]} \limit  \E\left [\sum_{i\in \N} |\gamma^{\llambda}_i|^2 \right ]<\infty, \qquad \text{in probability.}\]
\end{lemma}
\subsection{Subcritical estimates for $\sigma_2$ and $\sigma_3$: Proofs of Lemmas~\ref{lem:sigma2Concentration} and \ref{lem:sigma3Concentration}}

The proofs of Lemmas~\ref{lem:sigma2Concentration} and \ref{lem:sigma3Concentration} rely on estimates for the expected values $\E_p|\C(v)|$ and $\E_p|\C(v)|^2$, respectively, for $p<p_c$. The first one is in \eqref{eq:subcritSucespt}, and we may proceed to the proof of \cref{lem:sigma2Concentration}.

\begin{proof}[Proof of Lemma \ref{lem:sigma2Concentration}] Write 
$$ Y = \sum_{\substack{j \geq 1 \\ |\C_j|\geq M_s}} |\C_j|^2 = \sum_{v}|\C(v)|{\bf 1}_{\{|\C(v)|\geq M_s\}} \, .$$ 
By \eqref{eq:ClusterTail}, $\E|\C(v)|{\bf 1}_{\{|\C(v)|\leq M_s\}} = O(V^{1/3} \alpha_m^{2})$. And since, by \eqref{eq:subcritSucespt}, $\chi(p_s) = \Theta(V^{1/3} \alpha_m^{1/3})$,  we deduce 
$$\E Y = V\chi({p_s})(1- O(\alpha_m^{5/3})).$$
 To upper bound the second moment of $Y$ we drop the requirement that $|\C(v)|\geq M_s$ and bound
$$ \E Y^2 \leq  \sum_{u,v}\E |\C(v)||\C(u)| = \sum_{u,v}\E |\C(v)||\C(u)| \ind_{u \nlr v} + \sum_{u,v}\E |\C(v)||\C(u)| \ind_{u \lr v} \, .$$
We bound the first term by $V^2\chi(p)^2$ using the BK inequality. The second term equals $\sum_v \E|\C(v)|^3 = V\E|\C(v)|^3$ which, by the tree-graph inequalities (see (6.94) in \cite{Grimmett}), is at most $O(V \chi(p)^5)$. Hence Var$(Y) = O(V^{8/3} \alpha^{7/3})=O(V^2\chi(p_s)^2 \alpha_m^{5/3})$ and the lemma follows by Chebychev's inequality.
\end{proof}

Before the proof of \cref{lem:sigma3Concentration}, we first obtain a sharp estimate on $\E|\C(v)|^2$ in the subcritical phase:

\begin{lemma}\label{lem:sigma3FirstMom} For any $p=p_c(1-\eps)$ where $\eps \gg V^{-1/3}$ and $\eps=o(1)$ we have
$$ (1-o(1)) \chi(p)^3 \leq \E|\C(v)|^2 \leq \chi(p)^3 \, .$$
\end{lemma}
\begin{proof}
The upper bound is just tree-graph inequalities, i.e., if $v\lr x$ and $v \lr y$, then there must exists $z$ such that $\{v\lr z\}\circ\{z \lr x\}\circ\{z \lr y\}$, thus using the BK inequality and summing over $x,y$ then $z$ gives the upper bound of $\chi(p)^3$. 
To prove the lower bound let $x,y,z$ be three vertices such that $v,x,y,z$ are distinct vertices and write $A_{x,y,z}$ for the event that (see Figure \ref{Fig:3.6})
\begin{compactenum}[i)]
\item there exist three edge disjoint simple paths $\gamma_{v,z}, \gamma_{z,x}, \gamma_{z,y}$, between $v$ and $z$, between $z$ to $x$ and between $z$ to $y$, respectively;
\item the paths $\gamma_{v,z}, \gamma_{z,x}, \gamma_{z,y}$ only intersect at $z$ that is an endpoint of all three paths
\item removing the vertex $z$ separates $v,x,y$ to three separate connected components in $G_p$.
\end{compactenum}
 \begin{figure}[t]
\centering
\includegraphics[scale=0.8]{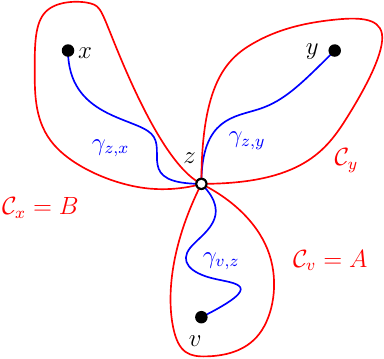}
\caption{\label{Fig:3.6}A representation of the event $A_{x,y,z}$. The paths $\gamma_{v,z}$, $\gamma_{z,x}$, $\gamma_{z,y}$ are blue. The interior of the red curves represents $\C_v, \C_x,\C_y$ which do \emph{not} contain $z$. 
} 
\end{figure}%
We first claim that for any fixed $x\neq y$ both distinct from $v$, the events $A_{x,y,z}$, $z\in V\setminus\{v,x,y\}$ are disjoint. Indeed, assume that $A_{x,y,z_1}$ and $A_{x,y,z_2}$ occur where $z_1\neq z_2$ and write $\C^{(z_1)}_v,\C^{(z_1)}_x,\C^{(z_1)}_y$ for the connected components in $G_p \setminus \{z_1\}$ of $v,x,y$ respectively. We reach a contradiction by examining to which connected component $z_2$ belongs to. If $z_2 \in \C^{(z_1)}_v$, then we must have that both $\gamma_{z_2,x}$ and $\gamma_{z_2,y}$ use the vertex $z_1$, contradicting $A_{x,y,z_2}$. If $z_2 \in \C^{(z_1)}_x$, then the path $\gamma_{v,z_2}$ and $\gamma_{z_2,y}$ must both use the vertex $z_1$, again contradicting $A_{x,y,z_2}$, and similarly if $z_2 \in \C^{(z_1)}_y$. Lastly, if $z_2 \in G \setminus \C^{(z_1)}_v\cup \C^{(z_1)}_x\cup \C^{(z_1)}_y$, then all three paths $\gamma_{v,z_2}, \gamma_{z_2,x}, \gamma_{z_2,y}$ must use the vertex $z_1$ and we obtain another contradiction. 

Secondly, if $A_{x,y,z}$ occurs for some $z$, then $v \lr x$ and $v\lr y$. We deduce that
\begin{equation} \sum_{x,y: |\{x,y,v\}|=3} \proba(v \lr x, v\lr y) \geq \sum_{x,y,z : |\{x,y,v,z\}|=4} \proba(A_{x,y,z}) \, . \label{eq:Lem3.6Split}\end{equation}
We now lower bound $\proba(A_{x,y,z})$. Let $\C_x$ denote the connected component of $x$ in $G_p\backslash \{z\}$ when $x\lr z$ and let $\C_x:=\emptyset$ when $x\snlr z$. Define similarly $\C_y$, $\C_v$. Given a set of vertices $A$ such that $z \not \in A$, by conditioning on the event $\C_v=A$ we mean that we condition on the status of all edges needed to determine that $\C_v=A$, that is, on all the open and closed edges between the vertices of $A\cup \{z\}$ (it must be that these open edges span a connected graph on $A$ and if $v \lr z$, then also on $A\cup \{z\}$) and on all the closed edges of the form $\{a,w\}$ where $a\in A$ and $w\not \in A \cup \{z\}$. Note that in this conditioning it is possible that some edges between the vertices of $A$ and $z$ are closed and some are open; in the case that $v\lr z$ then at least one such edge must be open. 

For any two disjoint sets $A$ and $B$ we condition as above on $\C_v=A$ and $\C_x=B$ and get
\[ \label{lemchi2:step1} \proba(A_{x,y,z}) = \sum_{\substack{A,B : v\in A, x \in B \\ A \cap B= \emptyset, y \not \in A \cup B} } \proba(\C_v=A, \C_x=B)\proba(z \lr y \mid \C_v=A,\C_x=B) \, .\]
Conditioned on $\C_v=A,\C_x=B$, as explained above, and given $y\not \in A\cup B$, the event $z \lr y$ occurs if and only if there exists an open path from $z$ to $y$ that avoids any of the vertices in $A\cup B$, or in other words, $z \lr y$ off $A\cup B$. Hence
$$ \proba(z \lr y | \C_v=A,\C_x=B) = \proba(z \lr y \textrm{ off } A\cup B) \, .$$
This equals the probability that $y$ is connected to $z$ minus the probability that $y$ is connected to $z$ and any open path connecting $y$ to $z$ visits $A\cup B$. Hence we bound from below using the BK inequality
\begin{equation} \label{eq:lemchi2:step2} \sum_{x,y,z} \proba(A_{x,y,z}) \geq S_1 - S_2 \end{equation}
where 
$$ S_1 = \sum_{\substack{ x,y,z \\ |\{x,y,z\}|=3}} \sum_{\substack{A,B : v\in A, x \in B, \\ A \cap B= \emptyset, y \not \in A \cup B} }  \proba(\C_v=A, \C_x=B) \proba(z \lr y) \, ,$$
and
$$ S_2 = \sum_{\substack{ x,y,z \\ |\{x,y,z\}|=3}} \sum_{\substack{A,B : v\in A, x \in B \\  A \cap B= \emptyset, y \not \in A \cup B} } \proba(\C_v=A, \C_x=B)  \sum_{w \in A\cup B} \proba(y \lr w)\proba(w \lr z)   \, .$$

We start by lower bounding $S_1$. We claim that if $A$ and $B$ are two disjoint sets of vertices with $v\in A$ and $x\in B$, then $\proba(\C_x = B \mid \C_v = A) = \proba(\C_x = B \off A)$. Indeed, to determine whether $\C_x=B$ off $A$ we need to observe all open edges between the vertices of $B \cup \{z\}$ and all closed edges of the form $\{b,w\}$ where $b\in B$ and $w \not\in A\cup B \cup \{z\}$. On the other hand, to determine $\C_v=A$ the set of edges we observe, as explained above, is disjoint from the set of edges, that we just described, required to determine $\C_x=B$ off $A$. Next, we have that 
\begin{align*} \sum_{\substack{B : A\cap B = \emptyset, x\in B, y\not \in B}} \proba(\C_x = B \textrm{ off } A) &\geq  \proba(x \lr z, y \nlr z \off A) \\ &\geq \proba(x \lr z \off A) - \proba(x\lr z, y \lr x) \, .\end{align*}
We use the BK inequality to bound 
$$ \proba(x \lr z \off A) \geq \proba(x \lr z) - \sum_{w\in A} \proba(x \lr w)\proba(w \lr z) \, ,$$
and 
$$ \proba(x\lr z, y \lr x) \leq \sum_{w\in V} \proba(x \lr w) \proba(w \lr z) \proba(w \lr y) \, .$$
Thus 
\begin{equation} S_1 \geq S_1^{(a)} - S_1^{(b)} \, ,\label{eq:Center3.6} \end{equation}
where
$$ S_1^{(a)} = \sum_{\substack{ x,y,z \\ |\{x,y,z\}|=3}} \proba(z \lr y) \sum_{\substack{A : v\in A, y\not \in A}} \proba(\C_v=A) \Big [ \proba(x \lr z) - \sum_{w\in A} \proba(x \lr w)\proba(w \lr z) \Big ] \, ,$$
and
$$ S_1^{(b)} = \sum_{\substack{ x,y,z,w \\ |\{x,y,z\}|=3}} \proba(z \lr y) \proba(x \lr w) \proba(w \lr z) \proba(w \lr y) \sum_{\substack{A : v\in A, y\not \in A}} \proba(\C_v=A)  \, .$$

To lower bound $S_1^{(a)}$ we have by the BK inequality
\begin{align*} \sum_{A : v \in A, y \not \in A} \proba(\C_v=A)  &\geq \proba(v \lr z, v \nlr y) \\ &\geq \proba(v \lr z) - \sum_{w\in V} \proba(v \lr w) \proba(w \lr z) \proba(w \lr y) \, ,\end{align*}
and that 
\begin{multline*} \sum_{A : v \in A, y \not \in A} \proba(\C_v=A) \sum_{w \in A} \proba(x \lr w)\proba(w \lr z) \\ = \sum_{w\in V\setminus \{z\}} \sum_{\substack{A : v \in A, \\ w\in A, y \not \in A}} \proba(\C_v=A) \proba(x \lr w)\proba(w \lr z) \, ,\end{multline*}
since $z \not \in A$. We upper bound this (since it appears with a negative sign) by
\begin{multline*} \sum_{w\neq z} \proba(v \lr z, v\lr w) \proba(x \lr w)\proba(w \lr z)\\ \leq \sum_{w\neq z, u \in V} \proba(v \lr u)\proba(u \lr z) \proba(u \lr w) \proba(x \lr w)\proba(w \lr z) \, ,\end{multline*}
using the BK inequality. Putting all these together gives a lower bound of
\begin{align*} S_1^{(a)} &\geq \sum_{\substack{ x,y,z \\ |\{x,y,z\}|=3}} \proba(z \lr y) \proba(x \lr z) \proba(v \lr z) \\
&- \sum_{\substack{ x,y,z,w \\ |\{x,y,z\}|=3}} \proba(z \lr y) \proba(x \lr z) \proba(v \lr w) \proba(w \lr z) \proba(w \lr y) \\
&- \sum_{\substack{ x,y,z,w,u \\ |\{x,y,z\}|=3, w\neq z}} \proba(z \lr y) \proba(v \lr u)\proba(u \lr z) \proba(u \lr w) \proba(x \lr w)\proba(w \lr z) \, .
\end{align*}

To lower bound the first sum, we first sum over $x \not \in \{v,y,z\}$, then over $y \not\in \{z,v\}$ and lastly over $z \neq v$. This gives a lower bound of $(\chi(p)-3)(\chi(p)-2)(\chi(p)-1)$ which is $(1-o(1))\chi(p)^3$. To upper bound the second we first sum over $x$ and get a contribution of $\chi(p)$, then three terms over $z\neq y$ which by the triangle condition in \eqref{eq:TriangleCondition} is $o(1)$, lastly the sum over $w$ gives another $\chi(p)$; we get a total bound of $o(\chi(p)^2)$. To upper bound the third sum, we first sum over $x$ and $y$ and get a contribution of $\chi(p)^2$, then over $z\neq w$ three terms and get again a $o(1)$ factor by \eqref{eq:TriangleCondition}, lastly we sum over $u$ and get another $\chi(p)$; getting a total contribution of $o(\chi(p)^3)$. We deduce that 
\begin{equation} S_1^{(a)} \geq (1-o(1)) \chi(p)^3. \label{eq:a3.6} \end{equation}

 Bounding from above $S_1^{(b)}$ is easier since we may just bound in its definition the rightmost sum over $A$ by $\proba(v \lr z)$ so 
\begin{equation} S_1^{(b)} \leq \sum_{x,y,z,w} \proba(z \lr y) \proba(x \lr w) \proba(w \lr z) \proba(w \lr y) \proba(v \lr z) = O(\chi(p)^2) \, , \label{eq:b3.6} \end{equation} 
by summing first over $x$ and getting a $\chi(p)$ factor, then over $y,w$ and getting an $O(1)$ factor by \eqref{eq:TriangleCondition}, lastly the remaining sum over $u$ is $\chi(p)$. We conclude from \eqref{eq:Center3.6} and $\eqref{eq:a3.6}-\eqref{eq:b3.6}$ that
\begin{equation} S_1 \geq (1-o(1)) \chi(p)^3 \, . \label{eq:lem3.6_S1} \end{equation} 

Next we bound from above $S_2$. We first change the order of summation and get
$$S_2 = \sum_{\substack{ x,y,z,w \\ |\{x,y,z\}|=3, w\neq z}}  \proba(y \lr w)\proba(w \lr z) \sum_{\substack{A,B : v\in A, x \in B, w \in A\cup B \\  A \cap B= \emptyset, y \not \in A \cup B} } \proba(\C_v=A, \C_x=B) \big ]  \, .$$
Next we split the rightmost sum over $A$ and $B$ according to whether $w\in A$ or $w\in B$. When $w\in A$ we upper bound
$$ \sum_{\substack{A,B : v\in A, x \in B, w \in A \\  A \cap B= \emptyset, y \not \in A \cup B} } \proba(\C_v=A, \C_x=B) \leq \proba \big( \{v \lr z, v\lr w\} \circ \{x \lr z\} \big ) \, ,$$
and replacing in the above $w\in A$ by $w\in B$ we obtain the upper bound  $\proba (\{v \lr z\} \circ \{x \lr z, x\lr w\})$. We use the BK inequality, then we use it again in the usual manner to deal with $\{v \lr z, v\lr w\}$ and $\{x \lr z, x\lr w\}$. This gives a bound 
\begin{align*} S_2 &\leq  \sum_{\substack{x,y,z,w,u \\ z \neq w }} \proba(y \lr w)\proba(w \lr z) \proba(v \lr u) \proba(u \lr z) \proba(u \lr w) \proba(x \lr z) \\
&+\sum_{\substack{x,y,z,w,u \\ z \neq w }} \proba(y \lr w)\proba(w \lr z)\proba(v \lr z)\proba(x \lr u)\proba(u\lr z)\proba(u\lr w) \, .\end{align*}
For the first part, we sum over $x$ and $y$ and get $\chi(p)^2$ contribution, then three terms over $z\neq w$ using \eqref{eq:TriangleCondition} and get an $o(1)$ factor, lastly over $w$ to get another $\chi(p)$; we get a total contribution of $o(\chi(p)^3)$. The second term is handled similarly, summing over $y$ and $x$ first, then over $w,u$ using \eqref{eq:TriangleCondition} and lastly over $z$. This gives another $o(\chi(p)^3)$ contribution. Therefore,
\begin{equation} S_2 = o(\chi(p)^3). \label{eq:lem3.6_S2} \end{equation}
Putting together \eqref{eq:Lem3.6Split}, \eqref{eq:lemchi2:step2}, \eqref{eq:lem3.6_S1} and \eqref{eq:lem3.6_S2} concludes the proof.
\end{proof}

With \cref{lem:sigma3FirstMom} under our belt, we are now ready for the proof of \cref{lem:sigma3Concentration}, which is similar to that of \cref{lem:sigma2Concentration}. 

\begin{proof}[Proof of \cref{lem:sigma3Concentration}] Put 
$$ Y = \sum_{\substack{j \geq 1 \\ |\C_j|\geq M(\eps)}} |\C_j|^3 = \sum_{v \in V} |\C(v)|^2 {\bf 1}_{\{ |C(v)|\geq M(\eps) \}} \, .$$
By \eqref{eq:ClusterTail}, by Abel summation, we have $\E |\C(v)|^2 {\bf 1}_{\{ |C(v)|\leq M(\eps)\}} \leq 2C\sum_{k=1}^{M(\eps)} k^{1/2} = O(M(\eps)^{3/2})$. Since $M(\eps) = o(\eps^{-2})$ this is $o(\eps^{-3})$ and so 
by \cref{lem:sigma3FirstMom} and \eqref{eq:subcritSucespt} we obtain that $\E Y = (1+o(1)) V \chi(p)^3$. For the second moment we bound
$$ \E Y^2 = \sum_{u,v} \E \left[ |\C(v)|^2|\C(u)|^2 \right ]= \sum_{u,v} \E\left[ |\C(v)|^2|\C(u)|^2\ind_{u \nlr v} \right ]+ \sum_{u,v} \E \left [|\C(v)|^2|\C(u)|^2\ind_{u \lr v} \right ]\, .$$
The first term on the right-hand side we bound using the BK inequality by $V^2 \chi(p)^6$. The second term on the right-hand side equals $\sum_v \E |\C(v)|^5$ which is $O(V \chi(p)^9)$ by the tree-graph inequalities (i.e., see (6.94) in \cite{Grimmett}). The latter estimate is $o(V^2 \chi(p)^3)$ since $\chi(p) = o(V^{1/3})$. We deduce that 
$\E Y^2 = (1+o(1))V^2 \chi(p)^6$ and conclude the proof using Chebyshev's inequality.
\end{proof} 

\newcommand{\tw}{{\mathtt w}}
\newcommand{\bw}{{\mathbf w}}
\newcommand{\bP}{{\mathbb P}}
\newcommand{\bE}{{\mathbb E}}

\subsection{Convergence of the susceptibility for multiplicative graph: Proof of \cref{CVX}} \label{Sec:MultX}
Lemma~\ref{lem:conditions_for_Hps} asserts that the conditions of \cite[Proposition 4]{Aldous97}, repeated in \eqref{AldousAssumptions}, hold in probability. Hence, to show \cref{CVX}, it suffices to prove that in the setting of Aldous \cite[Proposition 4]{Aldous97} one also has convergence of the expectation of the $\ell^2$ norm of the vector of sizes. We emphasize the fact that, in this entire section, the weights are deterministic as in \cite{Aldous97}. For $n\ge 1$, let $w^{(n)}=(w^{(n)}_1,w^{(n)}_2,\dots)$ be a sequence of non-negative weights, and let $q_n\ge 0$ be the associated time parameter. 

\begin{proposition}\label{prop:L2ConvMultGraph}Suppose that \eqref{AldousAssumptions} holds for $(w^{(n)}_i)_{i\ge 1}$ and $q_n$ as $n\to\infty$. Then, as $n\to \infty$, 
\[\label{thm:AldousEsp} \bE\Bigg [\sum_{i\geq 1} \weight{\C^{(n)}_i}^2 \Bigg] \limit  \E\Bigg [\sum_{i\ge 1} |\gamma^\llambda_i|^2 \Bigg ]<\infty.\]
\end{proposition}

Before proceeding, we mention the following corollary that we will later use to determine the position within the critical window in \cref{sec:position}. It is undeniably floklore, but we failed to find a reference. The classical Erd\H{o}s--R\'enyi random graph is also a multiplicative graph, and we obtain:
\begin{corollary} \label{cor:Masses} For every $\lambda\in \R$, in the Erd\H{o}s--R\'enyi model $G(n,1/n+\llambda n^{-4/3})$, as $n\to \infty$, we have $\E|\C(v)|/n^{1/3}\to \kkappa(\lambda) := \lim_{n\to \infty}  \E [\sum_{i\in \N} |\gamma^{\llambda}_i|^2  ]$.
\end{corollary} 
\begin{proof}Fixing $w_i=n^{-2/3}$ for $i=1,\dots, n$, we have $\sigma_2=n^{-1/3}$, $\sigma_3=n$ so that \eqref{AldousAssumptions} is satisfied with $q= n^{1/3}+\lambda$. The corresponding edge probability is $1-\exp(qn^{-4/3})=1/n+\lambda n^{-4/3}+o(n^{-4/3})$. \cref{prop:L2ConvMultGraph} then implies that $n^{-4/3}\E \sum_{i\geq 1}|\C_i|^2 \to \kkappa(\lambda)$ and since $\E[|\C(v)|]=n^{-1}\E[\sum_i |\C_i|^2]$ the claim follows. 
\end{proof}

Lemma 25 of~\cite{Aldous97} shows that $\E[\sum_{i\ge 1} |\gamma^\llambda_i|^2] <\infty$.
Furthermore, Proposition~4 of \cite{Aldous97} shows that $(\weight{\C^{(n)}_i})_{i\ge 1}$ converges in distribution to $(|\gamma^\llambda_i|)_{i\ge 1}$ with respect to $\ell^2$. It follows in particular that the $\ell^2$-norms converge in distribution, namely $S_n:=\sum_{i\ge 1} \weight{\C^{(n)}_i}^2 \to S:=\sum_{i\geq 1} |\gamma^\llambda_i|^2$ in distribution. So proving \cref{prop:L2ConvMultGraph} boils down to showing that $(S_n)_{n\ge 1}$ is a uniformly integrable sequence. We shall use the following criterion.

\begin{lemma}\label{lem:criterion_ui}
Let $(X_n)_{n\ge 1}$ be a tight sequence of non-negative random variables. Suppose that 
	\begin{equation}\label{eq:towards_ui}
	K:=\sup_{n\ge 1} \big\{\E[X_n^2] - 2 \E[X_n]^2\big\} <\infty\,
\end{equation}
then $(X_n)_{n\ge 1}$ is uniformly integrable.
\end{lemma}
\begin{proof} 
First note that if $(X_n)_{n\ge 1}$ is bounded in $L^1$, then \eqref{eq:towards_ui} implies that it is also bounded in $L^2$; a standard argument then shows that $(X_n)_{n\ge 1}$ is uniformly integrable. So let us now prove that $(X_n)_{n\ge 1}$ is bounded in $L^1$. The Cauchy--Schwarz inequality implies that for any $\alpha\ge 0$,  
\[\E[X_n \1_{X_n\geq \alpha}]^2 \le \E[X_n^2]\cdot \proba(X_n \geq \alpha) \le \E[X_n^2] \cdot \sup_{n\ge 1} \proba(X_n \geq \alpha)\,.\]
Since $(X_n)_{n\ge 1}$ is tight, there exists $\alpha_0\ge 0$ large enough that $\sup_n \proba(X_n \geq \alpha_0)\leq1/3$. Now $X_n \1_{X_n\geq\alpha_0} \geq X_n - \alpha_0$, so either $\E X_n \leq \alpha_0$, or $\E[X_n \1_{X_n\geq \alpha}]^2 \geq (\E X_n - \alpha_0)^2$. In the second case we plug this into the left-hand side above and obtain that
\[3(\E[X_n]-\alpha_0)^2 \le \E[X_n^2] \le  2 \E[X_n]^2 + K \, ,\]
where the second inequality is due to \eqref{eq:towards_ui}.
We obtain that either $\E X_n \leq \alpha_0$ or by rearranging that $\E[X_n]^2 \le 6 \alpha_0 \E[X_n] + K$. This immediately implies that $\sup_n \E[X_n]<\infty$ as desired. 
\end{proof}

The proof of the uniformly integrability of $(S_n)_{n\ge 1}$ requires to control both the large and small components. This requires different arguments which we perform separately.  

\begin{lemma}\label{lem:multiplicative_graph_largest_cc}
Suppose that \eqref{AldousAssumptions} holds for $(w^{(n)}_i)_{i\ge 1}$ and $q_n$ as $n\to\infty$. 
Then, for all $x\ge 0$ large enough, we have $\limsup_n \bP(\weight{\C_1^{(n)}}>x) \le 2 e^{-x/2}$.
\end{lemma}

The proof of \cref{lem:multiplicative_graph_largest_cc} is based on the exploration process used by Limic in \cite{Limic2019a} (see also \cite{Uribe2007}), which is more convenient than the one used by Aldous \cite{Aldous97}. Define $Y_t^{(n)}=-t+\sum_i w_i^{(n)} \1(E^{(n)}_i\le t)$, where $(E^{(n)}_i)_{i\ge 1}$ is a collection of independent exponential random variables with respective rates $q_nw^{(n)}_i$, and let $\underline Y_t^{(n)}=\inf\{Y_s^{(n)}: 0\le s\le t\}$ denote the infimum process. Then the lengths of the excursions of the process $(Y_t^{(n)})_{t\ge 0}$ strictly above its infimum process $(\underline Y^{(n)}_t)_{t\ge 0}$ are distributed like the sum of vertex weights of the connected components $(\weight{\C_i^{(n)}})_{i\ge 1}$ of the multiplicative graph (\cite[Proposition 5]{Limic2019a} and \cite[Theorem~2.1]{BrDuWa2022b}). In the following, we identify $(\weight{\C_i^{(n)}})_{i\ge 1}$ with the sorted lengths of the excursions. Limic \cite[Proposition~6]{Limic2019a} showed that, under the conditions in \eqref{AldousAssumptions}, the process $(Y^{(n)}_t/\sigma_2)_{t\ge 0}$ converges in distribution to Brownian motion with drift and deduced the convergence of the excursion lengths, and hence of the weights of the connected components. \cref{lem:multiplicative_graph_largest_cc} provides quantitative bounds, valid for all large enough $n$, on the tail of $\weight{\C^{(n)}_1}$.

During each excursion the infimum process $Y^{(n)}_t$ is constant and takes a single value. For each $i\geq 1$ we let $L^{(n)}_i$ denote this value for the $i$th largest excursion. 
The following fact shows that the process $(|\underline{Y}_t|)_{t\ge 0}$ is an exact local time. 
\newcommand{\cF}{{\mathcal F}}
\begin{lemma}\label{lem:local_time}
Conditionally on $(\C_i^{(n)})_{i\ge 1}$ the random variables $(L_i^{(n)})_{i\ge 1}$ are independent exponential random variables with respective rates $(q_n \weight{C_i^{(n)}})_{i\ge 1}$.
\end{lemma}

\newcommand{\cG}{{\mathcal G}}
\begin{proof}
Recall that $(\C^{(n)}_i)_{i\ge 1}$ denotes the collection of connected components, in decreasing order of the sum of the weights, breaking ties using the minimum label if necessary. Let $k$ denote the number of connected components. 
Set $\tau_0=0$, and let $\tau_1<\tau_2<\dots<\tau_k$ denote the increasing reordering of the $(L_i^{(n)})_{i\ge 1}$. So $(\tau_i)_{i\in [k]}$ are the values of $|\underline{Y}_t^{(n)}|$ on each excursion, in increasing order. It will also be convenient to know in which order the connected components arrive: for each $i\in [k]$ let $J_i=j$ if the $i$th connected component to be discovered (at time $\inf\{t\ge 0: |\underline{Y}_t^{(n)}| = \tau_i\}$) is $\C^{(n)}_j$. To prove the claim, we shall verify that the collection of inter-arrival times and permutation of the connected components $(\tau_i-\tau_{i-1}, J_i)_{1\le i\le k}$ have the correct joint distribution. 

For each $i$ let $U_i$ denote the first discovered vertex of component $\C^{(n)}_{J_i}$, namely $U_i=u$ if $\inf\{t\ge 0: |\underline{Y}_t^{(n)}|=\tau_i\}=E^{(n)}_u$. We denote by $\Pi^{(n)}$ the partition of $[n]$ induced by the connected components $(\C^{(n)}_i)_{i\ge 1}$. By Theorem~3.2 of \cite{BrDuWa2022b}, for any fixed partition $\pi$ of $[n]$ having $k$ parts denoted by $C_1,\dots, C_k$, in decreasing order of the sum of weights, breaking ties using minimum element, any $t_1,t_2,\dots, t_k\in \R$ and any $(u_1,u_2,\dots, u_k)\in C_{j_1}\times \dots \times C_{j_k}$ we have
\begin{align*}
\bP(\Pi^{(n)}=\pi, \tau_i \in dt_i, U_i=u_i, i\in [k])
&=\bP(\Pi^{(n)} = \pi) \cdot \mathbf{1}_{t_1<\dots<t_k} \!\!\!\prod_{1\le i\le k} q_n w^{(n)}_{u_i} e^{-q_n t_i \weight{C_{j_i}}} dt_1\cdots dt_k\, ,
\end{align*}
where $C_{j_i}$ is the component containing $u_i$, or equivalently, the component whose corresponding excursion started at time $\tau_i$. It follows that the conditional distribution of $(\tau_i,J_i)_{i\ge 1}$ is specified by
\begin{align*}
	\bP(\tau_i \in dt_i, J_i=j_i, i\in [k]~|~ \Pi^{(n)}=\pi)
	& = \mathbf{1}_{t_1<\dots<t_k}  \!\!\!\!\!\!\! \sum_{u_1\in C_{j_1}, \dots, u_k\in C_{j_k}} \prod_{1\le i\le k} q_n w^{(n)}_{u_i} e^{-q_n t_i \weight{C_{j_i}}}\cdot dt_1\cdots dt_k\\
	& = \mathbf{1}_{t_1<t_2<\dots<t_k} \prod_{1\le i\le k} q_n \weight{C_{j_i}} e^{-q_n t_i \weight{C_{j_i}}} dt_1\cdots dt_k\,.
\end{align*}
Changing variables to obtain the distribution of the inter-arrival times $(\tau_i-\tau_{i-1})_{i\in [k]}$, we obtain, for $x_1,x_2,\dots, x_k\ge 0$
\begin{align*}
	\bP(\tau_i-\tau_{i-1} \in dx_i, J_i=j_i, i\in [k]~|~ \Pi^{(n)}=\pi) 
	& = \prod_{1\le i\le k} q_n \weight{C_{j_i}} e^{-q_n (x_1+\dots+x_i) \weight{C_{j_i}}} dx_1\cdots dx_k\\
	& = \prod_{1\le i\le k} q_n \weight{C_{j_i}} e^{-x_i q_n (\weight{C_{j_i}}+\cdots+\weight{C_{j_k}})}  dx_1\cdots dx_k\,.
\end{align*}
It is a standard fact that this is the joint distribution of the inter-arrival times between independent exponential random variables with rates $(q_n \weight{C_i})_{i\in [k]}$ and the corresponding permutation of the indices. The claim follows.
\end{proof}

We are now ready to prove \cref{lem:multiplicative_graph_largest_cc}.

\begin{proof}[Proof of Lemma~\ref{lem:multiplicative_graph_largest_cc}]
By Lemma~\ref{lem:local_time} 
\begin{equation}\label{eq:tail_local-time_largest}
\bP(L^{(n)}_1 > \sigma_2~|~\weight{\C_1^{(n)}}= x) 
= \exp(-xq_n \sigma_2) \le e^{-x/2}\,,
\end{equation}
for all $n$ large enough, since $q_n\sigma_2\to 1$ by the assumptions in \eqref{AldousAssumptions}.

Now consider the value of the exploration process $(Y_t^{(n)})_{t\ge 0}$ at time $x$: observe that if $L^{(n)}_1\le \sigma_2$ and $\weight{\C^{(n)}_1}> x$, then $Y_x^{(n)}\ge -\sigma_2$. Indeed, either we have not yet discovered $\C_1^{(n)}$ at time $x$, which implies that $Y_x^{(n)}\ge \underline Y_x^{(n)} > -L^{(n)}_1\ge -\sigma_2$, or we have, but then $x$ belongs to this excursion interval, and hence $Y_x^{(n)}\ge \underline Y_x^{(n)} = - L^{(n)}_1$. Together with \eqref{eq:tail_local-time_largest}, this implies 
\begin{align}\label{eq:tail_largest}
	\bP(\weight{\C_1^{(n)}} > x) 
	& = \bP^{(n)}(\weight{\C_1^{(n)}}> x, L^{(n)}_1 \le \sigma_2|) 
	+ \bP(\weight{\C_1^{(n)}}> x,L^{(n)}_1 > \sigma_2) \notag \\ 
	& \le \bP(Y_x^{(n)} \ge -\sigma_2)+ 
	\bP(L^{(n)}_1 > \sigma_2~|~ \weight{\C_1^{(n)}}> x) \notag \\ 
	& \le \bP(Y_x^{(n)} \ge - \sigma_2) +  e^{-x/2}\,.
\end{align}

Since $Y^{(n)}_x$ is a sum of independent terms, bounded by $\max_i w_i^{(n)}$, and whose expected values are easy to control, one easily obtains a bound on the first term above using a standard concentration inequality. Indeed, routine bounds then the asymptotics in \eqref{AldousAssumptions} yield, for all $n$ large enough,
\begin{align*}
	\E[Y^{(n)}_x] 
	= \sum_{i\ge 1} w^{(n)}_i (1-\exp(-x q_n w^{(n)}_i)) - x 
	\le  \sigma_2 x \bigg( q_n - \frac 1 {\sigma_2}\bigg) - \frac {x^2}4 q_n^2 \sigma_3
	\sim \sigma_2 \bigg(x \lambda - \frac{x^2}4\bigg)\,.
\end{align*}
Similarly, we have for all $n$ large enough,
\begin{align*}
	\text{Var}[Y^{(n)}_x] 
	\leq \sum_{i\ge 1} (w^{(n)}_i)^2 (1-\exp(-x q_n w^{(n)}_i))
	\le  \sigma_3 x q_n\sim \sigma_2^2 x.
\end{align*}
So for all $x$ and $n$ large enough (depending only on $\lambda$) by Bernstein's inequality \cite[Corollary 2.11]{BoLuMa2012a} (with $v=\sigma_2^2 x$ and $c=\max_i w^{(n)}_i$),
\begin{align*}
	\bP(Y^{(n)}_x > -\sigma_2) 
	& \le \bP(Y^{(n)}_x - \E[Y^{(n)}_x] > \sigma_2 x^2 /16)\\ 
	& \le \exp\left(- \frac{\sigma_2^2 x^4/16^2}{2\sigma_2^2 x + \max_i w^{(n)}_i x^2\sigma_2}\right) 
	\le e^{-x/2}\,,
\end{align*}
for all $x$ and $n$ large enough, using again \eqref{AldousAssumptions}. It follows that, for all $x$ large enough, and all $n$ large enough, $\bP(\weight{\C_1^{(n)}}>x) \le 2\exp(-x/2)$, which completes the proof.
\end{proof}

\begin{proof}[Proof of \cref{prop:L2ConvMultGraph}]
We shall now prove that \eqref{eq:towards_ui} holds for the sequence $(S_n)_{n\ge 1}$. We have
\begin{align*}
	\E[S_n^2] 
	& =  \E\Bigg[\sum_{i\ge 1} \weight{\C_i^{(n)}}^4\Bigg]+ \E\Bigg[\sum_{i\ne j} \weight{\C_i^{(n)}}^2 \weight{\C_j^{(n)}}^2 \Bigg]\\
	&\leq  \E[\weight{\C_1^{(n)}}^4]+ 2\E\Bigg[\sum_{i\ne j} \weight{\C_i^{(n)}}^2 \weight{\C_j^{(n)}}^2 \Bigg]\,,
\end{align*}
where the inequality follows by bounding the first sum using the fact that, for $i\geq 2$, we have $\weight{\C^{(n)}_i}^4\leq \weight{\C^{(n)}_{i-1}}^2 \weight{\C^{(n)}_{i}}^2$. Moreover we may upper bound the right-most expectation above using the BK inequality. Indeed, this expectation may be rewritten as 
\begin{align*}
	\bE\Bigg[\sum_{i\neq j} \weight{\C_i^{(n)}}^2 \weight{\C_j^{(n)}}^2 \Bigg]
	& =\sum_{a,b,c,d} w^{(n)}_a w^{(n)}_b w^{(n)}_c w^{(n)}_d \bP(a\lr b,c\lr d,a\nlr c)\\
	& \le  \sum_{a,b,c,d} w^{(n)}_a w^{(n)}_b w^{(n)}_c w^{(n)}_d \bP(a\lr b) \bP(c\lr d) = \bE[S_n]^2\,.
\end{align*}

Since $\sup_n \bE[\weight{\C_1^{(n)}}^4]<\infty$ by \cref{lem:multiplicative_graph_largest_cc}, we have $\sup_n \{\bE[S_n^2]-2 \bE[S_n]^2\} <\infty$. Therefore, recalling that $(S_n)_{n\ge 1}$ converges in distribution and is thus tight, \cref{lem:criterion_ui} applies, and the proof is complete. 
\end{proof}


\section{Convergence of the sprinkled component graph $G_\comp$} 
\label{sec:CompGraph}
For $m\in \N$, $\lambda\in \R$, let $p'_c(\lambda)\in(0,1)$ be the unique number satisfying
\begin{equation} 
\log\bigg(\frac{1-p'_c(\lambda)}{1-p_s}\bigg)=-\frac{q_\lambda}{mV^{1/3}}\,, \label{def:p'clambda} 
\end{equation}
where $q_\lambda:=V^{1/3}/\chi(p_s) + \lambda+o(1)$ is defined in \eqref{def:qlambda}. 

We define the component graph $G_\comp$ to be the graph with vertex set $\Cf_{p_s,M_s}$ (components of size at least $M_s$ in $H_{p_s}$), and such that for every $A\neq B\in \Cf_{p_s,M_s}$, the edge $(A,B)$ is in $G_\comp$ if and only if there exist vertices $a\in A$ and $b\in B$ such that the edge $\{a,b\}$ is open in $H_{p'_c(\lambda)}$.
It is easy to check that conditioned on $H_{p_s}$, independently for every distinct $A, B\in \Cf_{p_s,M_s}$ the edge $\{A,B\}$ lies in $G_\comp$ with probability 
\be p_{A,B}:=1-\left (\frac{1-p'_c(\lambda)}{1-p_s} \right)^{\Delta_{A,B}}=1-e^{-q_\lambda\Delta_{A,B}/(mV^{1/3})}, \label{def_p} \ee
where $\Delta_{A,B}$ denotes the number of hypercube edges with one endpoint in $A$ and the other in $B$. We write $\C^{\comp}_i$ for the $i$-th largest component of $G_\comp$. The goal of this section is to prove the following. 

\begin{proposition} \label{MultMasse2} For every $\lambda\in \R$, as $m\to \infty$, with respect to the $\ell^2$ topology,
\[V^{-2/3} (|\C^{\comp}_1|,|\C^{\comp}_2|,\ldots) \stackrel{\mathrm{(d)}}{\longrightarrow} \ZZ_\llambda.\]
\end{proposition} 

We remark that we also obtain convergence of the expectation of the $\ell^2$ norm to $\kappa(\lambda)$, see \cref{TotalMasses2}. Next, we let $d_\comp(\cdot,\cdot)$ denote the shortest path metric on $G_\comp$ and for every $i\geq 1$, let $\mu^\comp_i$ be the measure on $\C^\comp_i$ defined by $\mu^\comp_i(\{A\})=|A|/V^{2/3}$
for every $A\in \C^\comp_i$ and let $M^\comp_i$ be the mm-space
$$ M^\comp_i = (\C^\comp_i, \chi(p_s)V^{-1/3} \cdot d_\comp, \mu^\comp_i) \, .$$ 

\begin{proposition}\label{pro:conv_GP_comp}
For every $\lambda\in \R$, jointly with the convergence of $V^{-2/3}(|C_i^\comp|)_{i\ge 1}$ stated in Proposition~\ref{MultMasse2}, as $m\to \infty$ we have
\[ (M_1^\comp, M_2^\comp,\dots) \stackrel{\mathrm{(d)}}{\longrightarrow} \MM_\llambda \, ,\]
with respect to the product GP topology.
\end{proposition} 

These estimates are useful for the study of percolation on $H$ at $p'_c(\lambda)$ which is defined in \eqref{def:p'clambda} and is very explicit in terms of $V,m,\lambda, p_s, \chi(p_s)$; indeed in \cref{lem:estimate_pc'} we find an asymptotic expansion for $p_c'(\lambda)$. However we are interested in the point $p_c(\lambda)$ and it is not a priori clear that they are related. Our last result in this section, proved in \cref{sec:position}, is that the two points $p'_c(\lambda)$ and $p_c(\lambda)$ correspond to the same location in the scaling window. 

\begin{proposition} \label{Pro:PositionPc} For any choice of $o(1)$ in the definition of $q_\lambda$ at \eqref{def:qlambda} we have that $$\chi(p'_c(\lambda)) = (1+o(1)) \chi(p_c(\lambda)) \, .$$ Furthermore, we can choose the sequence $o(1)$ in \eqref{def:qlambda} so that $p'_c(\lambda)=p_c(\lambda)$. 
\end{proposition}

Let us also derive a corollary of \cref{Pro:PositionPc} that is interesting in its own right and that will also be useful for us in the next section. We first observe that by elementary calculus we can write a very sharp asymptotic expansion for $p'_c(\lambda)$ in which every term is explicit, except $\chi(p_s)$. 
\begin{lemma} \label{lem:estimate_pc'}
For every $\lambda\in \R$, as $m\to \infty$, we have 
\[ p'_c(\lambda) =p_s+ \frac 1 m \bigg(\frac {1-p_s}{\chi(p_s)}+\lambda V^{-1/3}+ o(V^{-1/3})\bigg) .\]
\end{lemma} 
\begin{proof}
The definition of $p'_c(\lambda)$ in \eqref{def:p'clambda} can be rewritten as,  
\[ p'_c(\lambda)=1-(1-p_s)\exp(-q_\lambda/(mV^{1/3}))\,, \]
where $q_\lambda=V^{1/3}/\chi(p_s)+\lambda+o(1)$ is given by \eqref{def:qlambda}. Recall that we defined $p_s$ in \eqref{def:ps} so that 
by \eqref{eq:subcritSucespt} we have $\X(p_s)\sim V^{1/3}\alpha_m$, and in particular $q_\lambda=o(mV^{1/3})$. Since $e^{-x}\sim 1-x+O(x^2)$ as $x\to 0$ we have
$$ p'_c(\lambda) = p_s + {q_\lambda \over mV^{1/3}} - {p_s q_\lambda \over mV^{1/3}} + O(p_s q_\lambda^2 / m^2 V^{2/3}) \, .$$
Now $q_\lambda/mV^{1/3} = m^{-1}(\chi(p_s)^{-1} + \lambda V^{-1/3} + o(V^{-1/3}))$ and $p_s = \Theta(m^{-1})$; plugging these into the last displayed equality concludes our proof.
\end{proof}

We may now state the aforementioned corollary which informally states that the scaling windows of the Erd\H{o}s--R\'enyi graph and of the hypercube are in fact asymptotically isometric.
\begin{corollary} \label{cor:Window} 
For any two real numbers $\lambda_2 > \lambda_1$ we have $p_c(\lambda_2)-p_c(\lambda_1)\sim (\lambda_2-\lambda_1)/(mV^{1/3})$  as $m\to \infty$.
\end{corollary} 
\begin{proof}
This is a direct consequence of \cref{Pro:PositionPc} and \cref{lem:estimate_pc'}.
\end{proof}

\subsection{Bounding the $\ell^2$ distance between the edge probabilities}
\label{sec:technical_estimate}

The goal of this section is to prove the following estimate about the proximity of the collections of (local) connection probabilities $(p_{A,B})$ from \eqref{def_p} and $(q_{A,B})$ from \eqref{def_q}, governing the true component graph and its multiplicative approximation, respectively.
\begin{proposition}\label{prop:DeltaBis} For every $\lambda\in \R$, we have
\[ \E_{p_s} \Bigg[ \sum_{A\neq B\in \Cf_{p_s,M_s}} \left ( p_{A,B} -q_{A,B} \right)^2 \Bigg]=O(\chi(p_s)^3/V ) . \]
\end{proposition}

Recall that $q_{\lambda}=V^{1/3}/\chi(p_s)+\lambda+o(1)$. By the definitions in \eqref{def:qab} and \eqref{def_q}, each term of the sum in left-hand side above is 
\[\big( p_{A,B} -q_{A,B} \big)^2 = \left (e^{-q_{\lambda} |A||B|V^{-4/3}}  -e^{-q_\lambda\Delta_{A,B}/(mV^{1/3})}\right)^2
\le \frac{q_\lambda^2}{V^{2/3}} \left (\frac{|A||B|}V  - \frac{\Delta_{A,B}}m\right)^2\,,
\]
since $x\mapsto e^{-x}$ is 1-Lipschitz on $[0,\infty)$. Observing that $q_{\lambda}\sim V^{1/3}/\chi(p_s)\sim \alpha_m^{-1/3}$ since $p_s=p_c(1-\alpha_m^{-1/3}V^{-1/3})$, proving \cref{prop:DeltaBis} boils down to showing the following.  
\begin{proposition}\label{prop:Delta} For every $\lambda\in \R$, we have
\[ \E_{p_s} \Bigg [ \sum_{A\neq B\in \Cf_{p_s,M_s}} \bigg ( { \Delta_{A,B} \over m\chi(p_s)} -{|A||B|\over {V\chi(p_s)}} \bigg )^2\Bigg ]=O\Big ( {\chi(p_s)^3 \over V} + \alpha_m \Big )\,. \]
\end{proposition}

The proof consists of the following three lemmas (Lemmas~\ref{lem:A2B2}, \ref{lem:ABDeltaAB} and \ref{lem:DeltaABSquared}), each estimating one of the three terms we obtain when expanding the square in the left-hand side. 
\begin{lemma} \label{lem:A2B2} 
We have 
\[ \E_{p_s}\Bigg[ \sum_{A\neq B\in \Cf_{p_s,M_s}} |A|^2 |B|^2 \Bigg]\leq V^2 \chi(p_s)^2. \]
\end{lemma} 
\begin{proof} 
Dropping the condition on the sizes of the connected components in $\Cf_{p_s,M_s}$, we obtain
\[\sum _{A \neq B \in \Cf_{p,M}}|A|^2 |B|^2 
\leq \sum_{u,v,x,y} {\bf 1} _{\{u \lr v, x\lr y, u\nlr x\}} 
\leq \sum_{u,v,x,y} {\bf 1} _{\{ u \lr v\} \circ \{x\lr y\}} \, ,\]
where the events involving connectivity all refer to the percolated hypercube at level $p_s$. Taking expectation and using the BK inequality yields the desired upper bound. 
\end{proof}

\begin{lemma} \label{lem:ABDeltaAB} We have 
\[  \E_{p_s} \Bigg [\sum_{A\neq B\in \Cf_{p_s,M_s}} |A| |B| \Delta_{A,B} \Bigg]\geq Vm \chi(p_s)^2 \Big (1-O\Big ( {\chi(p_s)^3 \over V} + \alpha_m \Big )\Big ) \, . \]
\end{lemma}
\begin{proof} The expected value in the left-hand side equals 
\be \label{eq:AfterApplyingOffMethod} 
\sum_{u\sim u',x,y} \proba_{p_s}(u \lr x, u' \lr y, u \nlr u',\,\,|\C(u)| \geq M_s, \,\,|\C(u')| \geq M_s) \, ,
\ee
where the range of summation is on vertices $u,u',x,y$ such that $u$ and $u'$ are neighbors in $H$. 

We would like to bound this from below using Aizenman's ``off'' method but for that we first want to remove the events $|\C(u)| \geq M_s$ and $|\C(u')| \geq M_s$. To this end, by the BKR inequality we have to bound
\[ \sum_{u\sim u',x,y} \proba_{p_s}(u \slr x, u' \slr y, u \snlr u',\,\,|\C(u)| \leq M_s) \leq \sum_{u\sim u',x,y} \proba_{p_s}(|\C(u)| \leq M_s, u\slr x )\proba_{p_s}(u'\slr y). \]
Then summing over all $y$ yields a term at most $\X(p_s)$, then over all $x$ yields a term $|\C(u)|$. Thus,
\begin{align*} \sum_{u\sim u',x,y} \proba_{p_s}(u \slr x, u' \slr y, u \snlr u',\,\,|\C(u)| \leq M_s) & \leq  \X(p_s) \sum_{u\sim u'}  \E_{p_s}[|\C(u)| \1_{|\C(u)|\leq M_s}]\,.
\end{align*}
By our choice of $p_s$ and $M_s$, and by \eqref{eq:ClusterTail}, the last expectation is $O(\sqrt{M})=O( \alpha_m^2 V^{1/3})=o(\alpha_m \X(p_s))$. Then summing over all $u\sim u'$ yields a term $Vm$, so the left hand-side above is $o(Vm\alpha_m \X(p_s)^2)$.

Hence, by the union bound, in order to prove the desired lower bound on \eqref{eq:AfterApplyingOffMethod}, it suffices to show 
\be \label{eq:AfterOff3} 
\sum_{u\sim u',x,y} \proba_{p_s}(u \lr x, u' \lr y, u \nlr u') 
\geq Vm \chi(p_s)^2 \Big (1-O\Big ( {\chi(p_s)^3 \over V} + \alpha_m \Big )\Big ) \, .\ee
To this end, we now conveniently apply  Aizenman's ``off'' method: 
By conditioning on $\C(u)$, we rewrite the sum in the left-hand side of \eqref{eq:AfterOff3} as 
\begin{equation}\label{eq:AfterOff3_sum}
    \sum _{u\sim u', x, y}\sum_{\substack{A:u,x\in A\\ u',y' \not \in A}} \proba_{p_s}(\C(u)=A) \proba_{p_s}(u' \lr y \mid \C(u) = A) \, .
\end{equation}
For the subsets $A$ for which $u,x\in A$ but $u' \not \in A$ we have
$$ \proba_{p_s}(u' \lr y \mid \C(u) = A) = \proba_{p_s}(u' \lr y \off A) \, .$$
Observe that if one of $y,u'$ lies in $A$, then $\proba_{p_s}(u' \lr y \off A)=0$ (also note that $\proba_{p_s}(u' \lr y \mid \C(u) = A)=1$ when $y\in A$). So in order to lower bound the sum in \eqref{eq:AfterOff3_sum}, we may drop the constraints that $u'\not\in A$ and $y'\not\in A$ and replace the conditioning on $\C(u)=A$ with ``off $A$''. We deduce from all this that we may bound from below the left-hand side of \eqref{eq:AfterOff3} by
$$ \sum _{u\sim u',x, y}\sum_{A: u,x\in A} \proba_{p_s}(\C(u)=A) \proba_{p_s}(u' \lr y \off A) \, .$$

Focusing on $\proba_{p_s}(u' \lr y \off A)$, we have
$$ \proba_{p_s}(u' \lr y \off A) \geq \proba(u' \lr y) - \proba_{p_s}(u' \lr y \onlyon A) \, ,$$
where we recall the event $\{u' \lr y \onlyon A\}$ means that $u' \lr y$ occurs but it does not occur off $A$. When this latter event occurs, there must exist a $z\in A$ such that $\{u' \lr z\}\circ \{z \lr y\}$. Hence, by the BK inequality, the sum in the left-hand side of \eqref{eq:AfterOff3} is at least 
\begin{equation}\label{eq:almost_there}
 \sum _{u\sim u',x, y} \proba_{p_s}(u \lr x)\proba_{p_s}(u' \lr y) - \sum _{u\sim u',x, y}\sum_{A: u,x\in A} \proba_{p_s}(\C(u)=A) \sum_{z\in A} \proba_{p_s}(u' \lr z) \proba_{p_s}(z \lr y) \, .
 \end{equation}
Summing first over $x,y$ and then over $u\sim u'$, we see that the first sum equals $mV \chi(p_s)^2$.

So we conclude by showing that the second sum is of smaller order. First, after changing the order of summation, we see that, for fixed $u,x,z$, we have 
\[\sum_{A: \\u,x,z\in A} \proba_{p_s}(\C(u)=A) = \proba_{p_s}(u \lr x, u \lr z)\,.\]
If $u \lr x$ and $u \lr z$ then there must exist a $w$ such that $\{u \lr w\} \circ \{w \lr x\} \circ \{w \lr z\}$, so that the BKR inequality implies 
that the second sum in \eqref{eq:almost_there} above is at most 
\begin{align*}
    &\sum _{\substack{u \sim u'\\ x, y,w,z }} \proba_{p_s}(u \lr w)\proba_{p_s}(w \lr x)\proba_{p_s}(w \lr z) \proba_{p_s}(u' \lr z) \proba_{p_s}(z \lr y) \\
    &\quad = \chi(p)^2 \sum _{\substack{u\sim u' \\ w,z }} \proba_{p_s}(u \lr w) \proba_{p_s}(w \lr z) \proba_{p_s}(z \lr u') \, ,
\end{align*}
by summing over $x$ and $y$. Finally, summing this triple term over $w,z$ is precisely the triangle diagram, so by \eqref{eq:TriangleCondition} the estimate \eqref{eq:AfterOff3} is proved and we conclude the proof.
\end{proof}

\begin{lemma}\label{lem:DeltaABSquared} We have 
\[ \E_{p_s} \Bigg[ \sum_{A\neq B\in \Cf_{p_s,M_s}} \Delta_{A,B}^2 \Bigg ]\leq m^2\chi(p_s)^2 (1-O(\alpha_m) ) . \]
\end{lemma}
\begin{proof}We have
\[ \E_{p_s} \Bigg [\sum_{A\neq B\in \Cf_{p_s,M_s}} \Delta_{A,B}^2 \Bigg ] = \sum_{u\sim u', v\sim v'} \proba_{p_s}(u \lr v, u' \lr v', u \nlr u', |\C(u)| \geq M_s, |\C(v)| \geq M_s ) \, , \]
where it is understood that the range of summation is over vertices $u,u',v,v'$ such that $u\sim u'$ and $v\sim v'$.
We split the above sum depending on whether or not the vertices are connected by a path of length at least $m_0$ where $m_0=T_{\rm mix}(\alpha_m)$ is defined in \cref{main_thm_general}. We first bound 
$$ \sum_{u \sim u', v \sim v'} \proba_{p_s}(u \stackrel{\geq m_0}{\lr} v, u' \lr v', u \nlr u', |\C(u)| \geq M_s, |\C(v)| \geq M_s ) \, ,$$
where by $u \stackrel{\geq m_0}{\lr} v$ we mean that there exists an open path of length at least $m_0$ connecting $u$ to $v$. 
The event above implies $\{u \stackrel{\geq m_0}{\lr} v\} \circ \{ u' \lr v' \}$ so by the BK inequality, the above sum is at most 
$$ \sum_{u \sim u', v \sim v'} \proba_{p_s}(u \stackrel{\geq m_0}{\lr} v) \proba_{p_s}(u' \lr v'). $$
We bound the first factor using \eqref{eq:UniformConnection} and pull the bound $(1+O(\alpha_m))\chi(p_s)V^{-1}$ out of the sum. We then sum over $v'$ to obtain a contribution of $\chi(p_s)$, and finally over $v$ and $u \sim u'$ to get another contribution of $Vm^2$; thus the last sum is bounded above by $(1+O(\alpha_m)) m^2 \chi(p_s)^2$. Hence it remains to show that
$$\sum_{u\sim u', v\sim v'} \proba_{p_s}(u \stackrel{m_0}{\lr} v, u' \lr v', u \nlr u', |\C(u)| \geq M_s, |\C(v)| \geq M_s ) = o(m^2 \chi(p_s)^2) \, .$$
We proceed as before, now spliting the sum according to the length of the shortest path between $u'$ and $v'$. Consider first the case when the shortest path between $u'$ and $v'$ has length at least $m_0$: using the BK inequality, \eqref{eq:UniformConnection} then \eqref{eq:IntVolume} we obtain
\begin{align*}
    &\sum_{u\sim u', v\sim v'} \proba_{p_s}(u \stackrel{m_0}{\lr} v, u' \stackrel{\geq m_0}{\lr} v', u \nlr u', |\C(u)| \geq M_s, |\C(v)| \geq M_s ) \nonumber \\ 
&\quad \leq C\chi(p_s) V^{-1} \sum_{u\sim u', v\sim v'}\proba_{p_s}(u \stackrel{m_0}{\lr} v) \leq C m_0 m^2 \chi(p_s) = o(m^2 \chi(p_s)^2) \, ,\nonumber 
\end{align*}
since $m_0 = o(\chi(p_s))$ by our assumptions in \cref{main_thm_general}.
Finally, it remains to show that
\be\label{eq:Delta2Goal} 
\sum_{u\sim u', v\sim v'} \proba_{p_s}(u \stackrel{m_0}{\lr} v, u' \stackrel{m_0}{\lr} v', u \nlr u', |\C(u)| \geq M_s, |\C(v)| \geq M_s ) = o(m^2 \chi(p_s)^2)\, .\ee
We first claim that on the event in the left-hand side of \eqref{eq:Delta2Goal}, there must exist vertices $z,z'$ such that 
$$ \{u  \stackrel{m_0}{\lr} z\} \circ \{z  \stackrel{m_0}{\lr} v\} \circ \{ |\C(z)|\geq M_s/m_0\} \circ  \{u'  \stackrel{m_0}{\lr} z'\} \circ \{z'  \stackrel{m_0}{\lr} v'\} \circ \{ |\C(z')|\geq M_s/m_0\} \, .$$
To see this, let $P$ be an open path of length at most $m_0$ between $u$ and $v$ and assume that $|\C(u)|\geq M$. Then there must exist a vertex $z\in P$ where $|\C(z)|\geq M/m_0$ in the graph where the edges $P$ are erased. 
So, from the BK inequality and \eqref{eq:ClusterTail}, the left-hand side of \eqref{eq:Delta2Goal} is at most 
$$ {C m_0 \over M_s} \sum_{\substack{z, u\sim u'\\z', v\sim v' }} \proba_{p_s}(u  \stackrel{m_0}{\lr} z) \proba_{p_s}(z  \stackrel{m_0}{\lr} v) \proba_{p_s}(u' \stackrel{m_0}{\lr} z') \proba_{p_s}(z'  \stackrel{m_0}{\lr} v') \, ,$$
which by symmetry equals 
\begin{equation}\label{eq:midstep} 
 {Cm_0 \over M_s}
\sum_{u\sim u', v\sim v'} \Bigg [ \sum_z \proba_{p_s}(u  \stackrel{m_0}{\lr} z) \proba_{p_s}(z  \stackrel{m_0}{\lr} v) \Bigg ] ^2 \, .
\end{equation}

The remainder of the proof relies on the non-backtracking random walk. Using \eqref{eq:PercolationNBWRW} we see that the expression in \eqref{eq:midstep} is bounded above by
\be\label{eq:midstep2} 
 {C m_0 \over M_s} \sum_{u\sim u', v\sim v'} \Bigg [ \sum_{t_1=1}^{m_0} \sum_{t_2=1}^{m_0} \sum_z \pp^{t_1}(u,z) \pp^{t_2}(z,v) \Bigg ]^2 \, .\ee
We use \eqref{eq:NBWconvolution} in \eqref{eq:midstep2}, and sum over $u',v'$ to obtain a factor of $m^2$. 
This yields an upper bound of
$$ {C m_0 m^2 \over M_s} \sum_{u,v} \sum_{t_1,t_2,t_3,t_4=1}^{m_0} \pp^{t_1+t_2}(u,v; t_1) \pp^{t_3+t_4}(u,v; t_3) \, ,$$
where we recall that, for $t<s$, $\pp^s(u,v; t)$ denotes the probability that a non-backtracking walk which is allowed to traverse back an edge at time $t$, starts at $u$ and lies at $v$ at time $s$ (as defined below \eqref{eq:NBWconvolution}). We proceed crudely and bound $\pp^{t_3+t_4}(u,v; t_3) \leq 1$ and $\sum_{v}\pp^{t_1+t_2}(u,v; t_1) =1$. Summing over $u$ gives a factor of $V$, and the sum over $t_1,t_2,t_3,t_4$ gives another factor of $m_0^4$. Hence \eqref{eq:midstep} is upper bounded by 
\[ {C m_0^5 m^2 V \over M_s} = o(m^2 \chi(p_s)^2) \, ,\] 
by our assumption in \cref{main_thm_general} that $m_0 = O(V^{1/15} \alpha_m)$. This concludes the proof.
\end{proof} 

\subsection{Counting the number of bad pairs of vertices}
\label{sec:counting_bad_pairs}

\newcommand{\mat}{T}
\newcommand{\disc}{\Xi}

We now show that thanks to \cref{prop:DeltaBis} we can couple $G_\times$ and $G_\comp$ so that they are close in a sense we define now. We use the standard simultaneous coupling between the two random graphs: let $\{U_{A,B}\}_{A,B\in \Cf_{p_s,M_s}}$ be i.i.d.\ random variables distributed uniformly on $[0,1]$ and put the edge $(A,B)$ in $G_\comp$ or $G_\times$ if and only if $U_{A,B}\leq p_{A,B}$ or $U_{A,B}\leq q_{A,B}$, respectively. 

The estimates in this section will be \emph{conditioned} on the subcritical graph $H_{p_s}$, i.e., in the quenched setup. To that aim we write $\proba^\star(\cdot )$ and $\E^{\star}(\cdot)$ for $\proba(\cdot |H_{p_s})$ and $\E(\cdot|H_{p_s})$, respectively. Also, for two sequences $(X_m)_{m\in \N}$, $(Y_m)_{m\in \N}$ of real random variables, we say that $X_m=O_{\proba}(Y_m)$ if $(X_m/Y_m)_{n\in \N}$ is tight, and that $X_m=o_{\proba}(Y_m)$ if $(X_m/Y_m)_{n\in \N}$ converges in probability to $0$.

For $A,B \in \Cf_{p_s,M_s}$ we write $\neq_{A,B}$ be the event that there exists a self-avoiding path between $A$ and $B$ that is present in one of $G_\times$ or $G_\comp$, but not in the other.

\begin{proposition} \label{TechnicalCompare} Under the coupling above
 \[ \E^\star\Big [ \sum_{A,B\in \Cf_{p_s,M_s} } |A| |B|\1_{\neq_{A,B}} \Big ]=o_{\proba}(V^{4/3}).\]
\end{proposition} 

It will be convenient to use the following four square matrices with zero diagonal whose rows and columns are indexed by $\Cf_{p_s,M_s}$. For any $A \neq B$ in $\Cf_{p_s,M_s}$ we define
\begin{table}[h]
\centering
\begin{tabular}{ll}
$\mat_\times(A,B) := \proba^{\star}(A\stackrel{G_\times}{\lr}B)$ \qquad \qquad & $\mat_\comp(A,B) := \proba^\star(A\stackrel{G_\comp}{\lr}B)$ \\
$\mat_{\comp \neq\times}(A,B) := \proba^\star(\neq_{A,B})$ \qquad \qquad & $\disc(A,B):=|q_{A,B}-p_{A,B}|$\, . \\
\end{tabular}
\end{table}

The last piece of notation we use concerns the Frobenius norm of a matrix $M=(m_{i,j})_{i,j\in I}$ which is denoted by $\|M\|:=(\sum_{i,j\in I} m_{i,j}^2)^{1/2}$. We begin by estimating the norm of $\mat_\times$; this estimate involves only the multiplicative graph (relevant notation, $w_A, q_\lambda, \sigma_2$ is defined Section~\ref{Sec:Mult}).

\begin{lemma}\label{sigma2 to sigma1} We have $ \|\mat_\times\|=O_\proba(V^{1/3}/\chi(p_s)).$
\end{lemma}
\begin{proof} Let $A\neq B\in \Cf_{p_s,M_s}$. If $A\stackrel{{G_\times}}{\lr} B$, then either the edge $(A,B)$ is open in $G_\times$ or there exists a $C\in \Cf_{p_s,M_s}$ such that the edges $(A,C)$ and $(C,B)$ are open or there exists $A',B'\in \Cf_{p_s,M_s}$, such that the edges $(A,A')$, $(B',B)$ are open and $A'\stackrel{{G_\times}}{\lr} B'$ using a path avoiding these two edges. Using the BK inequality we get 
\begin{align*}
\proba^\star(A\lr^{G_\times} B )
& \leq q_{A,B}+\sum_{C} q_{A,C} q_{C,B}+\sum_{A',B'} q_{A,A'} q_{B',B} \proba^\star(A'\lr^{G_\times} B') \\   
&\leq q_\lambda w_A w_B+\sum_{C} w_A w_B w_C^2 q_\lambda^2+\sum_{A',B'} q_{\lambda} w_A w_{A'} q_{\lambda} w_{B'} w_{B} \proba^\star(A'\lr^{G_\times} B')\,,
\end{align*}
where the second inequality follows from the fact that $q_{A,B}\leq q_\lambda w_A w_B$ for any $A\neq B\in  \Cf_{p_s,M_s}$. The right-hand side above factorizes, yielding
\begin{equation} 
\proba^\star(A\lr^{G_\times} B )\leq w_A w_B(q_\lambda+q_\lambda^2 \sigma_2+q_\lambda^2 S). \label{eq:ProbMult}
\end{equation}
where 
\[ S: =\sum_{A,B\in \Cf_{p_s,M_s}} w_Aw_B \proba^\star(A\lr^{G_\times} B) \, . \]
We square this, sum over $A,B$ and take a square root to obtain that 
$$ \|\mat_\times\| \leq q_\lambda\sigma_2 +q_\lambda^2 \sigma_2^2 +q_\lambda^2 \sigma_2 S \, .$$
Now \cref{prop:L2ConvMultGraph} which gives that $\E S = O(1)$ hence $S=O_\proba(1)$. Furthermore, $q_\lambda=O(V^{1/3}/\chi(p_s))$ by definition in \eqref{def:qlambda}, and \cref{lem:sigma2Concentration} gives that $\sigma_2=O_\proba(\chi(p_s)/V^{1/3})$, concluding our proof.
\end{proof}

Next we bound the norm of $\mat_{\comp\neq \times}$ by the norms of $T_\times, T_\comp$ and $\disc$.

\begin{lemma} \label{MassEstimates} 
We have 
$\| \mat_{\comp\neq \times} \| \leq \|\disc\|(1+\| \mat_\comp\|)(1+\| \mat_\times\|)$.
\end{lemma}
\begin{proof}  We claim that if the event $\neq_{A,B}$ occurs, then there exists an edge $(C,D)$ such that either $(C,D)$ is closed in $G_\times$ but open in $G_\comp$, or, open in $G_\times$ but closed in $G_\comp$ and 
$$ \{A\stackrel{G_\comp}{\lr}C\} \circ \{D\stackrel{G_\times}{\lr}B\} \, ,$$
occurs disjointly. Indeed, if $A$ and $B$ are connected in $G_\comp$ but not in $G_\times$, then we consider the path connecting them in $G_\comp$ and take $(C,D)$ to be the last edge on this path that is \emph{not} in $G_\times$. In the other case, if $A$ and $B$ are connected in $G_\times$ but not in $G_\comp$, then we consider the path connecting them in $G_\times$ and take the first edge in this path that is \emph{not} in $G_\comp$. 
Hence, the union bound and BK inequality give
\[ | \proba^\star(\neq_{A,B})| 
\leq \sum_{C,D  \in \Cf_{p_s,M_s}} \proba^\star(A\lr^{G_\comp} C)\cdot |q_{C,D}-p_{C,D}| \cdot \proba^\star(D\lr^{G_\times} B)\,,
\]
for any $A,B \in \Cf_{p_s,M_s}$. Note that the right-hand side is just the $(A,B)$ entry of the matrix product $(\Id+\mat_\comp)\disc (\Id+\mat_\times)$. Thus, the triangle inequality and the sub-multiplicativity of the Frobenius norm imply that 
\begin{align*} 
\| \mat_{\comp\neq \times}\| 
& \leq \|(\Id+\mat_\comp) \disc (\Id+\mat_\times) \|
\\ & \leq \|\disc\|+\|\mat_\comp\disc \| + \|\disc \mat_\times \|+\| \mat_\comp\disc\mat_\times\|
\\ & \leq \|\disc\|(1+\|\mat_\comp\|)(1+\|\mat_\times\|).\qedhere
\end{align*}
\end{proof}

This allows us to bound $\| \mat_{\comp}\|$. 
\begin{lemma}\label{lem:TcompNorm} We have $\|\mat_\comp\|=O_\proba(V^{1/3}/\chi(p_s))$.
\end{lemma}
\begin{proof} By the triangle inequality we have
\begin{align*} 
\big|\|\mat_{\comp}\|-\| \mat_\times\|\big|^2 
\leq \|\mat_{\comp}-\mat_\times \|^2 
& \leq  \sum_{A,B\in \Cf_{p_s,M_s}}  |\proba^\star(A\lr^{G_\times}B)-\proba^\star(A\lr^{G_\comp}B) |^2  \notag
\\ & \leq \sum_{A,B\in \Cf_{p_s,M_s}} \proba^\star(\neq_{A,B})^2 
\leq \| \mat_{\comp\neq \times} \|^2. 
\label{TooSmallToBeAFish}
\end{align*}

Hence by Lemma~\ref{MassEstimates}, 
$$ \big| \|\mat_\comp\|-\|\mat_\times  \|\big| \leq \|\disc\|(1+\| \mat_\comp\|)(1+\| \mat_\times\|) \, .$$
Lemma~\ref{prop:DeltaBis} implies that $\|\disc\|^2=O_\proba(\chi(p_s)^3/V)$ so $\|\disc\|=o_\proba(\chi(p_s)/V^{1/3})$ since $\chi(p_s)=o(V^{1/3})$. Together with \cref{sigma2 to sigma1}, this implies that $\|\disc\|(1+\| \mat_\times\|) = o_\proba(1)$, hence
$$ \big| \|\mat_\comp\|-\|\mat_\times  \|\big| = o_\proba(1+\| \mat_\comp\|) \, ,$$
which by the triangle inequality gives the desired result. \end{proof}

\begin{proof}[Proof of Proposition \ref{TechnicalCompare}]
We first note that by Cauchy--Schwarz's inequality,
\begin{align*}
\E^\star\Bigg [ \sum_{A,B\in \Cf_{p_s,M_s} } |A| |B|\1_{\neq_{A,B}} \Bigg ]^2 
&\leq \sum_{A,B\in \Cf_{p_s,M_s}} |A|^2 |B|^2 \times\sum_{A,B\in \Cf_{p_s,M_s}} \proba^\star(\neq_{A,B} )^2\,.
\end{align*}
Lemma~\ref{lem:sigma2Concentration} shows that the first factor is $O_\proba(V^2\chi(p_s)^2)$ and the second is just $\|\mat_{\comp\neq \times} \|^2$. We bound the latter using \cref{MassEstimates} together with \cref{sigma2 to sigma1}, \cref{lem:TcompNorm} and Lemma~\ref{prop:DeltaBis}, yielding a bound $\|\mat_{\comp\neq \times} \|=o_{\proba}(V^{1/3}/\chi(p_s))$. Putting all these together gives
\begin{align*}
\E^\star\Bigg [ \sum_{A,B\in \Cf_{p_s,M_s} } |A| |B|\1_{\neq_{A,B}} \Bigg ]^2 = o_\proba(V^{7/3} \chi(p_s)) = o_\proba(V^{8/3}) \, ,
\end{align*}
since $\chi(p_s) = o(V^{1/3})$, concluding our proof. 
\end{proof} 
\subsection{Convergence of the sprinkled component graph} 
\label{Sec:CompTopo}

 \cref{TechnicalCompare} is precisely what is necessary to transfer the known asymptotic properties from 
 $(\C_i^\times)_{i\ge 1}$ to $(\C_i^\comp)_{i\ge 1}$. We start with the convergence of the sizes. 
\begin{proof}[Proof of \cref{MultMasse2}]
By \cref{pro:MCGMassConverges} we have
\[V^{-2/3} (|\C^{\times}_1|,|\C^{\times}_2|,\ldots) \stackrel{\mathrm{(d)}}{\longrightarrow} \ZZ_\llambda \, ,\]
where $\ZZ_\lambda=(|\gamma_i|)_{i\ge 1}$ and $(\gamma_i)_{i\ge 1}$ are the excursions of $(W^\lambda_t)_{t\ge 0}$ above its running infimum (see \cref{sec:intro}). Next \cref{TechnicalCompare} together with Markov's inequality implies that 
$$ V^{-4/3}\sum_{A,B} |A||B| \1_{\ne_{A,B}} \stackrel{\mathrm{(d)}}{\longrightarrow} 0 \, .$$

By Skorohod's representation theorem, we may assume without loss of generality that the convergences above both occur almost surely. Now for any fixed integer $k>0$ and $\e >0$ we denote by $\Omega_\e^k$ the event 
\be\label{eq:omegaEK} \Omega_\e^k = \Big \{ |\gamma_i| \geq |\gamma_{i+1}|+\e \quad\forall \,\, i=1,\ldots, k-1, \mathrm{\ and\ } |\gamma_k|\geq \e \Big \} \, .\ee
Since $(|\gamma_1|,\dots, |\gamma_k|)$ is absolutely continuous on $\R_+^k$ we have that $\proba(\Omega_\e^k) \to 1$ as $\e \to 0$ and $k$ fixed. Thus, on $\Omega_\epsilon^k$ for all $i\in [k-1]$ and $m$ large enough we have
\be \label{eq:MultClustersAreSeparated} |\C_i^\times| \geq |C_{i+1}^\times| + \e V^{2/3}/2 \quad \mathrm{and} \quad  |\C_k^\times|\ge \e V^{2/3}/2 \, .\ee

Assume now that $\Omega_\e^k$ holds and that $\e>0$ is arbitrarily small but fixed. We claim that for any $i=1,\ldots k$ there exists a unique $j \geq 1$ (later we will prove $j=i$) such that 
\be\label{eq:matchCiCj} |\C_i^\times \cap \C_j^\comp| \ge \big (1-{\e\over 8}\big )|\C_i^\times| \qquad \mathrm{and} \qquad |\C_j^\comp\setminus \C_i^\times|\le {\e \over 8} |\C_i^\times| \, .\ee
To show existence of such $j$ we set $x_{i,j}=|\C_i^{\times}\cap \C_j^\comp|/|\C^{\times}_i|$ and observe that
\begin{align*}
	\sum_{A,B} |A| |B| \1_{\neq_{A,B}} 
	& \ge \sum_{A,B\in \C_i^{\times}} |A| |B| \1_{A\stackrel{G_\comp}{\hspace{.15cm} \not \hspace{-.15cm} \lr} B} 
	= \sum_{j\ge 1} \sum_{A\in C^{\times}_i \cap \C^\comp_j} |A| \cdot \sum_{B\in \C^{\times}_i \setminus \C^\comp_j} |B|\\
	& = \sum_{j\ge 1} \sum_{A\in C^{\times}_i \cap \C^\comp_j} |A| \cdot |\C^{\times}_i|(1-x_{i,j}) \ge |\C^{\times}_i|^2 \cdot \Big(1-\max_j x_{i,j}\Big)\, .
\end{align*}
Since $|\C^{\times}_i|\geq \eps V^{2/3}/2$ and the left-hand side is $o(V^{4/3})$ we get that $\max_j x_{i,j}\to 1$ for any $i=1,\ldots k$ and so there exists $j$ such that the left-hand side of \eqref{eq:matchCiCj} holds. For this $j$ the right-hand side of \eqref{eq:matchCiCj} must hold as well, since otherwise $|\C_j^\comp \setminus \C_i^\times|\geq \e |\C_i^\times|/8$ implying that
$$ 	\sum_{A,B} |A| |B| \1_{\neq_{A,B}} \geq \sum_{A \in \C_i^\times, B \in \C_j^\comp \setminus \C_i^\times} |A||B| \geq \eps |\C_i^\times|^2 \geq \eps^3 V^{4/3}/64 \, ,$$
contradicting the fact that the left-hand side is $o(V^{4/3})$. This $j$ is unique, since if there were two distinct $j$'s satisfying \eqref{eq:matchCiCj}, then the corresponding components in $G_\comp$ would intersect. We denote this unique $j$ by $\pi(i)$. Note that similarly $\pi$ is injective: if there were two $i$'s corresponding to the same $j$ in \eqref{eq:matchCiCj}, then the corresponding components in $G_\times$ would intersect. 

We also deduce that there is some $j\in \{k,k+1,\ldots\}$ for which $|\C_j^\comp|\geq (1-\e) |\C_k^\times| \geq {\e \over 4}V^{2/3}$ (by \eqref{eq:MultClustersAreSeparated}) when $\e \leq 1/2$, hence $|\C_k^\comp|\geq {\e \over 4}V^{2/3}$. This allows us to repeat the same argument as above with the components of $G_\comp$ rather than $G_\times$, and obtain that for any $j\in \{1,\ldots, k\}$ there exists $i \geq 1$ such that
\be\label{eq:matchCiCj2} |\C_j^\comp \cap \C_i^\times| \ge \big (1-{\e \over 8} \big)|\C_j^\comp| \qquad \mathrm{and} \qquad |\C_i^\times\setminus \C_j^\comp|\le {\e\over 8} |\C_j^\comp| \, ,\ee
and similarly this $i$ is unique. We denote this unique $i$ by $\sigma(j)$ and note that again $\sigma$ is injective.

Now, if $\pi(i)=j$ or $\sigma(j)=i$, then $|\C_i^{\times}|/|C_j^\comp|\in [1-\e/4,1+\e/4]$. This and the fact that the sizes of $\C_1^\times, \ldots, \C_k^\times$ are separated by at least $\e V^{2/3}/2$ by \eqref{eq:MultClustersAreSeparated} imply that $\pi(1)=1$; indeed, otherwise $\pi(1)>1$ and we get that for all $j=1,\ldots, \pi(1)$ we have $|\C_j^\comp| \geq |\C_1^\times| - \e V^{2/3}/4$ and so both $j=1,2$ must be matched to $i=1$ by $\sigma$, contradicting the fact that $\sigma$ is injective. We deduce that $\pi(1)=1$ and $\sigma(1)=1$. By induction it folllows that $\pi(i)=i$ for all $i\in [k]$.

Recalling that $\Omega^k_\epsilon$ occurs with arbitrary high probability by choice of $\epsilon$, this also shows that, for every natural number $k$, 
\be\label{eq:convInProb} V^{-4/3} \sum_{i=1}^k \Big||\C^\times_i|-|\C^\comp_i|\Big| \to 0 \quad \text{in probability.}\ee

It remains only to prove the tightness in $\ell^2$ of $\sum_{i \geq 1}|\C_i^\comp|^2$. The triangle inequality implies that, for any $\epsilon>0$ and for any $k\ge 1$, if $\sum_{i>k}|\C_i^\comp|^2>\epsilon V^{4/3}$, then one of the next events must occur: either
\begin{align*}
	\sum_{i>k} |\C_i^\times|^2> \frac \epsilon 3 V^{4/3}, \quad\text{or}\quad
	\bigg|\sum_{i=1}^k (|\C_i^\times|^2 - |\C_i^\comp|^2)\bigg|>\frac \epsilon 3 V^{4/3}, \quad\text{or}\quad 
	\bigg|\sum_{i\ge 1} (|\C_i^\times|^2 - |\C_i^\comp|^2)\bigg|>\frac \epsilon 3 V^{4/3}\,.
\end{align*}
However, the convergence of $(V^{-2/3}|\C_i^\times|)_{i\ge 1}$ in $\ell^2$ implies that for any $\eta>0$, we can choose $k$ large enough that 
\[\limsup_{m}\proba\bigg(\sum_{i>k} |\C_i^\times|^2 > \frac\epsilon3 V^{4/3}\bigg) \le \eta/3\,.\]
This value of $k$ being fixed, the fact that the probability of the second event is no more than $\eta/3$ for all $m$ large enough is a  consequence of \eqref{eq:convInProb}. Finally, for the third event we have 
\begin{align*}
	\bigg|\sum_{i\ge 1} (|\C^{\comp}_i|^2 - |\C^\times_i|^2) \bigg|
	= \bigg|\sum_{A\ne B} |A| |B| (\1_{A \stackrel{G_\comp}{\lr} B}-\1_{A \stackrel{G_\times}{\lr} B})\bigg|
	\le \sum_{A\ne B} |A| |B| \1_{\ne_{A,B}} = o(V^{4/3}) \, ,
\end{align*}
This concludes the proof of \cref{MultMasse2}.
\end{proof}

We may now use \cref{pro:MCGDistConverges} and \cref{TechnicalCompare} to deal with the asymptotics for the distances in the component graph. 
\begin{proof}[Proof of \cref{pro:conv_GP_comp}]
We work on a probability space on which we have the almost sure convergence of $V^{-2/3}(|\C_i^\comp|)_{i\ge 1}$ and $V^{-2/3} (|\C_i^\times|)_{i\ge 1}$. Let $d_i^\comp$ and $d_i^\times$ denote respectively the metrics in $\C_i^\comp$ and $\C_i^\times$. 

By \cref{pro:MCGDistConverges}, $(M_i^\times)_{i\ge 1}\to \MM_\llambda$ for the  product-GP topology. Therefore, in order to prove the joint convergence of the collection of metric spaces $(M_i^{\comp})_{i\ge 1}$ with respect to the product GP-distance, it suffices to prove that for every $i\ge 1$ and every $\ell\ge 1$, there exists a coupling of random points $(\xi^\times_j)_{j=1}^\ell$ and $(\xi^{\comp}_j)_{j=1}^\ell$ which are respectively i.i.d.\ with distribution proportional to $\sum_{A\in \C_i^\times} |A|\delta_A$ and $\sum_{A\in \C_i^\comp} |A| \delta_A$, such that, in probability,  
\begin{equation}\label{eq:discr_distance_matrices}
\max_{1\le j,k\le \ell}\big\{|d_i^\times(\xi^\times_j, \xi^\times_k)-d_i^{\comp}(\xi^\comp_j, \xi^\comp_k)|\big\} \to 0\,.
\end{equation}
Indeed, if \eqref{eq:discr_distance_matrices} holds, then for every finite subset $S\subseteq \N$, the union bound implies that the convergence also holds if we further take the maximum over the indices $i\in S$. We are then left with the proof of \eqref{eq:discr_distance_matrices} for a single fixed $i\in \N$.

We start by coupling the random points. Recall the event $\Omega_\epsilon^k$ from \eqref{eq:omegaEK} in the proof of \cref{MultMasse2}. Fix a natural number $k\ge 1$. Let $\eta>0$ be arbitrary. Choose $\epsilon>0$ such that $\proba(\Omega_\epsilon^k)\ge 1-\eta/5$.  Then, let $m'$ be large enough that first, on $\Omega_\epsilon^k$, we have $|\C_i^\times|\ge \epsilon/2$ for all $m\ge m'$, and second 
\[ \proba\bigg(\min\bigg\{\frac{|\C_i^\times \cap \C_i^\comp|}{|\C_i^\times|}; \frac{|\C_i^\times \cap \C_i^\comp|}{|\C_i^\comp|}\bigg\} \ge 1-\frac \eta{5\ell}~\bigg|~\Omega_\epsilon^k\bigg) < \eta/5\,.\]
Then, for all $m\ge m'$, we may ensure that $\xi_j^\times=\xi^\comp_j$ for all $j\in [\ell]$ with probability at least $1-3\eta/5$. Observe also, that on the same event we also have $|\C_i^\times \cap \C_i^\comp| \ge \epsilon V^{2/3}/3$. 

When we do have perfect coupling we write $\xi_j$ for the common value of $\xi_j^\comp=\xi_j^\times$; then the distances $d_i^\times(\xi_j,\xi_k)$ and $d_i^\comp(\xi_j,\xi_k)$ may only differ on the condition that there is a self-avoiding path between $\xi_j$ and $\xi_k$ that exists in one of $G_\times$ or $G_\comp$, but not the other; this is precisely the event $\ne_{\xi_j,\xi_k}$. Furthermore, for a given pair $j,k\in [\ell]$, the conditional probability that the event $\ne_{\xi_j,\xi_k}$ occurs is 
\[\sum_{A,B\in \C_i^\times \cap \C_i^\comp} \frac{|A||B| \1_{\ne_{A,B}}}{|\C_i^\times \cap \C_i^\comp|^2}\,.\]
Choose now $\delta>0$ small enough that $9\ell^2 \delta/\epsilon^2 < \eta/5$, and finally, using \cref{TechnicalCompare} and Markov's inequality, $m''\ge m'$ large enough that 
\[\proba\bigg(\sum_{A,B} |A||B| \1_{\ne_{A,B}}> \delta V^{4/3}\bigg) < \eta/5\,.\]
Then, for $m\ge m''$, the union bound implies that 
\[\proba\Big(\max_{1\le j,k\le \ell}\Big\{\big|d_i^\times(\xi_j^\times, \xi_k^\times)-d_i^\comp(\xi^\comp_j, \xi_k^\comp)\big|\Big\}>0\Big) \le \eta\,.\]
Since $\eta$ was arbitrary, this proves the claimed convergence in \eqref{eq:discr_distance_matrices} and completes the proof.
\end{proof} 
\subsection{Position in the critical window}\label{sec:position}

The main goal of this section is to prove \cref{Pro:PositionPc}. Our proof is based on a comparison between the susceptibility at $p'_c$ and in the sprinkled component graph which we may rewrite respectively as

\[ V\chi(p'_c(\lambda))=\sum_{u,v\in\{0,1\}^m} \proba_{p'_c(\lambda)} \left (u \lr v \right ) ,\]
and, writing $\V_\star:=\bigcup_{A \in \comp_{p_s,M_s}} A$,  as 
\[ \E\Bigg[\sum_{i\geq 1} |\C^{\comp}_i|^2 \Bigg]=\sum_{u,v\in\{0,1\}^m} \proba_{p'_c(\lambda)} \left (u \stackrel{\V_\star}{\lr} v \right ).\]

It is thus natural to define, for any $p\in [0,1]$, the random variable $N(p)$ counting the number of (ordered) pairs of vertices connected in $H_p$, and such that every path between them goes through a connected component in $H_{p_s}$ of size \emph{less than} $M_s$. In particular, $N(p)$ includes the count of pairs of vertices of the connected components of size less than $M_s$ in $H_{p_s}$. Observe that $N(p)$ is a random variable measurable with respect to the simultaneous coupling described in \cref{sec:counting_bad_pairs}. 

We proceed by showing that $\E N(p)=o(V^{4/3})$ whenever $p$ is within (or below) the scaling window, then conclude by showing that $p'_c(\lambda)$ is indeed inside the scaling window. 

\begin{lemma} \label{lem:BadPair} For any $\Lambda\in \R$ we have $\E[N(p_c(\Lambda))]=o(V^{4/3})$ as $m\to \infty$.
\end{lemma}
\begin{proof} For this proof we write $p_c=p_c(\Lambda)$ for convenience. 
We bound $N(p_c)$ above by the number of pairs of vertices $u,v$ that are connected by a path $\gamma$ in $H_{p_c}$ that goes through a connected component $A$ in $H_{p_s}$ of size $|A|< M_s$.  Given two vertices $u,v$ our analysis depends on whether $u\in A$ and $v\in A$ or not. Let us consider first the case that $u \not \in A$ and $v\not \in A$. In this case, the path $\gamma$ must contain the first edge $(x',x)$ entering $A$ so that $x' \not \in A$ and $x\in A$ and last edge $(y,y')$ leaving $A$ so that $y\not \in A$. These imply that the following events occur disjointly: (see Figure \ref{Fig:4.15})
\begin{compactitem}
\item[(a)] $u$ and $x'$ are connected in $H_{p_c}$;
\item[(b)] $(x,x')$ and $(y,y')$ are closed in $H_{p_s}$ but open in $H_{p_c}$. Furthermore $x,y$ both lie in a common connected component of $H_{p_s}$ of size less than $M_s$, namely $A$. (This event is determined by the status of all the edges with at least one endpoint in $A$.)
\item[(c)] $y'$ and $v$ are connected in $H_{p_c}$.
\end{compactitem}
 \begin{figure}[t]
\centering
\includegraphics[scale=1.5]{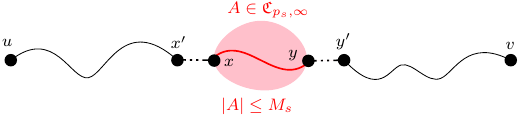}
\caption{\label{Fig:4.15}A representation  from left to right  of the events (a), (b), (c) of Lemma 4.15. There is a path between $u,v$ in $H_{p_c}$ crossing a small connected component $A$ of $H_{p_s}$. 
} 
\end{figure}%
By the union bound and the BKR inequality we bound the number of such $u,v$ by
\[ \sum_{u,(x,x'),(y,y'),v} \proba_{p_c}(u\lr x')
\proba\left ((x,x'),(y,y')\text{ satisfy (b)}\right) 
\proba_{p_c} (y'\lr v ). \]
Since $H$ is transitive, we may sum over all $u$ and over all $v$ so the last quantity equals
\[ \X(p_c)^2\sum_{(x,x'),(y,y')} \proba\left ((x,x'),(y,y')\text{ satisfy (b)} \right). \]
Next, given $H_{p_s}$, for every $x,y$ there exist at most $m$ edges $(x,x')$, $(y,y')$ that are closed in $H_{p_s}$, and each one is independently open in $H_{p_c}$ with probability $(p_c-p_s)/(1-p_s)$. Thus we may bound the last term above by
\[ \X(p_c)^2 m^2 \left (\frac{p_c-p_s}{1-p_s} \right )^2 \sum_{x,y} \proba\left (x,y \text{ lie in a connected component\ of $H_{p_s}$ of size at most $M_s$} \right), \]
which we may rewrite as 
\[ \X(p_c)^2 m^2 \left (\frac{p_c-p_s}{1-p_s} \right )^2 V \E_{p_s}[|\C|\1_{|\C|\leq M_s}]. \]
We now apply \eqref{eq:ClusterTail} giving that $\E_{p_s}[|\C|\1_{|\C|\leq M_s}]= O(\sqrt{M_s})$.  
Therefore, since $p_s:=p_c\cdot (1-V^{-1/3}\alpha_m^{-1/3})$, and $p_c\sim 1/m$, we have $(p_c-p_s)/(1-p_s) \sim  V^{-1/3}\alpha_m^{-1/3}m^{-1}$. Together with our choice $M_s=V^{2/3}\alpha_m^4$, we upper bound this sum by
\begin{align*}
 O(V^{2/3}m^2 \cdot V^{-2/3}\alpha_m^{-2/3} m^{-2} \cdot V \cdot V^{1/3}\alpha_m^{2} ) 
=O(V^{4/3} \alpha_m^{4/3}) = o(V^{4/3}) \, . 
\end{align*}
The other cases are easier and follow a similar reasoning which we briefly describe. If $u\in A$ and $v\in A$, then $v\in \C(u)$ and $|\C(u|\leq M_s$, summing over $u,v$ gives a contribution of $V \E_{p_s}[|\C|\1_{|\C|\leq M_s}]$ which is $o(V^{4/3})$. Lastly, if $u\not \in A$ but $v \in A$, a path connecting $u$ to $v$ has a first entry to $A$ edge $(x',x)$, using the same analysis as before and using the BKR inequality gives a contribution of at most 
$$ \chi(p_c) Vm (p_c-p_s) \E_{p_s}[|\C|\1_{|\C|\leq M_s}] = O(V^{4/3} m  V^{-1/3} m^{-1}\alpha_m^{-2/3} \sqrt{M_s}) \, ,$$
which again is $o(V^{4/3})$ concluding the proof.
\end{proof}

We can now prove that $p'_c(\lambda)$ lies within the critical window. 
\begin{lemma} \label{NotSupercritical} For every $\lambda\in \R$, there exists $\Lambda\in \R$, such that for every $m$ large enough, 
\[ p'_c(\lambda)\leq p_c(\Lambda). \] 
\end{lemma}
\begin{proof} Write $\ZZ_\llambda = (|\gamma^\llambda_1|,|\gamma^\llambda_2|,\ldots)$. By \cref{MultMasse2}, the sum $V^{-4/3} \sum_{i\geq 1} |\C_i^\comp|^2$ converges in distribution to $\sum_{i\geq 1} |\gamma_i^\lambda|^2$. Notably there exists $M>0$ such that for every $m$ large enough
\[ \proba_{p'_c(\lambda)}\Bigg(\sum_{i\geq 1} |\C_i^\comp|^2\geq MV^{4/3} \Bigg)\leq 1/3, \]
which we may rewrite as 
\begin{equation} \proba_{p'_c(\lambda)}\Bigg(\sum_{i\geq 1} |\C_i|^2 - N(p'_c(\lambda))\geq MV^{4/3} \Bigg )\leq 1/3. \label{19/10/0ha} \end{equation}

On the other hand, by \cite[Theorem 1.3 (b)]{HN17} if $\Lambda>0$ is large enough, for every $m$ large enough,
\[ \proba_{p_c(\Lambda)}\Bigg (\sum_{i\geq 1} |\C_i |^2\geq 2MV^{4/3} \Bigg )\geq 2/3. \]
So by Lemma \ref{lem:BadPair} and Markov's inequality, if $\Lambda>0$ is fixed large enough, then for every $m$ large enough
\begin{equation} \proba_{p_c(\Lambda)}\Bigg (\sum_{i\geq 1} |\C_i|^2-N(p_c(\Lambda))\geq MV^{4/3} \Bigg )\geq 1/2. \label{19/10/0hb}  \end{equation}
To conclude the proof, since the map $p\mapsto \sum_{i\geq 1} |\C_i|^2-N(p)$ is increasing, comparing \eqref{19/10/0ha} and \eqref{19/10/0hb} yields the desired result. 
\end{proof}

The last lemma was necessary to deduce the next result from Proposition \ref{MultMasse2}. 
\begin{lemma} \label{TotalMasses2} For every $\lambda\in \R$, as $m\to \infty$, we have 
$\E[\sum_{i\geq 1} |\C^{\comp}_i|^2]\sim V^{4/3} \kappa(\lambda). $
\end{lemma}
\begin{proof} We adapt the proof that we have already used in Section~\ref{Sec:MultX}. 
Since the argument does not change, we shall be faster. 
As we already have the weak limit in \cref{MultMasse2} it is enough to show that $\sum_{i\geq 1} |\C^{\comp}_i|^2$ is uniformly integrable. To do so, it suffices to show that as $m\to \infty$, 
\[ \E\Bigg [\bigg(\sum_{i\geq 1} |\C^{\comp}_i|^2\bigg )^2 \Bigg ] \leq O(1)+ 2\E\Bigg [\sum_{i\geq 1} |\C^{\comp}_i|^2\Bigg ]^2.\]

And to this end, again by using the BK inequality, it is enough to prove that $\E[ |\C^{\comp}_1|^4 ]=O(1)$. Observing that each connected component of $G^{\comp}$ is a subset of a connected component of the percolated hypercube $H_{p'_c(\lambda)}$, it suffices to show $\E_{p'_c(\lambda)}[ |\C_1|^4 ]=O(V^{8/3})$. 

By the tree-graph inequality (see (6.94) in \cite{Grimmett}) we have $\E_{p'_c(\lambda)}[ |\C_1|^4 ]=O(V\X(p'_c(\lambda)^5))$ thus the desired result follows from Lemma~\ref{NotSupercritical}.
\end{proof}

We now have all the key elements to prove Proposition~\ref{Pro:PositionPc}.
\begin{proof}[Proof of Proposition \ref{Pro:PositionPc}.] We first take $(q_\lambda^m)_{m\in \N}$ any sequence such that \eqref{def:qlambda} holds. By Lemma \ref{TotalMasses2} together with \cref{lem:BadPair} and \cref{NotSupercritical}, for every $\lambda\in \R$, as $m\to \infty$, we have
\[ \chi(p_c'(\lambda)) \sim V^{1/3} \kappa(\lambda).\]
On the other hand, by definition \eqref{def:pc} for every $\lambda\in \R$
\[ \chi(p_c(\lambda))= V^{1/3} \kappa(\lambda).\]
Since $\kappa(\cdot)$ is strictly increasing (see \cite[Corollary 24]{Aldous97}), it follows that for every $\lambda_1<\lambda_2$ as long as $m$ is large enough, 
\[ p'_c(\lambda_1)<p_c(\lambda_2) \quad \text{and} \quad  p_c(\lambda_1)<p'_c(\lambda_2). \]
Sandwiching $p_c(\lambda)$ between $p'_c(\lambda-\epsilon)$ and $p'_c(\lambda+\epsilon)$, and taking the limit as $\epsilon\to 0$, we apply Lemma~\ref{lem:estimate_pc'} to obtain
\[ p_c(\lambda)=p_s+\frac 1 m \cdot \bigg(\frac {1-p_s} {\chi(p_s)}+\lambda V^{-1/3}+o(V^{-1/3})\bigg)\,,\]
so that $p_c(\lambda)$ and $p_c'(\lambda)$ have the same asymptotic behavior. Finally, by elementary calculus as in the proof of Corollary~\ref{cor:Window}, we obtain
\[ - mV^{1/3}\log \bigg ( \frac{1-p_c(\lambda)}{1-p_s}\bigg)= \frac{V^{1/3}}{\chi(p_s)}+\lambda+o(1)\,.\]
It follows that, if we chose $q_\lambda$ as the left-hand side above to get $p_c(\lambda)=p'_c(\lambda)$ by \eqref{def:p'clambda}, then \eqref{def:qlambda} is still satisfied, so that this choice is valid. This concludes the proof.
\end{proof}

\subsection{Proof of Theorem \ref{main_thm}} \label{sec:mainthmPart1}
The proof is similar to the proof of \cref{MultMasse2} and we provide it here briefly for completeness. By \cref{Pro:PositionPc} we may assume that $p_c(\lambda)=p'_c(\lambda)$. We work on the simultaneous coupling that allows us to consider $H_{p_c}$ and $H_{p_s}$ under the same probability space. Recall (\cref{sec:position}) that $N=N(p_c)$ denotes the number of pairs of vertices $\{u,v\}$ of $H$ that are connected in $H_{p_c}$ and such that any path between them visits a component of $H_{p_s}$ of size less than $M_s$.

As usual we write $(\C_i)_{i\geq 1}$ for the connected components of $H_{p_c}$ in decreasing order of their sizes and $(\C^\comp_j)_{j\geq 1}$ for the components of $G_\comp$. We abuse notation and write $\C_i^\comp$ for $\cup_{A \in \C_i^\comp} A$, that is, each component of $G_\comp$ will be considered here as a subset of vertices of $H$. In particular, since $p_c(\lambda)=p'_c(\lambda)$, for every $j \geq 1$ there exists a unique $I_j$ such that $\C^\comp_j \subset \C_{I_j}$. 

Then by \cref{lem:BadPair}, $N(p_c)=o_\proba(V^{4/3})$, and we may rewrite $N(p_c)$ as 
\[ N(p_c) 
= \sum_{i \in \N}  \bigg( |\C_i|^2-\sum_{j\in \N:I_j=i} |\C^\comp_j|^2 \bigg)
\leq \sum_{i \in \N} |\C_i| \left (|\C_i|-\max_{j\in  \N:I_j=i} |\C^\comp_j| \right ). \]
Hence, for every $\e>0$, as long as $m$ is large enough, for every $i\in \N$ with $|\C_i|\geq \e V^{2/3}$ there exists a $J_i$ such that $|\C_{i}|-|\C^\comp_{J_i} |\leq \e V^{2/3}/4$. Note that $J_i$ is in this case unique.

Then as $\e\to 0$, by \cref{MultMasse2} with high probability the component sizes $(|\C_j^\comp|)_{1\le  j\le k}$ are separated by at least $4\e V^{2/3}$ and larger than $4\e V^{2/3}$. It follows, by reproducing the inductive argument below \eqref{eq:matchCiCj2}, that for every $1\leq a\leq k$ we have $I_a=a$ and $J_a=a$. Thus for every $1\leq i \leq k$,  $| |\C_i|-|\C^\comp_i||\leq \e V^{2/3}$.
By \cref{MultMasse2} we deduce $V^{-2/3}(|\C_1|,\ldots,|\C_k|)$ converges in distribution to $(|\gamma_1|,\ldots,|\gamma_k|)$ where $(\gamma_i)_{i\ge 1}$ are excursions of $(W^\lambda_t)_{t\ge 0}$ above its running infimum (see \cref{sec:intro}). 
Lastly, tightness follows as usual since $\E \sum_{i \geq 1} |\C_i|^2 = \kappa( \lambda) V^{4/3}$.
\qed
 
\section{Convergence of metric space in critical hypercube percolation} \label{sec:FullConvergence} 

The goal of this section is to complete the proof of \cref{thm:main_metric}. We begin with some tightness estimates proving first that the conditions of \cref{GP=>GHP} hold, and second that the sequence of mm-spaces $(M_i)_{i\ge 1}$ is tight for the $L^4$ topology. It follows that it suffices to prove that the convergence in \eqref{thm:ABBGInHypercubeMain} holds with respect to the product GP topology. By \cref{equivGP} this amounts to proving that the ${k \choose 2}$ rescaled distances between every pair of $k$ independent uniformly drawn vertices within finitely many connected components converge to corresponding quantities in the limit sequence $\MM_\lambda$. 

This joint GP convergence of multiple connected components is easily reduced to the case of a single one. In Section~\ref{Conclusion} we present the main argument carrying out the comparison between one fixed connected component of the component graph $G_\comp$ and the corresponding one in $H_{p_c}$. We put everything together in Section~\ref{sec:proof_main_thm_metric}, where the proof of \cref{thm:main_metric} is formally completed.

\subsection{Tightness of the critical hypercube percolation} \label{Sec:GHPtight} 
We will need the following tightness result to verify the second condition in \cref{GP=>GHP} as well as in a few other places in the proofs contained in this section.

\begin{proposition} \label{thm:ghpTightness2} Consider percolation at $p_c(\lambda)$. We have for every $\eta>2$,
$$ \lim_{\e \to 0} \sup_m \proba \Big ( \exists x\in V(H),\delta\leq 1 \, : \, \partial B(x,\delta V^{1/3})\neq \emptyset \andd |B(x,\delta V^{1/3})|\leq \e \delta^{\eta} V^{2/3} \Big ) = 0 \, .$$
\end{proposition}
\begin{proof}
If there exists a vertex $x$ and $\delta\leq 1$ such that 
$$\partial B(x,\delta V^{1/3})\neq \emptyset \quad \andd \quad |B(x,\delta V^{1/3})|\leq \e \delta^{\eta} V^{2/3}\ , $$
then by taking the unique $k\in \N$ such that $2^{-k}\leq \delta < 2^{-k+1}$ we deduce that there exists $k\in \N$ such that the event
$$ E_k := \{\exists x,\, \partial B(x,2^{-k} V^{1/3})\neq \emptyset \andd |B(x,2^{-k} V^{1/3})|\leq \e2^{\eta-\eta k} V^{2/3} \}$$
occurs. We upper bound $\proba(E_k)$ by applying \eqref{eq:ExistsLongThin} with $R_k=2^{-k}V^{1/3}$ and $M_k= \e2^{\eta-\eta k} V^{2/3}$. We are allowed to since $\eta>2$ implies that $2^{-k-1} V^{1/3} \geq c \e 2^{\eta-\eta k} V^{2/3} V^{-1/3}$ and $ 2^{-k}V^{1/3}\geq c\sqrt{\e 2^{\eta-\eta k}V^{2/3}}$ hold for all $k$ whenever $\eps$ is smaller than some fixed small positive constant. So for some $C,c>0$,
\[ \proba(E_k) \leq C2^{k}V^{1/3} e^{-c2^{-2k+\eta k}/\e}\frac{V}{\e 2^{\eta-\eta k}V^{2/3}}\leq \e^{-1}C2^{k+\eta k}e^{-c2^{(\eta-2)k}/\e}\]
Summing over all $k$, the probability of the desired event is at most 
$$ C \sum_{k\geq 0} \e^{-1} 2^{k+\eta k} e^{-c2^{(\eta-2)k} /\e}  = O(\e^{-1} e^{-c/\e}) \, ,$$
which tends to $0$ as $\eps\to 0$. 
\end{proof}
Next, to deduce the convergence of \cref{main_thm} from the weaker convergence for the weak GHP topology, it suffices to show the following result:
\begin{lemma} \label{GHP4tight} Consider percolation at $p_c(\lambda)$ for some fixed $\lambda\in \R$. For any fixed $c>0$, we have 
\[ \lim_{k\to \infty} \limsup_{m\to \infty} \proba \Bigg (\sum_{i\geq k} |\C_i|^4>c V^{8/3} \Bigg )=0\,
\qquad
\text{and}
\quad
\lim_{k\to \infty} \limsup_{m\to \infty} \proba \Bigg (\sum_{i\geq k} \diam(\C_i)^4>c V^{4/3} \Bigg )=0.\]
\end{lemma}

\begin{proof}We focus on the second part as the first is simpler and can be proven similarly. We begin by applying \cref{thm:ghpTightness2} with some fixed $\eta>2$ to be chosen later
$$ \lim_{\e \to 0} \sup_m \proba \Big ( \exists i\geq 1,\delta \leq 1 \, : \, \diam(\C_i) \geq \delta V^{1/3} \andd |\C_i|\leq \e \delta^{\eta} V^{2/3} \Big ) = 0 \, .$$
Hence when choosing $\delta=\min\{1,\diam(\C_i)/V^{1/3}\}$ we obtain 
$$ \lim_{\e \to 0} \sup_m \proba \Big ( \exists i\geq 1 \, : \, |\C_i|\leq \e \min\big\{V^{2/3},\diam(\C_i)^{\eta}V^{2/3-\eta/3}\big\} \Big ) = 0 \, ,$$
or in other words, for any $\alpha>0$, as long as $\e>0$ is small enough, with probability at least $1-\alpha$ we have
\begin{equation}\label{eq:DiamVolume} \forall i\geq 1 \,\, |\C_i|\leq \e V^{2/3} \Longrightarrow \diam(\C_i)^{4}\leq \e^{-4/\eta}|\C_i|^{4/\eta} V^{4/3-8/(3\eta)}\,.
\end{equation}

On the other hand, for any $\beta>0$ we have as usual
$$ \E \Bigg [ \sum_{i \geq 1}|\C_i|^{4/\eta}{\bf 1}_{|\C_i|\leq \beta V^{2/3}} \Bigg ]= V\E \Big [  |\C(v)|^{4/\eta -1}{\bf 1}_{|\C(v)|\leq \beta V^{2/3}} \Big ] \, ,$$ 
and, when $\eta<8/3$ so that $4/\eta-1\in (1/2,1)$, we can estimate the latter using \eqref{eq:ClusterTail}, obtaining
\[V\E\Big [|\C(v)|^{4/\eta -1}{\bf 1}_{|\C(v)|\leq \beta V^{2/3}}  \Big ]  
\leq V \sum_{1\le k \le \beta V^{2/3}} k^{4/\eta-2} \proba(|\C|\ge k)
\le \frac{8-2\eta}{8 -3\eta} C \beta^{{8-3\eta \over 2\eta}} V^{{8 \over 3\eta}} \, .\]
By Markov's inequality, it follows that with probability at least $1-\alpha$
$$ \sum_{i \geq 1}|\C_i|^{4/\eta}{\bf 1}_{|\C_i|\leq \beta V^{2/3}} \leq \frac{8-2\eta}{8 -3\eta} C\alpha^{-1}\beta^{{8-3\eta \over 2\eta}} V^{{8 \over 3\eta}} \, .$$

Together with \eqref{eq:DiamVolume} this implies that with probability at least $1-2\alpha$ for all $\e,\beta>0$ small enough,
$$ \sum_{i \geq 1} \diam(\C_i)^{4} {\bf 1}_{|\C_i| \leq \min(\e,\beta)V^{2/3}} \leq \frac{8-2\eta}{8 -3\eta} C \alpha^{-1} \e^{-4/\eta} \beta^{{8-3\eta \over 2\eta}} V^{4/3}\, .$$
We now fix some $\eta\in(2,8/3)$, choose first $\e>0$ small enough so that \eqref{eq:DiamVolume} holds, then choose $\beta>0$ small enough (in terms of $\e$ and $\alpha$) so that the right-hand side of the above inequality is at most $cV^{4/3}$ (the constant $c>0$ is from the statement). Lastly, since $\E \sum_{i} |\C_i|^2 = V\chi(p_c) = O(V^{4/3})$ we obtain that there exists a $k$ such that with probability at least $1-\alpha$ we have that $|\C_i|\leq \min(\e,\beta)V^{2/3}$ for all $i\geq k$. With the above we conclude that with probability at least $1-3\alpha$ we have
$$ \sum_{i \geq k} \diam(\C_i)^{4} \leq cV^{4/3} \, ,$$
finishing the proof.
\end{proof}

For the proofs in the next section, we will need the following consequence of Proposition \ref{thm:ghpTightness2}.
\begin{lemma} \label{Lem:HTightness} Consider percolation at $p_c=p_c(\lambda)$, let $r\in \N$ and denote by $\{U_{j}\}_{j \in \N}$ a sequence of i.i.d.\ random variables which, given $H_{p_c}$, are distributed as uniform vertices of $\C_r$. Then for any $\e>0$ and $\delta>0$ there exists $N=N(\e,\delta)$ such that as long as $m$ is large enough
\[ \proba_{p_c}\big(\dH(\C_r,\{U_{j}\}_{1\leq j \leq N}) \geq \delta V^{1/3}\big) \geq 1-\e \, .\]
\end{lemma}
\begin{proof} Conditionally on $H_{p_c}$ let $\{V_{j}\}_{1 \leq j \leq K}$ be a maximal set of vertices so that the distance between any two is at least $\delta V^{1/3}/2$. It suffices show that for large enough $N$ (which may only depend on $\e,\delta$) the probability that for every $j \in [K]$ there exists $j' \in [N]$ such that $d(V_{j},U_{j'})\leq \delta V^{1/3}/2$ is at least $1-\e$.
For each $N$, by the union bound, conditionally on $H_{p_c}$, the probability of the complement of this event is at most
\[ \sum_{1\leq j \leq K} \bigg (1- \frac{|B(V_{j},\delta V^{1/3}/2)|}{|\C_r|} \bigg )^N.\]

We write $B_{r,\delta} = \min _{v\in \C_r} |B(v,\delta V^{1/3}/4)|$ and note that $K\leq |\C_r|/B_{r,\e}$ since the balls $B(V_{j},\delta V^{1/3}/4)$ are disjoint for ${1\leq j\leq K}$. 
Therefore
\begin{equation} \proba(\dH(\C_r,\{U_{j}\}_{1\leq j \leq N}) \geq \delta V^{1/3} \, \mid \, H_{p_c}) \leq   \frac{|\C_r|}{B_{r,\delta}}\left (1- \frac{B_{r,\delta}}{|\C_r|} \right )^N. \label{2/7/11h} \end{equation}
Next let $R>0$ and distinguish the right-hand side above depending on whether $|\C_r|/B_{r,\delta}\leq R$ or not. By taking the expectation in \eqref{2/7/11h} we get
\[ \proba_{p_c}(\dH(\C_r,\{U_{j}\}_{1\leq j \leq N})\geq \delta V^{1/3}) \leq R(1-1/R)^N+\proba(|\C_r|/B_{r,\delta}>R) \, .\]
\cref{thm:ghpTightness2} shows that $(V^{2/3}/B_{r,\delta})_{m\in \N}$ is a tight sequence, which together with the fact that $\E|\C_r|^2 \leq V\chi(p_c) = O(V^{4/3})$ implies that $(|\C_r|/B_{r,\delta})_{m\in \N}$ as is also tight. Hence we can choose $R=R(\e,\delta)<\infty$ large enough that the second term on the right-hand side  above is at most $\e/2$ uniformly for all $m$ large enough. We then take $N$ depending on $R$ large enough so that the first term in the right-hand side above is at most $\e/2$ and conclude the proof. 
\end{proof}

\subsection{Convergence in the Gromov--Prokhorov distance} \label{Conclusion}

We begin with some preparations and notation. Recall that we are working with the simultaneous coupling of $H_{p_s}$ and $H_{p_c}$, that is, we have i.i.d.\ random variables $\{U_e\}_{E(H)}$ uniform on $[0,1]$ and for any $p\in [0,1]$ the graph $H_p$ is just the collection of edges with $U_e \leq p$. In this way $H_{p_s}$ is a subgraph of $H_{p_c}$. Recall also that $\Cf_{p_s,M_s}$ is the set of connected components of the percolated hypercube $H_{p_s}$ with size at least $M_s$; it is the vertex set of the sprinkled component graph $G_\comp$ of \cref{sec:CompGraph} and the edges of $G_\comp$ are pairs of components of $\Cf_{p_s,M_s}$ which are linked by an hypercube edge $e$  with $U_e \in (p_s,p_c]$; indeed using \cref{Pro:PositionPc} we assume that $p'_c(\lambda)$ used in the definition of $G_\comp$ equals $p_c(\lambda)$. Also let $d_\comp$ be the shortest path metric on $G_\comp$. We define two distances on $V(H)$

\begin{definition}\label{def:distanceCubeComp} For any two vertices $u,v\in V(H)$
\begin{itemize}
  \item $d_\cube(u,v)$ is the length of the shortest path between $u$ and $v$ in $H_{p_c}$; we set $d_\cube(u,v)=\infty$ if $u,v$ are not connected in $H_{p_c}$.  

  \item $d_{\comp}(u,v) = d_\comp(\C_{p_s}(u),\C_{p_s}(v))$ where $\C_{p_s}(x)$ is the component of $x$ in $H_{p_s}$; we set $d_\comp(u,v)=\infty$ whenever $\C_{p_s}(u)$ and $\C_{p_s}(v)$ are not connected in $G_\comp$.
\end{itemize}
\end{definition}

\noindent {\bf Notation.} In the rest of this section we will often draw two random vertices $U,V$ which conditioned on $H_{p_c}$ are independently uniformly drawn vertices of $\C_r$ for some $r\in \N$ fixed. We will then claim that with high/low probability some event occurs; our meaning is always about the probability in the space of the simultaneous coupling described above, and \emph{not} the conditional probability given $H_{p_c}$. \\

The main goal of this subsection is to prove the following. In \cref{sec:proof_main_thm_metric} we explain how this together with some of the abstract theory presented in \cref{sec:topology} completes the proof of \cref{thm:main_metric}.

\begin{proposition} \label{CompareCompHyp} Fix $r\in \N$ and let $U,V$ be two vertices that conditioned on $H_{p_c}$ are independently uniformly drawn vertices of $\C_r$. Then for any $\delta>0$, with probability tending to $1$ we have
$$ | d_{\cube}(U,V)-d_\comp(U,V)\chi(p_s)| \leq \delta V^{1/3} \, .$$
\end{proposition}

For the proof we denote by $\vs = \cup_{A \in \Cf_{p_s,M_s}} A$ the set of vertices that are in components of size at least $M_s$ in $H_{p_s}$. 
We have implicitly argued that a random vertex $U$ drawn from the $r$-largest component $\C_r$ will belong to $\vs$ with high probability (see the first paragraph in the proof of \cref{Interpole2} below), but it is not clear that $U,V$ are connected using only vertices of the connected components of $\Cf_{p_s,M_s}$ --- indeed, because we have removed small components (of size less than $M_s$), we have to rule out the possibility that every path in $H_{p_c}$ between $U$ and $V$ visits such a small component. This is the content of the next lemma.

\begin{lemma} \label{Interpole2} Fix $r\in \N$ and let $U,V$ be two vertices that conditioned on $H_{p_c}$ are independently uniformly drawn vertices of $\C_r$. Then with probability at least $1-o(1)$ we have that $U,V \in \vs$ and $d_\comp(U,V)<\infty$, and the shortest path between $U$ and $V$ in $H_{p_s}$ only uses vertices in $\vs$.
\end{lemma}
\begin{proof} Recall that $\E N_{p_c} = o(V^{4/3})$ where $N_{p_c}$ is the number of pairs of vertices connected in $H_{p_c}$ such that every path between them goes through a vertex belonging to a connected component of $H_{p_s}$ of size at most $M_s$. In particular, the number of pairs of vertices in $\C_r$ where one of them does not belong to $\vs$ is $o_{\proba}(V^{4/3})$. By \cref{main_thm} we have that $|\C_r| = \Omega_\proba(V^{2/3})$ and so the first assertion of the lemma follows.

Next, set $\ell=V^{1/3}\alpha_m$ and apply \eqref{eq:ExistsLongThin} with $M=M_s$ and $R=\ell$ (it is immediate to verify the conditions for this choice of $M$ and $R$) to obtain
\be\label{eq:NoThinLongClustersAgain} \proba\big(\exists v, |\C_{p_s}(v)|\leq M_s, \diam(\C_{p_s}(v))\geq \ell\big)
=O(V/(\ell M_s)) e^{-cR^2/M_s}
=O(\alpha_m^{-5} e^{-c\alpha_m^{-2}})
=o(1). \ee

Let $\gamma$ be a geodesic path from $U$ to $V$ in $H_{p_c}$ and denote by $e_1,\ldots, e_k$ its edges. Assume first that $k\geq \ell$. By \eqref{eq:subcritMaxDiam} it is the case that $k\leq AV^{1/3}$ with probability at least $1-\e$ for any $\e>0$ as long as $A=A(\e)<\infty$ is large enough. The edges of $\gamma$ are independently closed in $H_{p_s}$ each with probability $1-p_s/p_c=V^{-1/3}\alpha_m^{-1/3}$ by our choice of $p_s$ in \eqref{def:ps}. Hence, conditionally on $\gamma$, the probability that there exist two $p_s$-closed edges of $\gamma$ within distance $\ell$ is, by the union bound, at most $k \ell (V^{-1/3}\alpha_m^{-1/3})^2 = o(1)$ by our choice of $\ell$ and our upper bound on $k$. By a similar calculation, the conditional probability of observing a $p_s$-closed edge in $\gamma$ within distance $\ell$ from either $e_1$ or $e_k$ is $2\ell V^{-1/3}\alpha_m^{-1/3}=o(1)$. Thus, with probability at least $1-O(\e)$ any vertex $v$ of $\gamma$ is such that $\diam(\C_{p_s}(v))\geq \ell$, and so \eqref{eq:NoThinLongClustersAgain} yields that the vertices of $\gamma$ all belong to $\vs$ with high probability. 

Lastly, if $k \leq \ell$, then by the same calculation, the probability that any of $\gamma$'s edges are $p_s$-closed is $o(1)$ hence with probability $1-o(1)$ the vertices $U$ and $V$ belong to same cluster in $H_{p_s}$, and $U,V\in\vs$ so in particular $d_\comp(U,V)=0$. 
\end{proof}

\begin{lemma} \label{Interpole1} Fix $r\in \N$ and let $U,V$ be two vertices that conditioned on $H_{p_c}$ are independently uniformly drawn vertices of $\C_r$. Then for any $\delta>0$, with probability $1-o(1)$, there exists a self-avoiding path $\Gamma$ in $G_\comp$ between (the $H_{p_s}$ components of) $U$ and $V$ of length $L(\Gamma)$ satisfying
\[ d_\cube(U,V)-\delta V^{1/3}\leq \chi(p_s) L(\Gamma)\leq d_\cube(U,V)+\delta V^{1/3}.  \]
\end{lemma}
\begin{proof} Let $\gamma$ be a geodesic from $U$ to $V$ in $H_{p_c}$ and let $e_1,\ldots, e_k$ be its edges ordered from $U$ to $V$. In the entire proof, we use \cref{Interpole2} and work on the event $\Omega$ of probability $1-o(1)$ on which all the vertices of $\gamma$ lie in $\vs$; Since  the maximal diameter of a component in $H_{p_s}$ is at most $C\chi(p_s)\log(V/\chi(p_s)^3)$ with probability $1-o(1)$ by \eqref{eq:subcritMaxDiam}, we may further assume that this is the case on~$\Omega$.  

 On $\Omega$, $\gamma$ corresponds to a path $\Gamma'$ on $G_\comp$: processing the edges $e_i$ sequentially, if $e_i$ is an edge between two distinct connected components of $H_{p_s}$, then add the corresponding edge to the path in $G_\comp$. Note that $\Gamma'$ is not necessarily self-avoiding, so we loop erase it, that is, we erase every loop as it is formed when we traverse $\Gamma'$ in its order of construction. We let $\Gamma$ denote the loop erasure of $\Gamma'$.

Each $e_i$ that we added to the path in $G_\comp$ must correspond to an edge of the hypercube which is $p_s$-closed. Hence the length of $\Gamma$ is at most the number of $p_s$-closed edges in $\gamma$; as in the proof of \cref{Interpole2}, this occurs independently with probability $1-p_s/p_c = (1+o(1))\chi(p_s)^{-1}$ by \eqref{eq:subcritSucespt}. Hence, towards the upper bound on $L(\Gamma)$: if $k \leq \delta V^{1/3}$, then the upper bound is trivial, otherwise, by the aforementioned independence and Chebyshev's inequality we obtain that $L(\Gamma) \leq (1+o(1)) k/\chi(p_s)$, concluding the upper bound since $d_\cube(U,V) = O_\proba(V^{1/3})$. 

To obtain a lower bound on $L(\Gamma)$ matching the upper bound we shall subtract the number of edges $e_i$ of $\gamma$ that are $p_s$-closed but that should not be counted in the upper bound above. These edges are the $p_s$-closed edges in $\gamma$ that either (see Figure \ref{Fig:5.7}) (i) have their two endpoints in the same component of $H_{p_s}$ (and hence correspond to a self-loop in $G_\comp$) or (ii) which correspond to an edge of $\Gamma'$ that lies a cycle in $G_\comp$. An edge $(x,x')$ counted in (i) must have a $p_c$-open path $\pi$ of length at most $C\chi(p_s)\log(V/\chi(p_s)^3)$ connecting $x$ to $x'$ without using $(x,x')$, since this is the maximal diameter of a component in $H_{p_s}$ by assumption, and two disjoint paths connecting respectively $U$ and $V$ to $\pi$. 

 \begin{figure}[t] 
\centering
\includegraphics[scale=1.2]{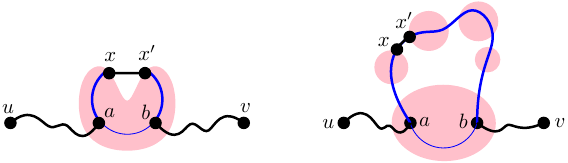}
\caption{\label{Fig:5.7}The edge $(x,x')$ is removed when loop erasing $\Gamma'$ in two ways: on the left case (i) and on the right case (ii). The path $\gamma$ is represented in thick black/blue line. Red blobs represent some connected components of $H_{p_s}$. Blue paths have length at most $L$. 
} 
\end{figure}

As for edges counted in (ii), by definition there must exist a connected component $A \in H_{p_s}$ such that $\gamma$ visits $A$ before and after going through the edge $(x,x')$. In other words there must exist a subpath $\tilde \gamma$ of $\gamma$ starting at $a\in A$ and ending at $b\in A$ that goes through the edge $(x,x')$. Since the maximal diameter of components in $H_{p_s}$ is at most $C\chi(p_s)\log(V/\chi(p_s)^3)$, there is a $p_s$-open path between $a$ and $b$ inside $C$. Also, since $\gamma$ is a geodesic, $\gamma'\subset \gamma$ also is, and so $\gamma'$ have length at most $C\chi(p_s)\log(V/\chi(p_s)^3)$. We conclude that $(x,x')$ is inside of a cycle of length at most $2C\chi(p_s)\log(V/\chi(p_s)^3)$. Additionally, $U$ and $V$ connect to this cycle by two disjoint paths in $H_{p_c}$.

We deduce from this discussion that the desired lower bound on $L(\Gamma)$ will follow once we have bound the number of such edges by $o_\proba(V^{1/3}/\chi(p_s))$. For this denote by $N$ the number of triplets $u,v, (x,x')$ of two vertices $u,v$ and an edge $(x,x')$ such that there exists vertices $a,b$ so that the following events occur disjointly:
\begin{itemize}
  \item The edge $(x,x')$ is $p_c$-open but $p_s$-closed, 
  \item $\{x \stackrel{L}{\lr} a\} \circ \{a \stackrel{L}{\lr} b\} \circ \{b \stackrel{L}{\lr} x'\}$ in $H_{p_c}$,
  \item $\{a \lr u\}\circ \{b \lr v\}$ in $H_{p_c}$,
\end{itemize}
where $L=2C\chi(p_s)\log(V/\chi(p_s)^3)=o(V^{1/3})$. It suffices to prove that $\E N = o(V^{5/3}/\chi(p_s))$: indeed, we have that $|\C_r| = \Omega_\proba(V^{2/3})$ by \cref{main_thm}, so if there were $\Omega(V^{4/3})$ pairs of vertices of vertices $u,v$ in $\C_r$ such that the number of such edges $(x,x')$ is $\Omega(V^{1/3}/\chi(p_s))$, we would get a contradiction to the $o(V^{5/3}/\chi(p_s))$ bound. We bound $\E N$ using the BKR inequality by
\be\label{eq:LowerBoundGamma} \sum_{u,v,a,b,(x,x')} (p_c-p_s) \proba_{p_c}(x \stackrel{L}{\lr} a) \proba_{p_c}(a \stackrel{L}{\lr} b) \proba_{p_c}(b \stackrel{L}{\lr} x') \proba_{p_c}(a \lr u) \proba_{p_c}(b \lr v) \, .\ee
We first sum over $u$ and $v$ the last two terms and get a contribution of $\chi(p_c)^2=O(V^{2/3})$. Next for the sum over $a,b$ over the three terms, considering $(x,x')$ as fixed, we note that for any two vertices $x,y$ and integer $L$, and any $0\leq p_1\leq p_2 \leq 1$ we have $\proba_{p_2}(x \stackrel{L}{\lr} y) \leq (p_2/p_1)^L \proba_{p_1}(x \stackrel{L}{\lr} y)$. Indeed, using our simultaneous coupling, conditioned on the existence of a $p_2$-open path of length at most $L$ between $x$ and $y$, choose the lexicographical first such path, the probability that it remains $p_1$-open is at least $(p_1/p_2)^L$; hence $\proba_{p_2}(x \stackrel{L}{\lr} y) (p_1/p_2)^L \leq \proba_{p_1}(x \stackrel{L}{\lr} y)$. We may thus take $p_2=p_c$ and $p_1=p_c(1-L^{-1})$ and upper bound \eqref{eq:LowerBoundGamma} by
$$ CV^{2/3} (p_c-p_s)  \sum_{(x,x')} \sum_{a,b} \proba_{p_1}(x \lr a) \proba_{p_1}(a \lr b) \proba_{p_1}(b \lr x') \,,$$
for some constant $C$ (we have $(p_2/p_1)^{3L}\le e^3$).
By \eqref{eq:subcritSucespt} we have that $\chi(p_1)=(1+o(1))L$, and so \eqref{eq:TriangleCondition} implies that the sum over $a,b$ is at most $O(\alpha_m+L^3/V)=O(\alpha_m \log \alpha_m^{-1})$. We sum this over $(x,x')$ getting a factor of $Vm$, the factor $p_c-p_s$ is as usual at most $Cm^{-1} V^{-1/3}\alpha_m^{-1/3}$. Putting all these together gives a bound of $O(V^{4/3} \alpha_m^{2/3}\log \alpha_m^{-1})$ which is indeed $o(V^{5/3}/\chi(p_s))$ since $\chi(p_s) = O(V^{1/3}\alpha_m^{1/3})$, concluding the proof. 
\end{proof}

We note that the last lemma already gives the required upper bound on $d_\comp(U,V)$ in \cref{CompareCompHyp}, the main obstacle is that the path $\Gamma$ constructed in \cref{Interpole1} is not necessarily the shortest path in $G_\comp$. However, we will soon prove (\cref{lem:GCompHaveNoSmallCycles}) that large components of $G_\comp$ do not have short cycles, so if somehow $d_\comp(U,V)$ is small, then in fact $\Gamma$ \emph{is} a shortest path between $U$ and $V$ in $G_\comp$.

A technical difficulty which will unfortunately arise shortly forces us to consider now the small components of $H_{p_s}$ as well as the larger ones (in particular, the proof of \cref{lem:DCompTightness} fails unless we work with $G_\rcomp$, see definition below). The reason is that we will need to make some rough comparisons between $d_\cube$ and $d_\comp$ that are valid for \emph{all} vertices in $\vs$ (not just random ones) --- this is the contents of \cref{TransformTight} and \cref{lem:DCompTightness}. When the vertices are not random, it may well be that a shortest path between them in $H_{p_c}$ does \emph{not} remain in $\vs$ and traverses through the small components of $H_{p_s}$. 

To overcome this we write $G_\rcomp$ for the graph whose vertex set consist of \emph{all} components of $H_{p_s}$ such that two vertices are connected if there exists an hypercube edge $e$ connecting the two corresponding components such that $U_e \in [p_s,p_c]$. We call $G_\rcomp$ the \textbf{full component graph}. Of course $G_\comp$ is a subgraph of $G_\rcomp$ so $d_\rcomp(u,v) \leq d_\comp(u,v)$ for any two vertices $u,v\in \vs$. We remark that we cannot pull the proofs of earlier sections (particularly that of \cref{sec:CompGraph}) with $G_\rcomp$; this occurs in various places, perhaps the most significant one is that the statement in \cref{prop:Delta} would fail if we did not remove the small components of $H_{p_s}$.

\begin{lemma} \label{TransformTight} For every $\e, \alpha>0$ there exists $\beta >0$ such that with probability at least $1-\e$, 
$$ \forall u,v\in \V_\star \quad d_\rcomp(u,v) \leq \beta V^{1/3}/\chi(p_s) \implies d_\cube(u,v) \leq \alpha V^{1/3} \, .$$
\end{lemma}
\begin{proof} 
First by \eqref{eq:subcritMaxDiam}, we may consider  $\lambda'<\lambda$ small enough such that the diameter of $H_{p_c(\lambda')}$ is smaller than $\alpha V^{1/3}$ with probability at least  $1-\e/2$, and then set $\beta=1/(3(\lambda-\lambda'))$. Assume now that there exist vertices $u,v\in \V_\star$ such that $\alpha V^{1/3} \leq d_\cube(u,v)$ and  $d_\rcomp(u,v) \leq \beta V^{1/3}/\chi(p_s)$. Then there is a shortest path $\gamma$ from $u$ to $v$ in $H_{p_c(\lambda)}$ of length at least $\alpha V^{1/3}$ with at most $\beta V^{1/3}/\chi(p_s)$ edges closed in $H_{p_s}$. 
Conditioned on $H_{p_c}$ and $H_{p_s}$, the probability that each $p_c$-open but $p_s$-closed edge is also closed in $H_{p(\lambda')}$ is
\[ \frac{p_c(\lambda)-p_c(\lambda')}{p_c(\lambda)-p_s}= (1+o(1)) (\lambda-\lambda')\alpha_m^{1/3} \, , \]
using Corollary \ref{cor:Window}, our definition of $p_s$ in \eqref{def:ps}, and the fact that $p_c(\lambda)\sim 1/m$. 
 So the conditional probability that $\gamma$ stays open in $H_{p(\lambda')}$ is at least 
 \[ 1-\beta V^{1/3} \chi(p_s)^{-1} (\lambda-\lambda')\alpha_m^{1/3}(1+o(1))= 1-(1+o(1))\beta(\lambda-\lambda')>1/2 \, , \]
 by our choice of $\beta$. If this occurs, then clearly $u$ and $v$ remain connected in $H_{p_c(\lambda')}$ and $\gamma$ is also a shortest path in $H_{p_c(\lambda')}$ between $u$ and $v$ implying that the diameter of $H_{p_c(\lambda')}$ is at least $\alpha V^{1/3}$. Hence
  \[ \proba\big(\exists u,v \in \vs: \alpha V^{1/3} \leq d_\cube(u,v) \text{ and }d_\comp(u,v) \leq \beta V^{1/3}/\chi(p_s)\big) \leq 2 \proba(\diam(H_{p_c(\lambda')})\geq \alpha V^{1/3})\leq \e \, ,\] 
  by our choice of $\lambda'$. This concludes the proof.
\end{proof}

We will also need a converse of the last lemma, that is, we would like to show that a small $d_\cube$ distance implies small $d_\rcomp$ distance for \emph{all} pairs $u,v\in \vs$. It is the proof of this lemma that would be invalid if we were to consider $d_\comp$ instead of $d_\rcomp$.

\begin{lemma} \label{lem:DCompTightness} For every $\delta_1>0$ there exists $\delta_2 >0$ such that with probability $1-o(1)$, 
$$ \forall u,v\in \V_\star \quad d_\cube(u,v) \leq \delta_2 V^{1/3} \implies d_\rcomp(u,v) \leq \delta_1 V^{1/3}/\chi(p_s) \, .$$
\end{lemma}
\begin{proof} Fix an arbitrary $\delta_1>0$ and let $\delta_2=\delta_2(\delta_1)>0$ be a small number that will be chosen at the end of the proof; we also set $r=2\delta_2 V^{1/3}$. 

If there exist vertices $u,v\in \vs$ so that $d_\cube(u,v) \leq r/2$ and $d_\rcomp(u,v) \geq \delta_1 V^{1/3}/\chi(p_s)$, then the shortest path between $u$ and $v$ in $H_{p_c}$ must induce a path of length at least $\delta_1 V^{1/3}/\chi(p_s)$ in $G_\rcomp$. This implies that there exists $K \geq \delta_1 V^{1/3}/\chi(p_s)$ and a finite sequence of vertices $A_0, A_1\ldots, A_K$ of $G_\rcomp$  such that $u\in A_0$ and $v\in A_K$ and a shortest path in $H_{p_c}$ from $u$ to $v$ of length at most $r/2$ that goes through the vertices of $A_0,\ldots, A_K$ successively without visiting any other vertices of $G_\rcomp$. Note that we do not require that the $A_i$'s be distinct since they do not necessarily form a shortest path in $G_\rcomp$. 

Since $u,v\in \vs$ it must be the case that $A_0$ and $A_K$ are both of size at least at least $M_s$, and by \eqref{eq:subcritMaxDiam} their diameter is $o(V^{1/3})$ with probability $1-o(1)$. Hence, any pair of vertices from $A_0$ and $A_K$ can take the role of $u$ and $v$ and that increases the length of this shortest path to at most $r/2+o(V^{1/3})\leq r$. We deduce that there are at least $M_s^2$ pairs of vertices $u_0,v_K$ such that there exist vertices $v_0,\ldots, v_{K-1}$ and $u_1,\ldots, u_K$ such that $v_j$ and $u_{j+1}$ are neighbors in $H$ for all $0\leq j\leq K-1$, as well as lengths $\{\ell_j\}^K_{j=0}$ of total sum at most $\sum_{0\le j \leq K} \ell_j \leq r$, so that the next $2K+1$ events occur disjointly (see Figure \ref{Fig:5.9}):
\begin{enumerate}
  \item For all $j=0,\ldots K$, $u_j$ is connected to $v_j$ in $H_{p_s}$ by a path of length at most $\ell_j$, 
  \item For all $j=0,\ldots K-1$ the edges $(v_j,u_{j+1})$ are $p_s$-closed but $p_c$-open.
\end{enumerate}
\begin{figure}[t] 
\centering
\includegraphics[scale=0.8]{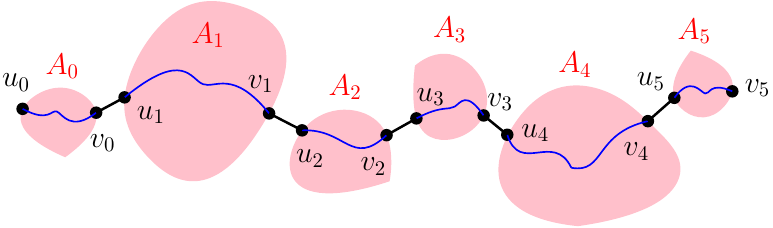}
\caption{\label{Fig:5.9} A path in $H_{p_c}$ with corresponding length $K=5$ in $G_{\tilde \comp}$ is represented. The red blobs represent connected components of $H_{p_s}$. Paths in $H_{p_s}$ are drawn in blue and edges of $H_{p_c}\backslash H_{p_s}$ are drawn in black.}
\end{figure}
We will bound the expected number of such pairs $u_0,v_K$. First we need to consider all the possible collections of lengths $\ell_j$'s but in order not to overcount too much, we enumerate them in logarithmic scales starting from $r/K$. In particular, let $L:\{0,\ldots,K\}\to \mathbb{N}$ be defined by
$$ L(j) = \begin{cases}
      1, & \text{if $\ell_j \leq {r\over K}$}\\
            i, & \text{if $\ell_j \in \Big [ {2^{i-1}r \over K}, {2^{i}r \over K}$\Big )}\,.
     \end{cases}
$$
The range of $L$ is in fact contained $\{1,\ldots, \lfloor \log_2 K\rfloor +1\}$ since $\ell_j \leq r$ for all $j$. Furthermore, since the total length $\sum_{0\le j \leq K} \ell_j$ is at most $r$ the number of $j$'s such that $L(j)=i$ is at most ${2K \over 2^i}$. Hence the number of such possible $L$'s is at most 
\begin{equation}\label{eq:number_maps_L}
\prod_{i=0}^{\lfloor \log_2 K\rfloor +1} {K \choose \lfloor 2K/2^i\rfloor} 
\leq \prod_{i=0}^{\lfloor \log_2 K \rfloor +1} (e2^i)^{\lceil 2K/2^i\rceil } \leq C^K \, ,  
\end{equation}
where $C>0$ is some universal constant (in the first inequality we used that ${a \choose b} \leq (ea/b)^b$). 

We now bound the expected number of pairs $u_0, v_K$ using the BKR inequality by
$$  \sum_{\substack{u_0,v_0\,\ldots, u_{K},v_K\\ (v_j,u_{j+1})\in E(H)}} \sum_L (p_c-p_s)^{K} \prod_{j=0}^{K} \proba_{p_s}\big(u_j \stackrel{2^{L(j)}r/K}{\lr} v_{j}\big) \, .$$
Given $L$, we sum first over $u_0$ using \eqref{eq:IntVolume} to obtain a factor of $C 2^{L(1)}r/K$. We then sum over $v_0$ giving a factor of $m$. We continue this way and get for each $j\in \{0,1,\dots, K\}$ a factor of $C 2^{L(j)}r/K$ and $K$ factors of $m$ in total. The sum over the last vertex $v_K$ gives a factor of $V$. We obtain an upper bound on the desired expectation of
$$ V \cdot (p_c m V^{-1/3} \alpha_m^{-1/3})^{K} \sum_L \prod_{j=0}^{K} C2^{L(j)}r/K \, .$$
Since $2^{L(j)-1}r/K \leq \ell_j$, we have that $\sum_j 2^{L(j)}r/K \leq 2r$. Hence by the arithmetic--geometric mean inequality we bound the product above by $(2Cr/K)^{K+1}$. Since the number of maps $L$ is at most $C^K$ by \eqref{eq:number_maps_L}, the expectation is at most
$$  V \cdot (p_c m V^{-1/3} \alpha_m^{-1/3})^{K} (2C^2 r/K)^{K+1} \, .$$
We now plug in the values of $r$ and the lower bound $K \geq \delta_1 V^{1/3}/\chi(p_s)$ and recall that by \eqref{eq:subcritSucespt} we have $\chi(p_s) \leq 2V^{1/3} \alpha_m^{1/3}$ and that $m p_c\leq 2$, so the above is bounded by 
$$ V \chi(p_s) (16 C^2\delta_2/\delta_1)^{K+1} \, .$$
We now choose $\delta_2>0$ small enough so that $16C^2\delta_2/\delta_1\leq 1/2$, sum over $K \geq \delta_1 V^{1/3}/{\chi(p_s)}$ (the first term is dominant) and thus bound this expectation above by $V^{4/3} 2^{-\delta_1 \alpha_m^{-1/3}}$. As mentioned earlier, the event in the statement of the lemma implies that there are at least $M_s^2$ such pairs $u_0,v_K$, hence by Markov's inequality the probability of the event is at most 
$$M_s^{-2} V^{4/3} 2^{-\delta_1 \alpha_m^{-1/3}} = \alpha_m^{-8} 2^{-\delta_1 \alpha_m^{-1/3}} = o(1) \, ,$$
as required.
\end{proof}

Our last preparatory lemma shows that $G_\rcomp$ has no short cycles in large components (this is clearly false for $H_{p_c}$ since there will be many cycles of length $4$ in large components).

\begin{lemma} \label{lem:GCompHaveNoSmallCycles}  For any $\e>0$ and $\tau>0$ there exists $\alpha=\alpha(\e,\tau)>0$ such that as long as $m$ is large enough, with probability at least $1-\e$, no connected components of $G_{\rcomp}$ of weight at least $\tau V^{2/3}$ contains a cycle of length less than $\alpha V^{1/3}/\chi(p_s)$.
\end{lemma}
\begin{proof} Let $X$ denote the number of vertices $x$ that belong to a component of $G_\rcomp$ that contains a cycle of length at most $\alpha V^{1/3}/\chi(p_s)$. We will show that $\E X = O(\alpha V^{2/3})$ from which the lemma follows immediately by Markov's inequality. If $x$ is such a vertex, then there exist $3 \leq \ell \leq \alpha V^{1/3}/\chi(p_s)$ and vertices $A_1,\ldots,A_\ell$ of $G_\rcomp$ (that is, components of $H_{p_s}$) that form a cycle in $G_\rcomp$ (in that order) and additionally either $x\in \cup _{i \leq \ell} A_i$, or $x$ is connected in $H_{p_c}$ to $A_1$. We write respectively $X_1$ and $X_2$ for the two corresponding contributions to $X$.

The presence of such a cycle implies that there exist vertices $v_1,u_1,\ldots, v_\ell,u_\ell$ such that $v_i,u_i\in A_i$ and $(v_i, u_{i+1})$ is an edge that is closed in $H_{p_s}$ but open in $H_{p_c}$ for each $1 \leq i \leq \ell$ (where $u_{\ell+1} = u_1$). By \cite[Theorem 1.2, (1.17)]{BCHSS1} with high probability all the components of $H_{p_s}$ are of size at most $O(\chi(p_s)^2 \log(V/\chi(p_s)^3))$; multiplying this by $\ell \leq V^{1/3}/\chi(p_s)$ is $o(V^{2/3})$, hence, the contribution $\E[X_1]$ of the vertices $x$ lying inside the cycle (in $\cup_{i\le \ell} A_i$) is at most $o(V^{2/3})$. 

We now proceed to bound the contribution $\E[X_2]$ of the vertices $x\not \in \cup_{i\leq \ell} A_i$. Observe that, in this case, there must exist $z \in \cup_{i\le \ell} A_i$ and $z' \not \in \cup_{i \leq \ell}A_i$ such that $(z,z')$ is an edge that is closed in $H_{p_s}$ but open in $H_{p_c}$ and $x$ is connected to $z'$ in $H_{p_c}$.  Without loss of generality, we will assume that $z\in A_1$. We obtain that there exist vertices $v_1,u_1,\ldots, v_\ell,u_\ell,z,z',w$ such that $(v_i,u_{i+1})$ and $(z,z')$ are edges such that the following events occur disjointly (see Figure \ref{Fig:5.10}): 
\begin{enumerate} 
  \item The edges $\{(v_i,u_{i+1})\}_{i=1}^\ell$ and $(z,z')$ are closed in $H_{p_s}$ but open in $H_{p_c}$. 
  \item There exists $w$ such that $\{v_1\lr w\}\circ\{u_1 \lr w\}\circ \{w \lr z\}$ in $H_{p_s}$.
  \item For any $2 \leq i \leq \ell$ the vertices $u_i$ and $v_i$ are connected by an open path in $H_{p_s}$.
  \item $x$ is connected to $z'$ in $H_{p_c}$. 
\end{enumerate}
\begin{figure}[t]
\centering
\includegraphics[scale=1.0]{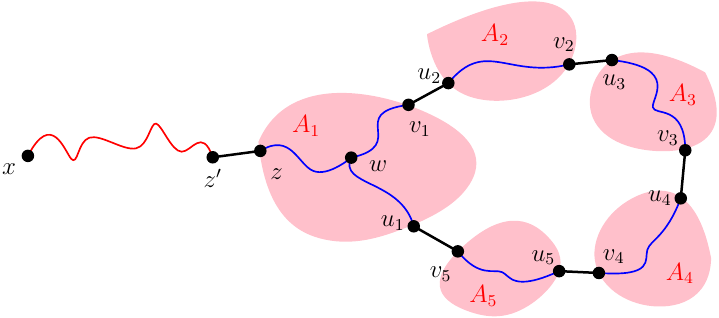}
\caption{\label{Fig:5.10} The cycle event from the proof of Lemma \ref{lem:GCompHaveNoSmallCycles} is represented. The red blobs represent some connected components of $H_{p_s}$. Blue paths are in $H_{p_s}$. Black segment are edges in $H_{p_c}\backslash H_{p_s}$. The path between $x$ and $z'$ represented in red is in $H_{p_c}$.
}
\end{figure}
By the BKR inequality we obtain that 
\begin{align*} 
\E X_2
\leq \sum_{\ell=1}^{\alpha V^{1/3}/\chi(p_s)} & (p_c-p_s)^{\ell+1} \sum_{\substack{x,(z,z'), w \\ v_1,u_1, \ldots, v_\ell,u_\ell}}  \proba_{p_c}(x \lr z') \proba_{p_s}(v_1 \lr w)\proba_{p_s}(u_1 \lr w) \\ 
&\times \proba_{p_s}(w \lr z)   \displaystyle \prod_{2 \leq i \leq \ell} \proba_{p_s}(v_i \lr u_{i}) \, .
\end{align*}
We now evaluate this sum. We begin by summing over $x$ to get a contribution of $\chi(p_c)=O(V^{1/3})$, then over $z$ to get a contribution of $\chi(p_s)$, then over $z'$ and get a contribution of $m$. This simplifies the above sum to
\begin{eqnarray}\label{eq:GCompNoCycles} \E X_2 \leq C \sum_{\ell} (p_c V^{-1/3}\alpha_m^{-1/3})^{\ell+1} m V^{1/3} \chi(p_s) \sum_{\substack{ w, v_1,u_1, \ldots, v_\ell,u_\ell}}  &\proba_{p_s}(v_1 \lr w)\proba_{p_s}(u_1 \lr w) \nonumber 
\\ &\times \displaystyle \prod_{2 \leq i \leq \ell} \proba_{p_s}(v_i \lr u_{i}) \, . \end{eqnarray}
Now, if any of the connection events above occur in a path of length at least $m_0$, we use \eqref{eq:UniformConnection} to bound that probability by $C\chi(p_s)/V$ and the rest of the sum is then easily seen to give a contribution of $V(m\chi(p_s))^\ell$. The contribution of this case to $\E X_2$ is thus
$$ \sum_{\ell=3}^{\alpha V^{1/3}/\chi(p_s)} C (p_c V^{-1/3}\alpha_m^{-1/3})^{\ell+1} V^{1/3} (m \chi(p_s))^{\ell +1} \chi(p_s) \, .$$
We have that $(mp_c)^\ell=1+o(1)$ (since $\ell \leq \alpha_m^{-1/3}$ and $p_c=m^{-1} + O(m^{-2})$), and since $\chi(p_s)=(1+O(\alpha_m))V^{1/3} \alpha_m^{1/3}$ we have $(\chi(p_s)V^{-1/3}\alpha_m^{-1/3})^{\ell+1} = 1+o(1)$, so we are left with a contribution of $\chi(p_s)V^{1/3}$ for each $\ell$, giving in total a contribution of $O(\alpha V^{2/3})$ to $\E X_2$ (and thus $\E X$), as desired. 

This is indeed the main contribution to $\E X_2$ and we are left to bound \eqref{eq:GCompNoCycles} in the case that all connection events occur in paths of length less than $m_0$. We use \eqref{eq:PercolationNBWRW} and bound the second sum in \eqref{eq:GCompNoCycles} by
$$ (1+o(1))^{\ell+2} m^\ell \sum_{\substack{t_{1}^{(a)},t_{1}^{(b)}, \\ t_2,\ldots,t_{\ell}=0}}^{m_0-1}  \sum_{\substack{w,v_1,\ldots v_\ell \\ u_1,\ldots, u_\ell}}  \pp^{t_{1}^{(a)}}(v_1,w) \pp^{t_{1}^{(b)}}(u_1,w) \pp^1(v_1,u_2) \prod_{i=2}^\ell \pp^{t_i}(u_i,v_i)\pp^1(v_{i},u_{i+1}) \, ,$$
and note that the terms $m\pp^1(v_i,u_{i+1})$ guarantee that $(v_i,u_{i+1})$ are edges. Using the natural generalization of \eqref{eq:NBWconvolution} this equals
$$ (1+o(1))^{\ell+2} V m^\ell \sum_{\substack{t_{1}^{(a)},t_{1}^{(b)}, \\ t_2,\ldots,t_{\ell}=0}}^{m_0-1} \pp^{\ell + t_1^{(a)} + t_1^{(b)} + t_2 \cdots + t_\ell}(v,v; t_1^{(a)},t_1^{(b)},1,t_2,1,\ldots,t_\ell,1) \, .$$
We proceed very crudely since we have a lot of room to spare, and bound the term in the sum by $1$ and just bound the above by $2^{\ell}Vm^\ell m_0^{\ell+1}$. We put this back into \eqref{eq:GCompNoCycles} and get a bound of 
$$ \sum_{\ell=3}^{\alpha V^{1/3}/\chi(p_s)}  C (2p_c V^{-1/3}\alpha_m^{-1/3}m m_0)^{\ell+1} V \cdot V^{1/3} \chi(p_s) \, .$$
Since we assume for \cref{main_thm_general} that $m_0=O(V^{1/15}\alpha_m$) and $p_c m\sim 1$, the factor in parenthesis goes to $0$ as $m\to \infty$. So the main contribution of the above sum comes from $\ell=3$, which is of order $O(1)(V^{-1/3}\alpha_m^{-1/3} m_0)^3 V^{4/3}\X(p_s)$. Inserting $\X(p_s)=o( V^{1/3})$ and $m_0=O(V^{1/15}\alpha_m)$ yields the desired bound of $o(V^{2/3})$. 
 \end{proof}

We now have everything in place to prove \cref{CompareCompHyp}.

\begin{proof}[Proof of \cref{CompareCompHyp}]
Recall that $U,V$ are independent uniformly drawn vertices of the $r$-th largest component $\C_r$ of $H_{p_c}$, where $r\in \N$ is a fixed number. We first apply \cref{Interpole1} and since $d_\comp(U,V) \leq L(\Gamma)$ we obtain that for any $\delta>0$ the event
\begin{equation}\label{eq:upper-bound_dcube}
 \chi(p_s)d_\comp(U,V) \leq d_\cube(U,V) + \delta V^{1/3}
\end{equation}
occurs with probability $1-o(1)$, so it remains to prove a lower bound on $d_\comp(U,V)$. 

To this end, we first note that the inequality $\chi(p_s)d_{\comp}(U,V) \geq d_\cube(U,V)-\delta V^{1/3}$ holds trivially if $d_\cube(U,V)\leq \delta V^{1/3}$. Also by \cref{thm:ghpTightness2} as $\tau \to 0$ we have
\[ \sup_m \proba(d_\cube(U,V)\geq \delta V^{1/3}, |\C_r|\leq \tau V^{2/3})\to 0 , \,\] 
so we may also assume that $|\C(U)|\geq \tau V^{2/3}$ for some small arbitrary $\tau>0$. Thus, it suffices to prove that 
for any $\tau>0$ and $\zeta>0$ we have 
\be\label{eq:MainProp} \proba\big(|\C_r|>\tau V^{2/3}, \chi(p_s)d_\comp(U,U')\geq  d_\cube(U,U')-\zeta V^{1/3}\big)=1- o(1) \, . \ee

Let $\e>0$ be an arbitrarily small but fixed constant. We will show that the event above has probability at least $1-O(\e)$. 
By \cref{lem:GCompHaveNoSmallCycles} we may take $\alpha=\alpha(\e,\tau)>0$ small enough such that as long as $m$ is large enough with probability at least $1-\e$ we have that
\begin{equation} |\C_r|>\tau V^{2/3} \implies \text{all cycles of $\C_r^{\rcomp}$ have length at least $8\alpha V^{1/3}/\X(p_s)$} \, ,\label{eq:NoShortCycles} \end{equation}
where $\C_r^{\rcomp}$ is just the $r$-th largest component of $G_\rcomp$ (the vertices of $\C_r^{\rcomp}$ are the $H_{p_s}$ components contained in $\C_r$). Assume this event holds. We now apply \cref{TransformTight} to obtain a constant $\beta=\beta(\alpha)>0$ so that with probability at least $1-\e$ we have 
\begin{equation}\label{eq:Compare3} \forall u,v\in \vs \quad d_\rcomp(u,v) \leq \beta V^{1/3}/\chi(p_s) \implies d_\cube(u,v) \leq \alpha V^{1/3} \, ,\end{equation}
and assume this event holds. For convenience we assume that $\beta V^{1/3}/\chi(p_s)\in \N$; otherwise we round and carry negligible errors, we omit the details.

Next by \cref{Interpole2} we know that $d_{\comp}(U,V)<\infty$ with probability at least $1-o(1)$. On this event we can take the shortest path between them in $G_\comp$ and split it into intervals of length $\beta V^{1/3}/\chi(p_s)$; more precisely,  
we find (random) vertices $x_1,\ldots, x_k$ such that $U=x_1$ and $V=x_k$ and 
\be \label{eq:SplitGeodesicToIntervals} \forall i=1,\ldots, k-2 \quad d_\comp(x_i,x_{i+1}) = \beta V^{1/3}/\chi(p_s) \qquad d_\comp(x_{k-1},x_k)\leq \beta V^{1/3}/\chi(p_s) \, , 
\ee
and
\be\label{eq:NumberOfIntervals} k \leq \beta^{-1}\chi(p_s) V^{-1/3} d_\comp(U,V) + 1  \, . \ee
We record the direct consequence of our choice of $\beta$ in \eqref{eq:Compare3} and our choice of the $x_i$'s:
\be \label{CubeDistXiXiplus} d_\cube(x_i,x_{i+1}) \leq \alpha V^{1/3} \qquad \forall i=1,\ldots,k-1 \, .\ee

We think of $\e, \alpha,\beta$ (as well as $\tau$ and $\zeta$ from the statement of \eqref{eq:MainProp}) as constants from now on, and set $\delta_1 \in (0,\alpha)$ to be a small parameter that we will choose at the end of the proof as a function of these constants. By Lemma \ref{lem:DCompTightness} there exist $\delta_2 = \delta_2(\delta_1)\in(0,\delta_1)$ so that with probability at least $1-o(1)$ we have
\begin{equation}\label{eq:Compare2} \forall u,v\in \vs \quad d_\cube(u,v)\leq \delta_2 V^{1/3} \implies d_\rcomp(u,v) \leq \delta_1 V^{1/3}/\X(p_s) \, ,\end{equation}
and we assume this event holds.

Next we apply \cref{Lem:HTightness} to obtain $N=N(\e,\delta_2)$ large enough so that with probability at least $1-\e$ we have
\begin{equation} d_{H}^\cube(\C_r,\{U_{j}\}_{1\leq j \leq N}) \leq \delta_2 V^{1/3} \, , \nonumber \label{eq:HausdorffApprox} \end{equation}
where given $H_{p_c}$ the random variables $\{U_{j}\}$ are i.i.d.\ uniform vertices in $\C_r$ and $d_{H}^\cube$ stands for the Hausdorff distance in $(H,d_\cube)$ (see \cref{sec:intro}). We assume this event holds and in particular, 
this implies that for each $i=1,\ldots, k$ there exists $j_i$ so that 
\be\label{eq:CubeDistAiUi} d_\cube(x_i, U_{j_i}) \leq \delta_2 V^{1/3} \, ,\ee
where we set $U_{j_1}=U=x_1$ and $U_{j_{k}}=V=x_k$. It follows by \eqref{eq:Compare2} that
\be\label{eq:CompDistAiUi} d_\rcomp(x_i, U_{j_i}) \leq \delta_1 V^{1/3}/\chi(p_s) \, ,\ee
and also \eqref{CubeDistXiXiplus} together with the triangle inequality and the fact that $\delta_2\leq \alpha$ imply that
\be\label{eq:CubeDistUiUiplus} d_\cube(U_{j_i},U_{j_{i+1}}) \leq (2\delta_2+\alpha) V^{1/3} \leq 3\alpha V^{1/3} \qquad \forall i=1,\ldots,k-1 \, .\ee

We now apply \cref{Interpole1} and obtain that with probability $1-o(1)$ for each pair $j_{i},j_{i+1}$ there exists a path in $G_\rcomp$ (in fact in $G_\comp$) between (the $H_{p_s}$ components of) $U_{j_i}$ and $U_{j_{i+1}}$ of length $L =L(i)$ satisfying
\begin{equation} d_\cube(U_{j_i},U_{j_{i+1}})-\delta_2 V^{1/3}\leq \chi(p_s) L\leq d_\cube(U_{j_i},U_{j_{i+1}})+\delta_2 V^{1/3} \, .  \label{eq:ConstructPath}  \end{equation}
By \eqref{eq:NoShortCycles} we learn that if $L \leq 4\alpha V^{1/3}/\chi(p_s)$, then $L= d_\rcomp(U_{j_i},U_{j_{i+1}})$. This bound on $L$ indeed holds by \eqref{eq:CubeDistUiUiplus} and the second inequality in \eqref{eq:ConstructPath}. Thus, for each $i=1,\ldots, k-1$ we obtain
$$ |\chi(p_s) d_\rcomp(U_{j_i},U_{j_{i+1}}) - d_\cube(U_{j_i},U_{j_{i+1}}) | \leq \delta_2 V^{1/3} \, .$$
By \eqref{eq:CubeDistAiUi} and \eqref{eq:CompDistAiUi} this implies that for each $i=1,\ldots, k-1$ 
$$ |\chi(p_s) d_\rcomp(x_i,x_{i+1}) - d_\cube(x_i,x_{i+1}) | \leq (2\delta_1+3\delta_2) V^{1/3} \leq 5\delta_1 V^{1/3} \, .$$

We now bound using this and the triangle inequality
$$ d_\cube(U,V) \leq \sum_{i=1}^{k-1} d_\cube(x_i, x_{i+1}) \leq \sum_{i=1}^{k-1} \chi(p_s)d_\rcomp(x_i,x_{i+1}) + 5\delta_1 k V^{1/3} \, .$$
We now use \eqref{eq:SplitGeodesicToIntervals}, \eqref{eq:NumberOfIntervals} and that $d_\rcomp\le d_\comp$ to bound from above the first term on the right-hand side by $\chi(p_s) d_\comp(U,V)$. 
The second term is bounded by $5\delta_1 \beta^{-1}\chi(p_s)d_{\comp}(U,V)+5 \delta_1 V^{1/3}$ using \eqref{eq:NumberOfIntervals}. Put together 
$$ d_\cube(U,V) \leq (1+5\delta_1 \beta^{-1}) \chi(p_s)d_\comp(U,V) + 5 \delta_1 V^{1/3} \, ,$$
or alternately $\chi(p_s)d_\comp(U,V) \geq d_\cube(U,V) - 5\delta_1(\beta^{-1} \chi(p_s) d_\comp(U,V)+ V^{1/3})$.
By \eqref{eq:MaxDiamCrit}, there exists $A=A(\e,\lambda)<\infty$ such that the diameter of any cluster in $H_{p_c}$ is at most $AV^{1/3}$  with probability at least $1-\e/2$; in this case, in particular, $d_\cube(U,V) \leq AV^{1/3}$. We now conclude the proof by choosing $\delta_1>0$ small enough so that first $5\delta_1 < \zeta/2$ and second
\[5\delta_1 \beta^{-1} \chi(p_s) d_\comp(U,V) \le 5 \delta_1 \beta^{-1} (AV^{1/3}+\delta_1 V^{1/3})< \zeta V^{1/3}/2\,,\] 
where the first inequality uses the fact that $\chi(p_s)d_\comp(U,V)\le d_\cube(U,V)+\delta_1 V^{1/3}$ with probability at least $1-\e/2$ by \eqref{eq:upper-bound_dcube}.
\end{proof}
 
\subsection{Proof of Theorem \ref{thm:main_metric}} \label{sec:proof_main_thm_metric}

In the following, in order to emphasize the dependence in $m$, we write $M^m_i=(\C_i,V^{-1/3}d_i^\cube,\mu_i^\cube)$ and $M_i$ for the $i$th largest connected component of $H_m$. We also write $\MM^m=(M_i^m)_{i\ge 1}$. Recall that $\MM_\lambda$ denotes the limit vector. We will write $M_i^\infty$, $i\ge 1$, for the components of $\MM_\lambda$, and let $\diam(M_i^\infty)$ and $|M_i^\infty|$ for the corresponding diameter and mass. Our aim is to prove that $\MM^m \to \MM_\lambda$ in distribution for the $L^4$ GHP topology, and we proceed by successive strengthenings. 

We start by proving that $\MM^m$ converges in distribution to $\MM_\lambda$ for the product Gromov--Prokhorov topology, that is: for every fixed finite $S\subset \N$, the collection $(M_i^m)_{i\in S}$ converges to $(M_i^\infty)_{i\in S}$, where the convergence of each component is with respect to the GP topology. To do so, denote by $d_i^\comp$ the shortest path metric in $\C_i^\comp$ and by $\mu_i^\comp$ the measure on $\C_i^\comp$ defined by $\mu_i^\comp(A) = V^{-2/3}|A|$ and lastly write $M_i^\comp$ for the mm-space 
$$ M_i^{\comp} = \Big(\C_i, \chi(p_s)V^{-1/3} d_i^\comp, \mu_i^\comp\Big) \, ,$$
exactly as above \cref{pro:conv_GP_comp}. 
Fix a finite subset $S\subseteq \N$, and let $\ell\ge 1$ be an arbitrary natural number. For each $i\in S$, let $(\xi^{\cube}_{i,j})_{j=1}^\ell$ be $\ell$ i.i.d.\ uniformly random vertices of $\C_i$. For each $i\in S$ and $j \in [\ell]$ let $\xi^\comp_{i,j}$ be the component containing $\xi^\cube_{i,j}$ in $H_{p_s}$ as long as this component is of size at least $M_s$, so that $\xi^\comp_{i,j}\in \C_i^\comp$; if this component is of smaller size, let $\xi^\comp_{i,j}$ be an independent sample of a uniform vertex in $\C_i^\comp$. Thus, for each $i\in S$, we obtain a coupling between $\ell$ i.i.d.\ uniform vertices $(\xi^{\cube}_{i,j})_{j=1}^\ell$ in $\C_i$ and $\ell$ i.i.d.\ uniform vertices (components) in $\C_i^\comp$. \cref{CompareCompHyp} implies that 
\begin{equation}\label{eq:discr_distance_matrices2}
\max_{i\in S}\max_{1\le j,k\le \ell}\left\{ \left |\frac{\X(p_s)}{V^{1/3}}d_i^\comp(\xi^\comp_j, \xi^\comp_k)-\frac{1}{V^{1/3}}d_i^{\cube}(\xi^\cube_j, \xi^\cube_k) \right |\right\} \to 0 \qquad \text{in probability}\,.
\end{equation}
We deduce the desired convergence in probability of $(M_i^m)_{i\in S}$ to $(M_i^\infty)_{i\in S}$ using \cref{equivGP}, and hence the convergence of $\MM^m$ to $\MM_\lambda$ in the product GP topology. 

We now prove the convergence of $\MM^m$ to $\MM_\lambda$ in the product GHP topology by relying on \cref{GP=>GHP}. Let $S\subset \N$ be finite. The first step consists in using \cref{thm:ghpTightness2} to prove that, for every $\delta>0$, 
\[\max_{i\in S} \max_{x\in \C_i} \frac{V^{2/3}}{|B(x,\delta V^{1/3})|}\]
is tight. For every $K>0$, we have 
\begin{align*}
	\proba\Big(\max_{i\in S} \max_{x\in \C_i} \frac{V^{2/3}}{|B(x,\delta V^{1/3})|}>K\Big) 
	&\le  \proba\Big(\max_{i\in S} \max_{x\in \C_i} \frac{V^{2/3}}{|B(x,\delta V^{1/3})|}>K, \partial B(x,\delta V^{1/3})\ne \emptyset \Big) \\
	& + \proba\Big(\max_{i\in S} \max_{x\in \C_i} \frac{V^{2/3}}{|B(x,\delta V^{1/3})|}>K, \partial B(x,\delta V^{1/3})=\emptyset \Big)\,.
\end{align*}
The first term in the right-hand side can be made as small as we want, uniformly in $m$, by choice of $K$ by \cref{thm:ghpTightness2}. On the other hand, the second term is bounded above by $\proba(\min_{i\in S} |\C_i|< V^{2/3}/K)$, which can also be made arbitrarily small, uniformly in $m$, also by choice of $K$ using the convergence of the component sizes in Theorem~\ref{main_thm}, and the well-known fact that the limit masses $(|\gamma_i|)_{i\ge 1}$ are all almost surely positive \cite{Aldous97}. The desired tightness follows, so that condition \emph{(ii)} of Lemma~\ref{GP=>GHP} is satisfied. Also we recall  and we deduce 
the claimed convergence of $\MM^m$ in the product GHP topology. 

Finally, we prove the convergence of $\MM^m$ to $\MM_\lambda$ in the $L^4$ topology. Using Skorohod representation theorem, consider a space where the convergence in the product GHP topology occurs almost surely. In that space, we now prove that $\distGHP(\MM^m; \MM_\lambda)$ tends to zero in probability. Observe first that (see e.g. \cite[Section 2.1]{Addario2017MST}) for any metric space $A$, we have $\dGHP(A,\emptyset)\le \diam(A)+|A|$, where $\emptyset$ denotes the trivial measured metric space (one point of mass zero) and $|A|$ is the total mass of $A$. It follows readily by the triangle inequality that, for any $i\ge 1$, $4^{-3} \dGHP(M_i^m,M_i)^4 \le \diam(M_i^m)^4+|M_i^m|^4+\diam(M_i^\infty)^4+|M_i^\infty|^4$. Let now $\epsilon,\eta>0$ be arbitrary. By \cref{GHP4tight}, there exists a $k'\in \N$ such that
\begin{align*}
\limsup_{m} \proba\Big(4^4 \sum_{i>k'} \diam(M_i^m)^4 > \epsilon\Big) < \eta/8 
\quad \text{and} \quad
\limsup_{m} \proba\Big(4^4 \sum_{i>k'} |M_i^m|^4 > \epsilon \Big) < \eta/8 \,.
\end{align*}
Let then $k\ge k'$ be large enough that we also have for the limit vector $\MM_\lambda$, 
\begin{align*}
	\proba\Big(4^4 \sum_{i>k} \diam(M_i^\infty)^4 > \epsilon\Big) < \eta/8 
\quad \text{and} \quad
\proba\Big(4^4 \sum_{i>k} |M_i^\infty|^4 > \epsilon \Big) < \eta/8 \,.
\end{align*}
It follows that, for this value of $k$, for all $m$ large enough, we have
\begin{align*}
	\proba(\distGHP(\MM^m; \MM_\lambda)> 2\epsilon) 
	\le \proba\bigg(\sum_{1\le i \le k} \dGHP(M_i^m, M_i^\infty)^4 > \epsilon\bigg) + \eta \,.
\end{align*}
The almost sure convergence of $\MM^m$ to $\MM_\lambda$ in the product topology implies that the first term in the right-hand side above tends to zero as $m\to\infty$, which completes the proof.

\bibliographystyle{unsrt}


\section*{Acknowledgments} The first and third authors are supported by ERC consolidator grant 101001124 (UniversalMap) as well as ISF grants 1294/19 and 898/23. We thank Eleanor Archer and Matan Shalev for useful discussions.

\bigskip

\noindent \textsc{Department of Mathematical Sciences, Tel Aviv University, Israel} 

\vspace{.1cm}

\noindent \textsc{LPSM, Sorbonne Universit\'e, and IUF, France} 

\vspace{.1cm}

\noindent  \textit{Emails:} \\
\texttt{ablancrenaudiepro@gmail.com} \\
\texttt{nicolas.broutin@sorbonne-universite.fr} \\
\texttt{asafnach@tauex.tau.ac.il} \par 
\addvspace{\medskipamount}



\end{document}